\documentclass[12pt]{amsart}
\usepackage{latexsym,amsmath, amscd, amssymb, amsthm,  euscript, mathrsfs, hyperref, stmaryrd}
\usepackage[dvipsnames]{xcolor}
\usepackage[all]{xy}
\usepackage{tikz}
\usetikzlibrary{arrows}
\usetikzlibrary{cd}
\usepackage{makecell}
\usepackage{enumitem}
\newtheorem{thm}[subsection]{Theorem}
\newtheorem{thm/def}[subsection]{Theorem/Definition}
\newtheorem{cor}[subsection]{Corollary}

\newtheorem{lem}[subsection]{Lemma}
\newtheorem*{lem*}{Lemma}

\newtheorem{prop}[subsection]{Proposition}
\newtheorem*{prop*}{Proposition}

\theoremstyle{definition}

\newtheorem{defn}[subsection]{Definition}

\theoremstyle{definition}

\theoremstyle{definition}

\newtheorem{rem}[subsection]{Remark}
\newtheorem*{rem*}{Remark}

\newtheorem{example}[subsection]{Example}
\numberwithin{equation}{subsection}

\newtheorem{pg}[subsection]{}

\newtheorem{warning}[subsection]{Warning}

\newcommand{\Q}{\mathbf{Q}}

\newcommand{\mc}{\mathcal }

\newcommand{\Z}{\mathbb{Z}}

\newcommand{\Sp}{\text{\rm Spec}}
\newcommand{\Spec}{\text{\rm Spec}}

\newcommand{\mls}{\mathscr}

\newcommand{\Pic}{\mathrm{Pic}}

\newcommand{\et }{\text{\rm \'et}}

\newcommand{\gp}{\text{\rm gp}}

\newcommand{\bm}{\mathbf{m}}
\newcommand{\bs}{\mathbf{s}}
\newcommand{\bD}{\mathbf{D}}
\newcommand{\bE}{\mathbf{E}}
\newcommand{\bd}{\mathbf{d}}
\newcommand{\bc}{\mathbf{c}}

\newcommand{\sN}{\mathscr{N}}

\newcommand{\fM}{\mathfrak{M}}
\newcommand{\fC}{\mathfrak{C}}

\newcommand{\cX}{\mc X}

\newcommand{\ba}{\mathbf{a}}

\newcommand{\bC}{\mathbf{C}}
\newcommand{\bb}{\mathbf{b}}
\newcommand{\zar}{\text{\rm zar}}
\usepackage{xcolor}

\newcommand\RACHELoff{\newcommand{\Commentr}[1]{}}
\newcommand{\Rachel}[1]{\Commentr{#1}}

\newcommand\MARTINoff{\newcommand{\Commentp}[1]{}}

\newcommand{\node}{\text{\rm node}}

\newcommand{\fm}{\mathfrak{m}}

\newcommand{\MS}[2]{M_{{#2}\rightarrow {#1}}}
\newcommand{\bMS}[2]{\overline M_{{#2}\rightarrow {#1}}}

\begin{document}

%%%%%%%%%%%%%%%%%%%%
% This code turns the TODO and NOTE comments on and off.

% % Comment/uncomment one of these lines at a time
% \NOTEoff
% %\NOTEoff

% Comment/uncomment one of these lines at a time
%\MARTINon
\MARTINoff

% Comment/uncomment one of these lines at a time
%\RACHELon
\RACHELoff
%%%%%%%%%%%%%%%%%%

\title{Curves with colliding points: logarithmic and stacky}
\author{Martin Olsson and Rachel Webb}

\begin{abstract}
    We introduce a new notion of generalized log twisted curves, which are marked nodal curves  with additional data at the marked points.  In the case when the markings are distinct this notion agrees with the notion of twisted curve introduced by Abramovich and Vistoli.  In addition to developing the basic notions and results, we study in this article the moduli of such curves as well as contraction maps between them.  This is motivated, in part, by applications to twisted stable maps which will be studied in  a subsequent article.
\end{abstract}

\maketitle
\setcounter{tocdepth}{1}
\tableofcontents

%\Rachel{pick up in section 3.4. then write the last section and its part of the intro.}
\section{Introduction}

% \Rachel{Higher order todos:
% \begin{itemize}
% \item \ref{R:flat-reduced} unimportant but interesting remark needs fleshing out
% \item \ref{S:contract-curve} I completely rewrote it. Martin please check.
% \item \ref{L:properties} I think is an additional needed lemma. --- Perhaps done, but Martin please check
% \item \ref{A:log-contractions} I rewrote it. I think all I really did was add details/notation, but historically my accuracy when I try to talk about log stuff is not high . . . so please check.
% \end{itemize}}

% \Rachel{organize by 2, most of 3. Then A, B, the rest of 3 (that deals with contractions), 5. 4 can still be an appendix.}

\subsection{Overview}
Moduli of weighted stable pointed curves were introduced by Hassett in \cite{Hassett} as compactifications of smooth pointed curves alternative to $\overline{\mc M}_{g, n}$. 
In these moduli spaces, objects are tuples $(C/S, \{s_i\}_{i=1}^n)$ where $C \to S$ is a nodal curve of genus $g$ and the $s_i$ are sections of the smooth locus, not necessarily disjoint, satisfying a certain stability condition.
In this article, we consider the problem of lifting such a tuple $(C/S, \{s_i\}_{i=1}^n)$ to a \textit{stacky} curve $\mls C \to S$ with coarse space $C$ and abelian stabilizer groups that are nontrivial only at nodes and markings. We build a theory that specializes to the twisted curves considered by Abramovich-Vistoli \cite{AV} and Abramovich-Olsson-Vistoli \cite{AOV} when the sections $s_i$ are disjoint. Our motivation is to construct proper moduli of weighted stable maps from these curves to tame Artin stacks, and we do this in the companion article \cite{OWarticle2}, generalizing work in \cite{AlexeevGuy, bayermanin} where stable map spaces from weighted stable pointed curves to projective varieties were studied.

In the present article, there are three main points.
\begin{itemize}
\item We define a \textit{generalized log twisted curve} over $S$ as a tuple $\bC$ that includes a nodal curve with not-necessarily-disjoint sections $(C/S, \{s_i\}_{i=1}^n)$ together with some log data. The tuple $\bC$ defines an associated stack $\mls C$ with coarse space $C$  whose nontrivial stabilizers are all abelian and supported at nodes and markings.
\item We analyze in detail contractions of genus-zero subcurves of generalized log twisted curves. This theory will be used in \cite{OWarticle2} to prove properness of the moduli spaces of weighted stable maps.
\item We characterize the tame Artin curves with abelian stabilizers that arise as stacks associated to generalized log twisted curves. 
\end{itemize}

% \subsection{A motivating example}
% We present an example that exposes the essential behavior of our theory. Let $k$ be an algebraically closed field of characteristic zero, let $S = \AA^1_k$, let $C  \to S$ be the trivial family with fiber $\PP^1_k$, and let $s_1, s_2$ be two sections that are disjoint except at the fiber over the origin in $S$. There are several generalized log twisted curves $\bC$ that extend this data. If we require that the fiber of the associated stack 
% $\mls C$ over a non-origin point of $S$ be isomorphic to $\PP^1_k$ rooted at two points, both to order 2, then there are exactly two possibilities. The stack associated to the first is just ...

\subsection{Generalized log twisted curves}
Given a pair $(C/S, \{s_i\}_{i=1}^n)$ consisting of a nodal curve and $n$ sections, consider the problem of finding a tame Artin stack $\mls C$ such that
\begin{equation}\label{eq:text}
\parbox{\dimexpr\linewidth-10em}{%
    %\begin{center}
    $\mls C$ has coarse space $C$ and nontrivial stabilizers that are all abelian and supported at markings and nodes of geometric fibers.%
    %\end{center}
  }
  \end{equation}
There is a naive way to construct such a $\mls C$: given a tuple of positive integers $r_1, \ldots, r_n$, let $\mls C_i$ denote the curve obtained by rooting $C$ along $s_i$ to order $r_i$, and then set $\mls C$ to be the fiber product of the $\mls C_i$ over $ C$. From here we see another way to construct curves $\mls C$ satisfying \eqref{eq:text}: given $\mls C$ constructed as a fiber product of root stacks as above, and given a morphism $f: \mls C \to \cX$ to an algebraic stack $\cX$, the relative coarse space of $f$ (defined as in \cite[\S 3]{AOV}) will be another stacky curve $\mls C'$ satisfying \eqref{eq:text}.

When $C$ is smooth, the stacky curves considered in this article are all relative coarse spaces of fiber products of root stacks, as sketched above. The idea of a generalized log twisted curve is to encode the choice of ``which'' relative coarse space to take in a certain sheaf of monoids on $C$, without reference to a target. When $C$ is nodal, stacky structure at the nodes can also be encoded by sheaves of monoids on $C$, as was done in \cite{AOV, LogTwisted}. Thus a generalized log twisted curve is defined in \ref{D:2.23} as a tuple
\[
\bC = (C, \{s_i\}_{i=1}^n, M_{S} \hookrightarrow M_{S}', \mls N) 
\]
where $M_{S} \hookrightarrow M_{S}'$ is a simple inclusion of log structures and $\mls N$ is an \textit{admissible sheaf of monoids} (admissible sheaves of monoids are defined in \ref{def:admissible-sheaf} below).  Here, $M_S$ is the canonical  log structure on $S$ arising from a log smooth morphism of log schemes $(C, M_C^{\node}) \to (S, M_S)$ (these log structures are independent of the sections $s_i$).

A generalized log twisted curve $\bC$ defines a stacky curve $\mls C$ satisfying \eqref{eq:text} as follows. By \cite{AOV, LogTwisted}, the simple inclusion $M_S \hookrightarrow M_S'$ defines a stack $\mls C^{\node}$ with coarse space $C$ and stacky structure supported at points that are nodes in their fibers. On the other hand, the admissible monoid $\mls N$ contains $\oplus_{i=1}^n s_{i, *} \mathbf{N}$ as a subsheaf, and $\oplus_{i=1}^n s_{i, *} \mathbf{N}$ is the characteristic sheaf of the log structure on $C$ corresponding to the marked points. By \cite{BV} the inclusion $\oplus_{i=1}^n s_{i, *} \mathbf{N} \hookrightarrow \mls N$ induces a ``stack of roots'' $\mls C^{\mls N}$ with coarse space $C$ and stack structure supported at the markings. We set
\[
\mls C := \mls C^{\node} \times_C \mls C^{\mls N}.
\]

 Let $\fM_{g, n}^{glt}$ be the fibered category whose fiber over a scheme $S$ is the groupoid of generalized log twisted curves over $S$. 
%\Rachel{notation: $\fM _{g, n}^{glt}$ is called $\fM _{g, n}^{glt}$ in section 2}
\begin{thm}
The fibered category $\fM _{g, n}^{glt}$ is a smooth algebraic stack locally of finite type over $\mathbf{Z}$.
%Moreover, it admits a Zariski covering by open substacks isomorphic to $\fM_{g,n}$, indexed by admissible monoids $N \subset \Q^n_{\geq 0}$.
\end{thm}
In fact, we show in \ref{thm:curves-are-algebraic} that $\fM_{g,n}^{glt}$ admits a Zariski open cover by substacks isomorphic to certain stacks of twisted curves in the sense of \cite{AOV}.  A key point, which is discussed further in \ref{SS:1.6} below, is that unlike the case when the sections $s_i$ are distinct the stack $\mls C$ does not determine $\bC $ in general.

\subsection{Contractions} 
We say a morphism $q: C \to D$ of nodal curves over $S$ is a \textit{contraction} if $q$ is surjective and the canonical map $\mls O_{D} \to Rq_*\mls O_{C}$ is an isomorphism. Equivalently, $q$ contracts a genus-zero subcurve of $C$ in every geometric fiber (see Corollary \ref{cor:fibers}). 

One place such morphisms arise is in the theory of weighted stable pointed curves in \cite{Hassett}, as we now explain. Recall that the definition of weighted stable curve depends on a choice of weight $\ba = (a_1, \ldots, a_n) \in [0, 1)^n$. If $\bb = (b_1, \ldots, b_n)$ is another choice of weight such that $b_i \leq a_i$ for each $i$, and if $(C/S, \{s_i\}_{i=1}^n)$ is $\ba$-weighted stable, then by \cite{Hassett} there is a canonical pair $(D/S, q: C \to D)$ where $D$ is a $\bb$-weighted stable curve over $S$ and $q$ is a contraction as defined above. 

A major aim of this paper is to extend the Hassett contractions described in the previous paragraph, or indeed any contraction of coarse curves, to generalized log twisted curves.
To this end we define a contraction $\bC \to \bD$ of generalized log twisted curves in \ref{D:6.2} to be a contraction $C \to D$ of the underlying coarse curves, together with additional data realizing compatibilities of the simple inclusions and admissible monoids. The definition is chosen so that a contraction $\bC \to \bD$ naturally induces a morphism $\mls C \to \mls D$ of associated stacks (see \ref{D:6.8}).  As an example, if $\mls C$ is the stack associated to a generalized log twisted curve $\bC$ and $f: \mls C \to \cX$ is a morphism to a tame stack $\cX$, the relative coarse space of $f$ will be the stack arising from a generalized log twisted curve $\bC'$ and $f$ will be the morphism associated to a contraction $\bC \to \bC'$ (this is \ref{lem:rel coarse}).

Our discussion of contractions culminates with the following result, which says that a contraction of coarse curves induces a canonical ``initial'' contraction of generalized log twisted curves. This theorem is stated more precisely later as \ref{T:8.1}. 

\begin{thm}\label{thm:contract-intro}
Let $\bC$ be a generalized log twisted curve with coarse space $C$ and let $q: C \to D$ be a contraction. Then there is a generalized log twisted curve $\bD$ with coarse curve $D$ and a contraction $\mathbf{q}: \bC \to \bD$ extending $q$ with the following universal property: If $\bC \to \bD'$ is a contraction such that the underlying contraction of coarse curves $C \to D'$ factors through $q$, then $\bC \to \bD'$ factors uniquely through $\mathbf{q}$.
\end{thm}

% As motivation for our study of contractions, recall the definition of weighted stable pointed curves given in [cite Hassett]. Given weights $\ba = (a_1, \ldots, a_n) \in [0, 1)^n$, an $\ba$-weighted stable pointed curve over a scheme $S$ is a tuple $(C, \{s_i\}_{i=1}^n)$ where $C \to S$ is a nodal curve of genus $g$ and $s_i$ are sections such that any $x \in C$, the sum of the weights of the $s_i$ equal to $x$ is at most 1, and such that the weighted log canonical bundle is ample. If $\bb = (b_1, \ldots, b_n)$ 

\subsection{Characterization of stacks associated to generalized log twisted curves}\label{SS:1.6}

Let $\fM_{g, n}$ denote the category of tame stacks whose coarse spaces are (relative) nodal curves of genus $g$, with nontrivial stabilizers supported at nodes and $n$ sections of the coarse curve. There is a natural functor $F: \fM^{glt}_{g, n} \to \fM_{g, n}$ given by taking the associated stack. When restricted to the subcategory where the markings are all distinct, this functor $F$ induces an equivalence between the category of log twisted curves and twisted curves \cite[A.5]{AOV}. However, as noted above $F$ is not an equivalence in general.

The functor $F$ will be faithful, but on objects where the markings coincide we give examples to show that it may not be full (see the discussion in \ref{ss:faith-not-full}). The stack associated to a generalized log twisted curve over an algebraically closed field is reduced and has abelian stabilizers, and one might at first expect that this characterizes the essential image of $F$. This is especially reasonable given that, due to work of Alqvist \cite{alqvist}, away from the nodes such a curve does arise from \textit{some} stack of roots construction, of which our $\mls C^{\mls N}$ is an example. However, we prove the following in \ref{ex:mu2}, \ref{ex:mu3} and \ref{ex:not}.

\begin{thm}
All tame abelian nodal orbicurves (see \ref{def:abelian curve}) with $\mu_2$ and $\mu_3$ stabilizers arise from generalized log twisted curves, but there is an example of an orbicurve with a $\mu_4$ stabilizer that does not arise from such.
\end{thm}

The mentioned example is, in fact, quite simple.
These results follow from our characterization \ref{C:4.13} of stacks that  arise from generalized log twisted curves. In deriving our characterization, we make Alqvist's result explicit for reduced abelian orbicurves with smooth coarse space, showing in \ref{p:tanc} that such a stack $\mls C$ can be constructed from a collection of inclusions of fine sharp monoids $\mathbf{N} \to M_i$ whose cokernels are the nontrivial stabilizer group schemes of $\mls C$.

%However this is not the case: consider the local curve $[\Spec(k[x,y]/(x^2-y^2) / \mu_4]$ where $\mu_4$ acts on $x$ with weight 1 and on $y$ with weight 3. We show in \ref{} that this local curve cannot arise from a generalized log twisted curve, although all tame abelian nodal orbicurves with $\mu_2$ and $\mu_3$ stabilizers do arise from such (\ref{} and \ref{}). In fact in \ref{} we 

\subsection{A word on the log perspective}

Associated to a marked curve $(C/S, \{s_i\}_{i=1}^n)$ there is a canonical log curve $(C, M_C) \to (S, \MS{S}{C})$ defined as follows. We have already mentioned the canonical log smooth curve $(C, M_C^{\node}) \to (S, \MS{S}{C})$ that is independent of the $s_i$. On the other hand, associated to each $s_i$ we have an inclusion of its (invertible) ideal sheaf $\mls I_i \to \mls O_C$, and this induces a log curve $(C, M_C^{s_i})$ over $S$ equipped with the trivial log structure. We define $M_C$ to be the coproduct (in the category of fine log structures) of $M_C^{\node}$ with all the $M^{s_i}_C$, and we note that the induced log map $(C, M_C) \to (S, \MS{S}{C})$ is no longer log smooth.

A generalized log twisted curve is by definition a marked nodal curve $(C/S, \{s_i\}_{i=1}^n)$ together with certain enhancements of the associated log structures $M_C$ and $\MS{S}{C}$, and the associated stack $\mls C$ can be viewed as a moduli space of certain log structures. In particular the stack $\mls C$ associated to a generalized log twisted curve has a canonical log structure $M_{\mls C}$ and there is a log morphism $(\mls C, M_{\mls C}) \to (C, M_C)$ extending the coarse moduli map (see \ref{SS:log-stack}). 

In the course of proving Theorem \ref{thm:contract-intro}, we show in Section \ref{A:log-contractions} that if $q: C \to D$ is a contraction of coarse curves over $S$, there is a canonical commuting diagram of log morphisms
\begin{equation}\label{eq:log-contr-intro}
\xymatrix{
(C, M_C)\ar[r]\ar[d]& (D, M_{D})\ar[d]\\
(S, \MS{S}{C})\ar[r]& (S, \MS{S}{D}).}
\end{equation}
extending $q$. If $\bC \to \bD$ is a contraction of generalized log twisted curves, the induced morphism $\mls C \to \mls D$ is uniquely determined by the fact that it extends to a morphism of log stacks filling in a certain commuting diagram \ref{R:3.24}. %(The diagram in question is a cube with \eqref{eq:log-contr-intro} and its analog for the stacky curves.)

\subsection{Organization of the paper}
In Section \ref{sec:glt-curves} we define generalized log twisted curves, including admissible sheaves, and prove that these form a smooth algebraic stack. Although an admissible sheaf is a priori a certain sheaf of monoids in the \'etale topology, we explain how it is equivalent to use certain Zariski sheaves of finitely generated abelian groups. Section \ref{S:section3} defines the stack $\mls C$ associated to a generalized log twisted curve $\bC$.

Sections \ref{A:appendixB}-\ref{sec:glt-contractions} discuss contractions. 
Section \ref{A:appendixB} discusses contractions of coarse curves and as such does not go too far beyond existing literature. The goal of Section \ref{A:log-contractions} is to construct the canonical log contraction diagram \ref{eq:log-contr-intro}, and Section \ref{sec:glt-contractions} discusses contractions of generalized log twisted curves.
Finally Section \ref{S:image} discusses the difference between the categories of generalized log twisted curves and tame abelian nodal orbicurves..

%Given weights $\ba = (a_1, \ldots, a_n) \in [0, 1)^n$, an  $\ba$-weighted stable pointed curve over a scheme $S$ is a tuple $(C, \{s_i\}_{i=1}^n)$ where $C \to S$ is a nodal curve of genus $g$ and $s_i$ are sections such that any $x \in C$, the sum of the weights of the $s_i$ equal to $x$ is at most 1, and such that the weighted log canonical bundle is ample. In this article, we consider the problem of enhancing such a tuple  

\subsection{Notation and conventions}\label{SS:conventions}
A \textit{prestable curve} over a scheme $S$ is an algebraic space $C$ over $S$ such that the structure morphism is flat, proper, locally finitely presented, of relative dimension 1, and whose geometric fibers are reduced connected nodal curves.

If $C \to S$ is a prestable curve, its \textit{smooth locus} is the open subspace $C^{sm} \subset C$ equal to the complement of the points that are nodes in their fibers. By \cite[\href{https://stacks.math.columbia.edu/tag/0C56}{Tag 0C56}]{stacks-project} this locus commutes with arbitrary base change.

An \textit{$n$-marked prestable curve} over a scheme $S$ is a pair $(C/S, \{s_i\}_{i=1}^n)$ where $C$ is a prestable curve over $S$ and $s_i: S \to C^{sm}$ are sections into the smooth locus of $C$. The sections $s_i$ do not need to be distinct. 

We often use the terminology \emph{coarse curve} to mean the underlying ordinary curve in a context when we also consider additional structure on that curve (either logarithmic data or stack structure).
%In this situation there are natural log structures on $C$ and $S$ - we refer to paragraph \ref{SS:2.19} for our notation in this regard.

Following the terminology of \cite{BV} we will distinguish between two kinds of charts for a log structure $M$ on a scheme $T$.  A \emph{Kato chart for $M$} is a map $\beta :P\rightarrow M$ from a fine monoid $P$ which induces an isomorphism $P^a\rightarrow M$ from the associated log structure. A \emph{Deligne-Faltings chart}, or just \emph{DF chart}, is a map $\beta :P\rightarrow \overline M$ which fppf locally on $T$ lifts to a Kato chart for $M$.

For an algebraic stack $\mc X$ a geometric point $\bar x\rightarrow \mc X$ is a morphism from the spectrum of an algebraically closed field.  The stabilizer group scheme $G_{\bar x}$ of $\bar x$ is the group scheme of automorphisms of $\bar x$.  It can be defined as the fiber product $\mc X\times _{\Delta, \mc X\times \mc X, (\bar x, \bar x)}\bar x.$  If $\mc X$ has finite diagonal (over some base) and therefore has a coarse moduli space $\pi :\mc X\rightarrow X$ then the stabilizer group scheme of a geometric point $\bar x\rightarrow X$ is defined to be the stabilizer group scheme of any lifting of $\bar x$ to $\mc X$.  Since any two such liftings are noncanonically isomorphic, this stabilizer group scheme is unique up to a conjugacy class of isomorphisms.  In particular, if the stabilizer group scheme is abelian then the stabilizer group scheme of a geometric point $\bar x\rightarrow X$ is well-defined.

Following \cite[10.3.2]{Olssonbook} for an algebraic stack $\mc X$ we write $\mls Div^+(\mc X)$ for the groupoid of generalized effective Cartier divisors on $\mc X$.  This is the groupoid of pairs $(\mc L, \alpha )$, where $\mc L$ is a line bundle on $\mc X$ and $\alpha :\mc L\rightarrow \mls O_{\mc X}$ is an $\mls O_{\mc X}$-linear map. When $\mc X$ is an algebraic space, we will also write $\mathfrak{D}iv^+_{\mc X}$ for the category fibered in groupoids over $\mc X_{\et}$ whose fiber over $V \to \mc X$ is $\mls{D}iv^+(V)$. 

\subsection{Acknowledgements}
Olsson was partially supported by NSF FRG grant DMS-2151946 and a grant from the Simons Foundation. Webb was partially supported by an NSF Postdoctoral Research Fellowship, award number 200213, and a grant from the Simons Foundation. Webb is grateful to Mark Gross for helpful conversations and to the Isaac Newton Institute for their hospitality.

\section{Generalized log twisted curves}\label{sec:glt-curves}
The main definition in the article is the definition \ref{D:2.23} of a generalized log twisted curve. As preparation, in Section \ref{ss:adm-monoid} we define admissible monoids of $\mathbf{Q}^n_{\geq 0}$(and an equivalent notion of an admissible subgroup of $\mathbf{Q}^n$) and prove basic properties of such objects. Section \ref{ss:adm-sheaf} extends this to a notion of an admissible \textit{sheaf} on a marked prestable curve $(C, \{s_i\}_{i=1}^n)$ and gives many equivalent characterisations of such objects. Finally in Section \ref{ss:glt} we define generalized log twisted curves and prove that they form an algebraic stack.

\subsection{Admissible monoids}\label{ss:adm-monoid}

We define admissible submonoids of $\mathbf{Q}^n_{\geq 0}$ and admissible subgroups of $\mathbf{Q}^n$ and show that these are equivalent notions. The reader more comfortable with groups may safely use our theory by noting only the group definition. 

\begin{defn}\label{def:admissible}   A submonoid $N\subset \mathbf{Q}^n_{\geq 0}$ is \emph{admissible} if $N$ is finitely generated and saturated and contains $\mathbf{N}^n$. 
\end{defn}

\begin{defn}   A subgroup $G\subset \mathbf{Q}^n$ is \emph{admissible} if it is finitely generated and contains $\mathbf{Z}^n$. 
\end{defn}

More generally if $I$ is a finite set, we will speak of admissible submonoids of $\mathbf{Q}^I_{\geq 0}$ (or admissible subgroups of $\mathbf{Q}^I$).

\begin{lem}\label{lem:monoid-vs-group}
There is a bijection between admissible submonoids of $\mathbf{Q}^n_{\geq 0}$ and admissible subgroups of $\mathbf{Q}^n$ given by sending $N$ to $N^{gp}$, with inverse sending $G$ to $G \cap \mathbf{Q}^n_{\geq 0}$. Moreover, for any collection of positive integers $\underline m = (m_1, \ldots, m_n)$, a subgroup of $\prod_{i=1}^n \mathbf{Z}/(m_i)$ determines an admissible subgroup of $\mathbf{Q}^n$, and every admissible subgroup arises in this way. 
\end{lem}
\begin{proof}
First note that if $G\subset \Q^n$ is an admissible subgroup then $(G \cap \Q^n_{\geq 0})^{gp}=G$. The forward containment is clear. For the reverse containment, let $\gamma \in G$ be an element and  choose  $\alpha \in \Z^n_{\geq 0}$ such that each coordinate of $\gamma + \alpha$ is nonnegative, i.e. $\gamma + \alpha \in G \cap \Q^n_{\geq 0}$. 
Then $\gamma = (\gamma +\alpha )-\alpha $ expresses $\gamma $ as a difference of two elements in $G\cap \Q^n_{\geq 0}$ (note that $\mathbf{Z}_{\geq 0}^n\subset G\cap \Q^n_{\geq 0}$),
so $\gamma\in (G \cap \Q^n_{\geq 0})^{gp}$.

Furthermore, since $G$ is finitely generated it is contained in $\frac{1}{m}\mathbf{Z}^n$ for some integer $m>0$ and therefore $G\cap \Q^n_{\geq 0} = G\cap \frac{1}{m}\mathbf{N}^n$ is an exact submonoid of $\frac{1}{m}\mathbf{N}^n$.  From this and \cite[2.1.9 (2)]{Ogus}
it follows that $G\cap \Q^n_{\geq 0}$ is fine and saturated.

Next we show that if $N\subset \Q^n_{\geq 0}$ is an admissible submonoid then $N^{gp} \cap \Q^n_{\geq 0}=N$. The reverse containment is clear. For the forward containment suppose $\alpha - \beta \in N^{gp} \cap \Q^n_{\geq 0}$, where $\alpha$ and $\beta$ are in $N$ and $\alpha-\beta$ has nonnegative coordinates. Then there is an integer $m>0$ such that $m(\alpha-\beta) \in \Z^n_{\geq 0} \subset N$. Since $N$ is saturated this implies $\alpha-\beta \in N$ implying that $N^\gp \cap \Q^n_{\geq 0}\subset N$.

To prove the second statement of the lemma, note that admissible subgroups of $\mathbf{Q}^n$ are in bijection with subgroups of 
$$
(\Q/\mathbf{Z})^n=\text{colim}_{\underline m}\prod _i\mathbf{Z}/(m_i).
$$
A subgroup of $\prod _i\mathbf{Z}/(m_i) \simeq \prod_i (1/m_i)\mathbf{Z}/\mathbf{Z}$ determines a subgroup of $(\Q/\mathbf{Z})^n$ and the preimage in $\Q^n$ is an admissible subgroup containing $\prod_{i=1}^n (1/m_i)\mathbf{Z}$. Since every finitely generated subgroup of $\Q^n$ contains $\prod_{i=1}^n (1/m_i)\mathbf{Z}$ for some sequence of integers $m_i$, every admissible subgroup arises this way.
\end{proof}

\begin{rem}\label{L:integral}
If $N \subset \mathbf{Q}^n_{\geq 0}$ is admissible, then the inclusion $\mathbf{N}^n \to N$ is integral. This is immediate from the definition \cite[Def~4.6.2(1)]{Ogus}: if $\lfloor q \rfloor$ denotes the componentwise floor and $\langle q \rangle $ denotes the componentwise fractional part of a vector of nonnegative rational numbers $q$, then an equality $a_1+q_1=a_2+q_2$ with $a_i \in \mathbf{N}^n$ and $q_i \in \mathbf{Q}_{\geq 0}^n$ implies equalities $a_1+\lfloor q_1 \rfloor = a_2 + \lfloor q_2 \rfloor$ and $\langle q_1 \rangle = \langle q_2 \rangle.$ Since $N$ is saturated we have $\langle q_i \rangle \in N$. 
%In fact, it follows from \cite[Thm~4.5.7]{Ogus} that the inclusion is flat.
\end{rem}

\begin{rem}\label{R:rank1} An admissible submonoid of $\mathbf{Q}_{\geq 0}$ is of the form $\frac{1}{m}\mathbf{N}$ for some $m\geq 1$.  Indeed this follows from \ref{lem:monoid-vs-group} and the fact that any subgroup of $\mathbf{Z}/(N)$ is of the form $\mathbf{Z}/(m)$ for some divisor $m|N$.
\end{rem}

\subsubsection{Quotients of admissible monoids}
 Limits and colimits, and in particular quotients, exist quite generally in the category  of integral monoids  \cite[Chapter I, \S 1]{Ogus}, but it will be useful to have concrete descriptions of certain quotients of monoids which we discuss in this subsection.

\begin{defn}
Let $I$ be a subset of $\{1, \ldots, n\}$.
\begin{itemize}
\item If $N \subset \mathbf{Q}^n_{\geq 0}$ is an admissible submonoid, the \textit{quotient} $N^I \subset \mathbf{Q}^I_{\geq 0}$ is the image of $N$ under the projection $\mathbf{Q}^n_{\geq 0} \to \mathbf{Q}^I_{\geq 0}$.
\item If $G \subset \mathbf{Q}^n$ is an admissible subgroup, the \textit{quotient} $G^I \subset \mathbf{Q}^I$ is the image of $G$ under the projection $\mathbf{Q}^n \to \mathbf{Q}^I$.
\end{itemize}
\end{defn}

\begin{lem}\label{L:2.7}
If $N \subset \mathbf{Q}^n_{\geq 0}$ is an admissible submonoid, then $N^I$ is also admissible. Likewise, if $G \subset \Q^n$ is admissible, so is $G^I$. Moreover, if $G = N^{gp}$, then 
\begin{equation}\label{eq:quotient-monoid-vs-group}
(N^I)^{gp} = G^I.
\end{equation}
\end{lem}
\begin{proof}
The group $G^I$ is admissible, since the image of a finitely generated group is finitely generated and the image of $\mathbf{Z}^n$ in $\Q^I$ is $\mathbf{Z}^I$. Once we show \eqref{eq:quotient-monoid-vs-group}, it follows from \ref{lem:monoid-vs-group} that $N^I$ is also admissible. But to show \eqref{eq:quotient-monoid-vs-group} it is enough to show
\[
N^I = G^I \cap \Q^I_{\geq 0}.
\]
The forward containment is clear from the definitions. Conversely, if $\gamma \in G^I \cap \Q^I_{\geq 0}$, then there is a lift $\widetilde{\gamma} \in G$ whose coordinates indexed by $I$ are nonnegative: Start with any lift $\tilde \gamma \in G$ of $\gamma $ and add a suitable element of $\mathbf{Z}^{I^c}_{\geq 0}\subset G$. Such a lift $\tilde \gamma $ lies in  $N = G \cap \Q^n_{\geq 0}$ implying that  $\gamma$ is in $N^I$.
\end{proof}

\begin{rem}
Note that $G^I$ is a quotient of $G$ in the category of abelian groups. Similarly, as a corollary of \ref{L:2.7}, we see that $N^I$ is a quotient of $N$ in the category of fine saturated monoids: in fact, 
\begin{equation}\label{eq:face}
N^I = N/F
\end{equation}
where $F$ is the face $\Q^{I^c}_{\geq 0} \cap N$ (see \cite[1.4.1]{Ogus} for the definition of a face). To deduce \eqref{eq:face} from \ref{L:2.7}, recall that
\[
N/F := \langle N, F^{gp} \rangle / F^{gp}
\]
where $\langle N, F^{gp} \rangle$ denotes the submonid of $N^{gp}$ generated by $N$ and $F^{gp}$, and note that if $G = N^{gp}$ then
\[
G^I = N^{gp} / (N^{gp} \cap \Q^{I^c}).
\]
It follows  that $(N/F)^{gp}$ is precisely $G^I$, which is the groupification of $N^I$ by \ref{L:2.7}. Since $N/F$ and $N^I$ have the same groupification, they are equal by \ref{lem:monoid-vs-group}.
\end{rem}

\begin{warning}\label{warn:pushout vs image}
If $G \subset \Q^n$ is an admissible subgroup and $I$ is a subset of $\{1, \ldots, n\}$ we can also consider  the pushout $\mathbf{Z}^I \oplus_{\mathbf{Z}^n} G$ in the category of abelian groups, which by the universal property of pushout comes equipped with a map $\mathbf{Z}^I\oplus _{\mathbf{Z}^n}G\rightarrow G^I$. However, the pushout $\mathbf{Z}^I\oplus _{\mathbf{Z}^n}G$ may have torsion and therefore is not admissible and the map to $G^I$ may not be an isomorphism.  In fact, this map is the quotient of $\mathbf{Z}^I\oplus _{\mathbf{Z}^n}G$ by its torsion subgroup, which implies that $G^I$ is the pushout of the diagram
$$
\xymatrix{
\mathbf{Z}^n\ar[r]\ar[d]& G\\
\mathbf{Z}^I&}
$$
in the category of torsion free abelian groups.

For example, if $G \subset \Q^2$ is the subgroup $\langle (\frac{1}{2}, 0), (0,\frac{1}{2})\rangle$ and $I$ has one element, then $\mathbf{Z}^I \oplus_{\mathbf{Z}^2} G \simeq \mathbf{Z} \times \mathbf{Z}/(2)$ and $G^I  = \langle \frac{1}{2} \rangle \subset \Q$.

Translated to the language of monoids this means the following. If $N \subset \mathbf{Q}^n_{\geq 0}$ is admissible, the pushout  in the category of integral monoids $\mathbf{N}^I \oplus_{\mathbf{N}^n} N$ is equal to the image of $\mathbf{N}^I \oplus N$ in $\mathbf{Z}^I \oplus_{\mathbf{Z}^n} G$. Hence the natural map
\[
\mathbf{N}^I \oplus_{\mathbf{N}^n} N \to N^I
\]
realizes $N^I$ as the pushout in the category of saturated sharp monoids.
% If $N \subset \mathbf{Q}^n_{\geq 0}$ is an admissible submonoid and $I$ is a subset of $\{1, \ldots, n\}$, the quotient $N^I$ is \textbf{not} equal to the pushout $\mathbf{N}^I \oplus_{\mathbf{N}^n} N$ in the category of integral monoids. Rather, the canonical map
% \[
% \mathbf{N}^I \oplus_{\mathbf{N}^n} N \to N^I
% \]
% is the quotient by the torsion subgroup of $\mathbf{N}^I \oplus_{\mathbf{N}^n} N$, which can be nontrivial in general. For example, \Rachel{add an example and check this. I'm talking about pushout in what category?}
\end{warning}

\begin{rem}\label{rem:define m}
If $N \subset \mathbf{Q}^n_{\geq 0}$ is an admissible submonoid then for each index $i$ there is a natural map $N\rightarrow N^{\{i\}}$.  Since a saturated finitely generated sharp  submonoid of $\mathbf{Q}$ containing $\mathbf{N}$ is equal to $\frac{1}{m}\mathbf{N}$ for some $m\geq 1$ each $N^{\{i\}}$ is a free monoid of rank $1$ and the induced map 
\[
N \subset \oplus_{i=1}^n N^{\{i\}}.
\]
gives a canonical free admissible submonoid of $\mathbf{Q}^n_{\geq 0}$ containing $N$. Similarly, if $G \subset \mathbf{Q}^n$ is admissible then $\oplus_{i=1}^n G^{\{i\}}$ is an admissible free abelian subgroup of $\mathbf{Q}^n$ containing $G$. 
\end{rem}

\subsection{Admissible sheaves}\label{ss:adm-sheaf}
%Let $S$ be a scheme and $C\rightarrow S$ a prestable curve with sections $\{s_i:S\rightarrow C\}_{i=1, \dots, n}$ into the smooth locus of $C$ (not necessarily distinct). 
Let $(C/S, \{s_i\}_{i=1}^n)$ be an $n$-marked prestable curve over a scheme $S$ (defined in \ref{SS:conventions}). We are about to define a notion of an admissible sheaf on $(C/S, \{s_i\}_{i=1}^n)$, and for this we must choose a topology on $C$.
Recall that $C$ may be an algebraic space, but nevertheless it makes sense to speak of sheaves on $C$ in both the Zariski and \'etale topologies (see \cite[\href{https://stacks.math.columbia.edu/tag/03YD}{Tag 03YD}]{stacks-project}). 
We will ordinarily work with the \'etale topology, but on occasion it will be useful to work with the Zariski topology (in particular for the proof of \ref{thm:curves-are-algebraic} below), and as we will see the notion of an admissible sheaf is independent of which of these topologies we use.

Ordinarily we will suppress the topology from the notation, with the \'etale topology being understood, but if there is the possibility for confusion we will incorporate the topology.  For example, we can consider the constant sheaves $\mathbf{Q}^n_{C_\et }$ and $\mathbf{Q}^n_{C_\zar }$ on $C$ in either topology.  Furthermore, the canonical map
\begin{equation}\label{eq:beta}
\beta: \Q^n_{\geq 0} \to \oplus s_{i, *} \mathbf{Q}_{\geq 0}.
\end{equation}
is defined in either topology, and if we wish to emphasize the topology we write $\beta _{C_\et }$ or $\beta _{C_\zar }$.  In fact, if $q:C_\et \rightarrow C_\zar $ is the canonical morphism of topoi then $q^{-1}\Q^n_{\geq 0, C_\zar }\simeq \Q^n_{\geq 0, C_\et }$ and $q^{-1}\beta _{C_\zar }=\beta _{C_\et }$.

% If $\mls G$ is a Zariski sheaf on $C$, then its stalk at a point $x \in |C|$ is the colimit
% \[
% \mls G_x := \lim_{\longrightarrow}\; \mls G(U)
% \]
% over open subsets $U$ of $|C|$ containing $x$.

It turns out that our desired notion of an admissible sheaf on $(C/S, \{s_i\}_{i=1}^n)$ can be realized by a whole spectrum of (equivalent) definitions, including in the Zariski vs \'etale topology. We now give names to just two of these definitions, at opposite ends of the spectrum:

\begin{defn}\label{def:admissible-sheaf}
An \'etale subsheaf $\mls N \subset \oplus s_{i, *} \mathbf{Q}_{\geq 0, S_\et }$ is \textit{admissible} if \'etale locally it is the image under the map $\beta _{C_\et }$ of the constant sheaf $N_{C_\et }$ associated to an admissible submonoid $N \subset \mathbf{Q}^n_{\geq 0}$. 
\end{defn}

\begin{defn}
A Zariski subsheaf $\mls G \subset \oplus s_{i, *} \mathbf{Q}_{S_\zar }$ is \textit{admissible} if Zariski locally it is the image under $\beta ^\gp _{C_\zar }$ of the constant sheaf $G_{C_\zar }$ associated to an admissible subgroup $G \subset \Q^n$.
\end{defn}

\begin{example}
If $N \subset \mathbf{Q}^n_{\geq 0}$ (resp. $G \subset \mathbf{Q}^n$) is admissible, the image of the associated constant sheaf under $\beta$ is an admissible sheaf of monoids (resp. groups) on $C$. We denote this image by $\mls N_N$ (resp. $\mls G_G$).
\end{example}
\begin{example}\label{E:rank1} If all the sections $s_i$ are disjoint, then it follows from \ref{R:rank1} that an admissible subsheaf of monoids contained in $\oplus s_{i*}\mathbf{Q}_{\geq 0}$ is of the form $\oplus s_{i*}\frac{1}{m_i}\mathbf{N}$ for various integers $m_i$.
\end{example}

\begin{lem}\label{lem:subsheaves}
Let $(C/S, \{s_i\}_{i=1}^n)$ be a $n$-marked prestable curve over a scheme $S$.
There are natural bijections between the following sets of sheaves.
\begin{enumerate}
\item [(i)] Admissible subsheaves of monoids $\mls N \subset \oplus s_{i, *} \mathbf{Q}_{\geq 0, S_\et }$
\item [(ii)] \'Etale subsheaves $\mls N \subset \oplus s_{i, *} \mathbf{Q}_{\geq 0, S_\et }$ that are Zariski locally the image of an admissible submonoid $N \subset \Q^n_{\geq 0}$ under the map $\beta _{C_\et }$
\item [(iii)] Zariski subsheaves $\mls N \subset \oplus s_{i, *} \mathbf{Q}_{\geq 0, S_\zar }$ that are Zariski locally the image of an admissible submonoid $N \subset \Q^n_{\geq 0}$ under the map $\beta _{C_\zar }$
\item [(iv)] Zariski subsheaves $\mls N \subset \oplus s_{i, *} \mathbf{Q}_{\geq 0, S_\zar }$ 
such that for every $x \in |C|$,  the stalk $\mls N_{x}$ (in the Zariski topology) is an 
admissible submonoid of $(\oplus s_{i, *} \mathbf{Q}_{\geq 0, S_\zar })_{x} 
\simeq \Q^I_{\geq 0}$ (where $I$ is the set of sections passing through 
$x$), and there is a Zariski neighborhood of $x$ where $\mls N$ 
is the image of the monoid $\mls N_{x}$ under $\beta: \Q^I_{\geq 0} 
\to \oplus_{x \in s_i} s_{i, *} \Q_{\geq 0, S_\zar }.$
\item [(v)] Admissible subsheaves of groups $\mls G \subset \oplus s_{i, *} \mathbf{Q}_{S_\zar }.$
\end{enumerate}
\end{lem}

\begin{rem} The bijections between the various sets in question are discussed in the proof.
\end{rem}

\begin{rem} We note that by recombining the different conditions in \ref{lem:subsheaves}, one can write down many equivalent definitions of an admissible sheaf on $(C/S, \{s_i\}_{i=1}^n)$. See \ref{lem:admissible-base} and \ref{lem:admissible-log}  for more equivalent characterizations.
\end{rem}

\begin{proof}[Proof of Lemma \ref{lem:subsheaves}]

The bijection between the sets in (i) and (ii) is the identity (as a map of sets); in other words, it is clear that a sheaf $\mls N$ as in (ii) is admissible, and we claim that an admissible sheaf $\mls N$ is always Zariski locally the image of an admissible submonoid under $\beta.$ Indeed, if $\mls N$ is admissible and $U \to C$ is an \'etale morphism such that $\mls N|_U$ is equal to $\mls N_N|_U$, then $p(U)$ is a nonempty Zariski open set of $C$ such that for every geometric point $\bar x \in p(U)$, the \'etale stalks of $\mls N$ and $\mls N_N$ at $\bar x$ are equal.   It follows that the subsheaves $\mls N$ and $\mls N_N$ of $\oplus s_{i, *} \mathbf{Q}_{\geq 0, S_\et}$ are equal.

 Let $\epsilon :C_\et \rightarrow C_\zar $ denote the natural morphism of topoi, and recall that $\epsilon_*$ is left exact. The bijection between the sets in (ii) and (iii) is induced by $\epsilon^{-1}$ and $\epsilon_*$. Indeed, these functors identify the sheaves labelled $\oplus s_{i, *} \mathbf{Q}_{\geq 0}$ in each topology, and they preserve subsheaves since both are left exact. These functors also preserve the property of being Zariski-locally isomorphic to some $\mls N_N$. If $\mls N$ is any Zariski sheaf then the adjunction map $\mls N \to \epsilon_*\epsilon^{-1}\mls N$ is an isomorphism by inspection, and if $\mls N$ is an \'etale subsheaf of $\oplus s_{i, *} \mathbf{Q}_{\geq 0, S_\et }$ then $\epsilon^{-1}\epsilon_*\mls N \to \mls N$ is an isomorphism by the folllowing general result \ref{lem:zar-et}.

The bijection between the sets in (iii) and (iv) is again the identity. If $\mls N$ is as in (iv), then in a Zariski neighborhood of $ x \in |C|$ we see that $\mls N$ is the image of the admissible monoid
\[
\mls N_{ x} \oplus \mathbf{N}^{n-|I|} \subset \Q^I_{\geq 0} \oplus \Q^{n-|I|}_{\geq 0} = \Q^n_{\geq 0}.
\]
Conversely, if $\mls N$ is as in (iii), let $ x \in |C|$ be a point where in a Zariski neighborhood we have $\mls N = \mls N_N$ for some admissible monoid $N$. Since taking stalks commutes with the image, we have
\begin{equation}\label{eq:stalk quotient}
\mls N_{x} = \mathrm{image}((\mls N_N)_{ x} \to (\oplus s_{i, *} \Q_{\geq 0, S_\zar })_{x}) = \mathrm{image}(N \to \Q^I_{\geq 0})
\end{equation}
where $I$ is the set of sections passing through $ x$. This is precisely the quotient monoid $N^I$, which is an admissible submonoid of $\Q^I_{\geq 0}$ by \ref{L:2.7}. This computation shows $N^I = \mls N_{ x}$. However, by shrinking the Zariski neighborhood we may assume that $\beta: \Q^n_{\geq 0, C_\zar } \to \oplus s_{i, *} \Q_{\geq 0, S_\zar }$ factors through $\Q^I_{\geq 0, C_\zar }$, and hence $N \to \mls N$ factors through $N^I$. In this neighborhood, $\mls N$ is the image of $N^I = \mls N_{x}$ as claimed.

The bijection between the sets in (iii) and (v) sends an admissible sheaf of subgroups $\mls G$ to $\mls G \cap \oplus s_{i, *} \mathbf{Q}_{\geq 0, S_\zar }$ (and in the opposite direction it sends $\mls N$ to $\mls N^{gp}$). This function is injective by \ref{lem:monoid-vs-group}, and if $G \subset \mathbf{Q}^n$ is an admissible subgroup surjecting onto $\mls G$ then $G \cap \Q^n_{\geq 0}$ is an admissible submonoid by \ref{lem:monoid-vs-group}, and it surjects onto $\mls G \cap \oplus s_{i, *} \Q_{\geq 0, S_\zar }$ by universal properties of fiber products. Hence it is enough to show that $\mls N = \mls N^{gp} \cap \oplus s_{i, *} \Q_{\geq 0, S_\zar }$. This can be checked on stalks, where we see that
\[\mls N_{ x} = \mls N^{gp}_{x} \cap (\oplus s_{i, *} \Q_{\geq 0, S_\zar })_{ x}\]
by \ref{lem:monoid-vs-group}, since $\mls N_{x}$ is an admissible submonoid of $(\oplus s_{i, *} \Q_{\geq 0, S_\zar })_{ x}$ by the equivalence of the sets in (iii) and (iv).
\end{proof}

\begin{rem}\label{rem:stalk quotient}
It follows from \ref{lem:subsheaves} that if $\mls N$ is an admissible sheaf on $(C/S, \{s_i\}_{i=1}^n)$, and if $\bar x: \Sp(k) \to C$ is a geometric point with image $x \in |C|$, then the \'etale stalk $\mls N_{\bar x}$ is equal to the colimit $\displaystyle \lim_{\longrightarrow} \mls N(U)$ over Zariski neighborhoods $U$ of $x$. In particular, if $\mls N$ is an admissible sheaf, then for any $x \in |C|$ we have a well-defined monoid
\[
\mls N_x := \mls N_{\bar x}
\]
for any geometric point $\bar x$ with image $x$. Moreover if $U \subset |C|$ is a Zariski neighborhood of $x$ and $N \subset \Q^n_{\geq 0}$ is an admissible monoid such that $\mls N|_U = \mls N_N$, then by \eqref{eq:stalk quotient} we have
\[
\mls N_{\bar x} = \mls N_x = N^I
\]
where $I \subset \{1, \ldots, n\}$ is the set of indices $i$ such that $s_i$ contains $x$.
\end{rem}
% We define admissible sheaves of both monoids and groups on the marked curve $(C, \{s_i\})$ in both the \'etale and Zariski topology, and show that all four notions are equivalent. 

% Recall that $C$ may be an algebraic space, so to begin we use the \'etale topology on $C$. Let $\beta$ denote the canonical morphism of \'etale sheaves of monoids
% \[
% \beta: \Q^n_{\geq 0} \to \oplus s_{i, *} \mathbf{Q}_{\geq 0}
% \]
% where $\Q^n$ is the constant sheaf.

% \begin{defn}
% An \'etale subsheaf $\mls N \subset \oplus s_{i, *} \mathbf{Q}_{\geq 0}$ (resp. $\mls G \subset \oplus s_{i, *} \mathbf{Q}$) is \textit{admissible} if \'etale locally it is the image of an admissible submonoid $N \subset \Q^n_{\geq 0}$ (resp. $G \subset \Q^n$) under the map $\beta $ (resp. $\beta^{gp}$).
% \end{defn}

\begin{lem}\label{lem:zar-et} Let $X$ be an algebraic space  and let $G$ be a sheaf of sets on the Zariski site of $X$. 

(i) The adjunction map $G\rightarrow \epsilon _*\epsilon ^{-1}G$ is an isomorphism.

(ii) Let $i:F\hookrightarrow \epsilon ^{-1}G$ be an inclusion of sheaves on $X_\et $.  Then the adjunction map $\epsilon ^{-1}\epsilon _*F\rightarrow F$ is an isomorphism.
\end{lem}
\begin{proof} 
For (i), note that if $v:V\rightarrow X$ is an \'etale morphism then $v(V)\subset X$ is open and therefore $\epsilon ^{-1}G$ is the sheaf associated to the presheaf $\epsilon ^{ps, -1}G$ sending $v:V\rightarrow X$ to $G(v(V))$. This implies, in particular, that the presheaf $\epsilon ^{ps, -1}G$ is separated. 
 
  If $U\subset X$ is an open subset and $s\in H^0(U_\et , \epsilon ^{-1}G)$ is a section then there exists an \'etale cover $\{v_i:V_i\rightarrow U\}_{i\in I}$ such that $s_i = s|_{V_i}$ is the image of a section of $t_i\in G(v_i(V_i))$.  Since the image of $V_i\times _UV_j\rightarrow U\subset X$ equals $v(V_i)\cap v(V_j)$ we find that the restrictions of $t_i$ and $t_j$ to this intersection map to the same element of $\epsilon ^{-1}G(V_i\times _UV_j)$, and since $\epsilon ^{ps, -1}G$ is separated also the same element in $\epsilon ^{ps, -1}G(V_i\times _UV_j) = G(v_i(V_i)\cap v_j(V_j))$.  It follows that the sections $t_i$ and $t_j$ agree on overlaps defining a section $t\in G(U)$ mapping to $s$.

For (ii) note that the functor $\epsilon _*$ is left exact so the map $\epsilon _*F\rightarrow \epsilon _*\epsilon ^{-1}G = G$ (using (i)) is an inclusion.  To verify that the adjunction map is an isomorphism it suffices to show that for every geometric point $\bar x\rightarrow X$ with image $x\in |X|$ the map (Zariski stalk on the left, \'etale stalk on the right)
$$
(\epsilon _*F)_x\rightarrow F_{\bar x}
$$
is an isomorphism.  The map is clearly injective since both sides compatibly include into $G_x = (\epsilon ^{-1}G)_{\bar x}$.  For the surjectivity, suppose $s\in F_{\bar x}$ is a section.  Its image in $(\epsilon ^{-1}G)_{\bar x}$ then defines an element $t_x\in G_{x}$.  Let $U\subset X$ be a Zariski neighborhood over which $t_x$ extends to a section $t\in G(U)$. Since the image of $t$ in $(\epsilon ^{-1}G)_{\bar x}$ lies in $F_{\bar x}$ there exists an \'etale neighborhood $V\rightarrow U$ of $\bar x$ such that the pullback $t_V$ of $t$ to $\epsilon ^{-1}G(V)$ lies in $F(V)$.  Since \'etale morphisms are open we can arrange, after shrinking on $U$, that the map $V\rightarrow U$ is surjective. Now observe that the two pullbacks 
$$
p_1^*t_V, p_2^*t_V\in F(V\times _UV)
$$
are equal since they become equal in $\epsilon ^{-1}G(V\times _UV)$ and $F\rightarrow \epsilon ^{-1}G$ is injective.  It follows that $t_V$ descends to a section of $F(U)$ which implies that the original $t\in G(U)$ in fact lies in $F(U)=\epsilon _*F(U)$.  The corresponding element of $(\epsilon _*F)_x$ then maps to $s$ in $F_{\bar x}$.
\end{proof}

We close this section by showing (in \ref{lem:admissible-base}) that for a sheaf on $(C/S, \{s_i\}_{i=1}^n)$, the property of being admissible is in fact Zariski local on the base. For $x \in |C|$, define a Zariski open subset $U_{ x} \subset C$ via
\[
U_{ x} := C \setminus \left( \bigcup_{J \subset \{1, \ldots, n\}} V_{ {x}, J}\right)
\]
where $V_{x, J}$ is defined to be the union of the irreducible components of $\cap_{j \in J} s_j$ that do not contain $x$.  Observe that if $s_i$ does not contain $ x$, then $s_i$ does not meet $U_{ x}$, so in particular the restriction of $\oplus s_{i, *} \Q$ to $U_{ x}$ is equal to $\oplus_{ x \in s_i} s_{i, *} \Q$.

\begin{lem}\label{lem:admissible-base}
If $\mls G \subset \oplus_{s_{i, *}} \Q$ is admissible, then for every $ x \in |C|$ the restriction of $\mls G$ to $U_{ x}$
is the image of $\mls G_{ x}$ under $\beta^{gp}: \Q^I \to \oplus_{ x \in s_i} s_{i, *}\Q$, where $I$ is the set of sections containing $ x$.
Consequently, a Zariski subsheaf $\mls G \subset \oplus_{s_{i, *}} \Q$ is admissible if and only if Zariski locally on the base $S$ it is the image of an admissible subgroup $G \subset \Q^n$.
\end{lem}
\begin{proof}
Let $\mls G'$ be the admissible sheaf of groups on $U_{ x}$ associated to $\mls G_{ x}$. For $ y \in |C|$, we let $I_{ y}$ denote the set of sections passing through $ y$.

It suffices to show that $\mls G$ and $\mls G'$ have the same stalks. Let $ y \in |U_{ x}|$. Let $V$ be the irreducible component of $\cap_{s_j \in I_{ y}} s_j$  that contains $ y$. The definition of $U_{ x}$ tells us that $ x \in |V|$ and also that $I_{ y} \subset I_{ x}$. By \ref{lem:subsheaves} there is a Zariski neighborhood of $ y$ (resp. $x$) where the stalk is a quotient of $\mls G_{ y}$ (resp. $\mls G_{x}$). Since $V$ is irreducible, we can find a $ z$ in $V$ that is in both of these neighborhoods; that is,
\[
\mls G_{ z} = (\mls G_{ y})^{I_{ z}} \quad \quad \quad \quad \mls G_{ z} = (\mls G_{ x})^{I_{ z}}.
\]
For generic such $ z$ we have $I_{ z} = I_{ y}$, and hence by the first equation we see $\mls G_{ z} = \mls G_{ y}$, and by the second we see $\mls G_{ y} = (\mls G_{ x})^{I_{ y}}$, which is precisely the stalk $\mls G'_{ y}$.

For the ``consequently'' statement of the lemma, let $\mls G$ be admissible on $(C/S, \{s_i\}_{i=1}^n)$ and let $\pi: C \to S$ be the structure morphism of the curve. For $ s \in S$ let $ x_1, \ldots,  x_r$  be the distinct marked points of $C_{ s}$. Define
\[
G := \bigoplus_{j=1}^r \mls G_{ x_j} \subset \Q^n \quad \quad \quad \quad \text{and} \quad \quad \quad \quad U = \bigcap_{j=1}^r \pi(U_{ x_j}).
\]
We claim that $\mls G$ is the image of $G$ on $\pi^{-1}(U)$. To prove it, it is enough to show that $\mls G$ and $\mls G_G$ have the same stalks at every point $ y \in |U|$. But we must have $ y \in |U_{ x_j}|$ for some $x_j$, and in particular $I_{ y} \subset I_{ x_j}$. Our previous discussion shows that $\mls G_{y} = (\mls G_{ x_j})^{I_{ y}}$, but this monoid is equal to $G^{I_{ y}} = (\mls G_G)_{ y}$ since $I_{ y} \subset I_{ x_j}$.

\end{proof}

\subsection{Log interpretation}
\label{SS:2.19}
In this section we define a canonical log curve $(C, M_C) \to (S, \MS{S}{C})$ associated to a marked nodal curve $(C, \{s_i\}_{i=1}^n)$. The discussion will lead to another way to think about admissible sheaves $\mls N$ on $(C, \{s_i\}_{i=1}^n)$.

%We now explain how the data of an admissible sheaf is related to the log curve defined by $(C/S, \{s_i\}_{i=1}^n).$

The log curve $(C, M_C) \to (S, \MS{S}{C})$ is defined as follows. By \cite[1.2]{OlssonTohoku} there is a canonical structure of a log smooth morphism
\[
(C, M^{\node}_C) \to (S, \MS{S}{C})
\]
where $f^*\MS{S}{C} \to M^{\node}_C$ is an isomorphism away from the nodes. The log structure $\MS{S}{C}$ depends on the morphism $C\rightarrow S$, but if the curve $C$ is clear from context we will write $M_S$ for $\MS{S}{C}$. (Warning: Sometimes in the literature one finds the notation $M_{C/S}$ in reference to the relative characteristic sheaf on $C$, which should not be confused with $\MS{S}{C}$).

% \Rachel{changed $\mls Div$ to $\mathfrak{D}iv$; is this correct? see \ref{SS:conventions}}
Moreover, since each section $s_i$ lands in the smooth locus of $C$, its ideal sheaf $\mls I_i$ is locally free and hence the dual of the inclusion $\mls I_i \to \mls O$ defines a Deligne-Faltings structure \begin{equation}\label{eq:df}
\mathbf{N} \to \mathfrak{D}iv^+_{C}
\end{equation}
in the sense of \cite[Def 3.1]{BV} (see \ref{SS:conventions} for the definition of $\mathfrak{D}iv^+_C$). In fact \eqref{eq:df} factors through a Deligne-Faltings structure 
\begin{equation}\label{eq:df2}
s_{i, *} \mathbf{N} \to \mathfrak{D}iv^+_{C},
\end{equation} and the quotient $\mathbf{N} \to s_{i, *} \mathbf{N}$ is a Deligne-Faltings chart for \eqref{eq:df2}. We let $M^{s_i}_C$ denote the log structure associated to \eqref{eq:df2} (see \cite[Thm~3.6]{BV}); in particular, the characteristic sheaf of $M^{s_i}_C$ is $s_{i, *}\mathbf{N}$.  Concretely if $U\rightarrow C$ is an \'etale morphism from a connected scheme $U$ with the preimage of $s_i(S)\subset C$ in $U$ nonempty and connected, then 
$$
M^{s_i}_C(U) = \coprod _{n\geq 0}\{\text{trivializations of $\mls I_i^{\otimes n}$}\}.
$$

We define
\[
M_C := M_C^{\node} \oplus_{\mls O^*_C} M^{\underline{s}}_C \quad \quad \quad \quad \text{where} \quad M^{\underline{s}}_C := \oplus_{\mls O^*_C} M^{s_i}_C.
\]

% The relationship to admissible sheaves is explained in the following lemma.

\begin{lem}\label{lem:admissible-log}
The characteristic sheaf $\overline{M}^{\underline s}_C$  is isomorphic to $\oplus s_{i, *}\mathbf{N}$, and the map $\beta$ defined in \eqref{eq:beta} is induced by a canonical DF chart
\[
\mathbf{N}^n \to \overline{M}^{\underline s}_C.
\]
In particular, an admissible sheaf $\mls N$ may be viewed as a subsheaf of $\overline{M}_{C, \Q}.$
\end{lem}
\begin{proof}
    This follows from the preceding discussion.
\end{proof}

\subsection{Generalized log twisted curves}\label{ss:glt}
In this section we introduce the category of generalized log twisted curves and show that they form an algebraic stack.

%It will be a fibered category over the category of schemes. \Rachel{including morphisms, not just isomorphisms, is new}

\begin{defn}\label{D:2.23} A \emph{generalized log twisted curve of genus $g$ with $n$ marked points} $\mathbf{C}$ over a scheme $S$ is a collection of data 
\begin{equation}\label{E:2.23.1}
\mathbf{C} = (C/S, \{s_i\}_{i=1}^n, \ell :M_S\hookrightarrow M_S', \sN )
\end{equation}
as follows:
\begin{enumerate}
    \item [(i)] $(C/S, \{s_i\}_{i=1}^n)$ is an $n$-marked prestable curve such that $C$ has genus $g$.
    \item [(ii)] $\ell $ is a simple morphism of log structures in the sense of \cite[1.5]{LogTwisted}.
    \item [(iii)] $\mls N \subset \oplus s_{i, *} \Q_{\geq 0}$ is an admissible subsheaf.
\end{enumerate}
Given two generalized log twisted curves over $S$ 
\begin{equation}\label{eq:some curves}
\mathbf{C}^{(j)} = (C^{(j)}/S, \{s_i^{(j)}\}_{i=1}^n, \ell :M^{(j)}_S\hookrightarrow M_S^{(j)\prime }, \sN ^{(j)} ), \ \ j=1,2
\end{equation}
a \textit{morphism} $\mathbf{C}^{(1)}\rightarrow \mathbf{C}^{(2)}$ lying over $S$ is a pair $(f,  \rho )$, where $f:C^{(1)}\rightarrow C^{(2)}$ is an isomorphism of curves over $S$ and $\rho :M_S^{(1)\prime }\rightarrow M_S^{(2)\prime }$ is an isomorphism of log structures on $S$, such that the following hold:
\begin{enumerate}
    \item [(a)] The morphism $f$ preserves the sections: $f \circ s^{(1)}_i = s^{(2)}_i$.
    \item [(b)] The isomorphism $\rho $ is compatible with the isomorphism $M_S^{(1)}\simeq M_S^{(2)}$ induced by $f$.
    \item [(c)] Under the isomorphism $f_*(\oplus s^{(1)}_{i, *} \Q_{\geq 0}) = \oplus s^{(2)}_{i, *} \Q_{\geq 0}$ we have $\mls N^{(2)} \subset f_*\mls N ^{(1)}$. 
\end{enumerate}
\end{defn}
\begin{rem} In the usual manner we can also consider morphisms of generalized log twisted curves over a morphism of schemes.
If $\mathbf{C}^{(1)}$ (resp. $\mathbf{C}^{(2)}$) is a generalized log twisted curve over $S^{(1)}$ (resp. $S^{(2)}$) and $g: S^{(1)} \to S^{(2)}$ is a morphism of schemes, a \emph{pullback} $g^* \mathbf{C}^{(2)}$ is a generalized log twisted curve given by pulling back each datum of $\mathbf{C}^{(2)}$. A \emph{morphism} $\mathbf{C}^{(1)} \to \mathbf{C}^{(2)}$ lying over $g$ is a morphism $\mathbf{C}^{(1)} \to g^*\mathbf{C}^{(2)}$ of objects over $S^{(1)}$.
\end{rem}

\begin{rem} By the bijection in \ref{lem:subsheaves} the data of $\mls N$ can be replaced by an admissible subsheaf of groups $\mls G\subset \oplus s_{i, *} \Q$. 
%Alternatively, following \ref{lem:admissible-log} one can view $\mls N$ as a subsheaf of $\overline{M}_{C, \Q}$ satisfying certain conditions.
\end{rem}

\begin{pg}
Consider the fibered category $\fM_{g,n}$ of log twisted curves of genus $g$ with $n$ untwisted markings. That is, the fiber of $\fM_{g,n}$ over a scheme $S$ is the groupoid of tuples
\[
\fM_{g,n}(S) := \{(C/S, \{s_i\}_{i=1}^n, \ell: M_S \hookrightarrow M'_S)\}
\]
where
\begin{enumerate}
\item[(i)] $(C/S, \{s_i\}_{i=1}^n)$ is an $n$-marked prestable curve of genus $g$
\item[(ii)] $\ell$ is a simple morphism of log structures.
\end{enumerate}
It follows from \cite[1.8]{LogTwisted} that the groupoid $\fM_{g, n}(S)$ is isomorphic to the groupoid of genus-$g$ $n$-marked twisted curves with stacky structure supported at the nodes.
\end{pg}
\begin{lem}\label{lem:twisted-nodes}
The category $\fM_{g,n}$ is a smooth algebraic stack, locally of finite type and with quasi-compact and separated diagonal over $\mathbf{Z}$.
\end{lem}
\begin{proof}
The category $\fM_{g, 0}$ is a smooth algebraic stack locally of finite type over $\mathbf{Z}$ by \cite[Thm~A.6]{AOV} and \cite[\href{https://stacks.math.columbia.edu/tag/0DSS}{Tag 0DSS}]{stacks-project}. Its diagonal is quasi-compact and separated by \cite[Rmk~A.7]{AOV} and \cite[\href{https://stacks.math.columbia.edu/tag/0DSQ}{Tag 0DSQ}]{stacks-project}.  Let $\fC_{g,0} \to \fM_{g,0}$ be its universal curve, and let $\fC_{g,0}^{sm} \subset \fC_{g,0}$ be the smooth locus. Then $\fM_{g, n}$ is the $n$-fold fiber product of $\fC_{g,0}^{sm}$ over over $\fM_{g,0}$. Since $\fC^{sm}_{g, 0} \to \fM_{g, 0}$ is smooth and separated the lemma follows. 
%Observe that $\fM_{g, n}$ is an algebraic stack locally of finite type over $\mathbf{Z}$, and since $\fC_g^{sm} \to \fM$ is smooth the stack $\fM_{g, n}$ is also smooth.
\end{proof}

Let $\fM _{g,n}^{glt}$ denote the fibered category over $Sch/\mathbf{Z}$ of generalized log twisted curves of genus $g$ and with $n$ marked points.

\begin{thm}\label{thm:curves-are-algebraic}
The fibered category $\fM _{g, n}^{glt}$ is a smooth algebraic stack, locally of finite type and with quasi-compact and separated diagonal over $\mathbf{Z}$. Moreover, it admits a Zariski covering by open substacks isomorphic to $\fM_{g,n}$, indexed by admissible monoids $N \subset \Q^n_{\geq 0}$.
\end{thm}
\begin{proof}
%  We have a functor of categories fibered in groupoids
% \begin{equation}\label{eq:forget}
% \fM_{g, n}^{new} \to \fM_{g, n}
% \end{equation}
% that forgets the admissible subsheaf $\sN$.  
We use the forgetful functor
\begin{equation}\label{eq:forget}
\fM^{glt}_{g,n} \to \fM_{g,n}
\end{equation}
To check that $\fM_{g,n}^{glt}$ is a stack for the \'etale topology we use \cite[\href{https://stacks.math.columbia.edu/tag/0CKJ}{Tag 0CKJ}]{stacks-project}, which means we must check two hypotheses. For the first, note that \eqref{eq:forget} is faithful. Indeed, the data of an arrow in $\fM_{g, n}^{glt}$ is the same as the data of an arrow in $\fM_{g, n}$: arrows in $\fM_{g, n}^{glt}$ satisfy the additional \textit{requirement} that they preserve the sheaf $\sN$. For the second, we must show that if $(C/S, \{s_i\}_{i=1}^n, M_S \hookrightarrow M_S')$ is an object of $\fM_{g,n}$ over $S$, the presheaf
\[
(T \to S) \mapsto \{ \text{admissible subsheaves}\; \sN \text{on $T \times_S C$}\}
\]
defines a sheaf on the \'etale site of $S$. One of the sheaf axioms holds because agreement of \'etale subsheaves of $\oplus s_{i,*}\Q^n_{\geq 0}$ can be checked at \'etale stalks. The other sheaf axiom holds because \'etale sheaves satisfy descent for the \'etale topology and admissibility is an \'etale local condition.

Next we check that the diagonal of $\fM _{g, n}^{glt}$ is representable.  We have seen that \eqref{eq:forget} is faithful, so it is enough to show that the map on $Isom$ sheaves induced by \eqref{eq:forget} is open. This amounts to checking that, if $(C/S, \{s_i\}_{i=1}^n)$ is a prestable curve, the locus where two admissible sheaves $\sN_1, \sN_2$ on $C$ agree is open. This follows from \ref{lem:subsheaves}, specifically the equivalence of (i) and (iv).

To show that $\fM^{glt}_{g,n}$ is representable it remains to find a smooth cover by a scheme. For any admissible monoid $N \subset \mathbf{N}^n$ let $\fM^{glt}_{g,n}(N) \subset \fM^{glt}_{g,n}$ be the subcategory where $\mls N$ is equal to the image of $N$ under the map $\beta$ defined in \eqref{eq:beta}. This is an open substack of $\fM^{glt}_{g,n}$ by \ref{lem:admissible-base}, and the union of these open substacks covers $\fM^{glt}_{g,n}$ by the same. On the other hand, the restriction of \eqref{eq:forget} to $\fM^{glt}_{g,n}(N)$ is an equivalence of categories. This shows that $\fM^{glt}_{g,n}$ has a Zariski covering by open substacks isomorphic to $\fM_{g, n}$.
The remaining desired properties of $\fM^{glt}_{g,n}$ now follow from \ref{lem:twisted-nodes}.

\end{proof}
\begin{rem} 
For a sequence of integers $\bm = (m_1, \ldots, m_n)$, let $\fM^{glt}_{g, n, \bm}$ denote the subcategory of $\fM^{glt}_{g,n}$ consisting of objects $\bC$ whose admissible monoid $\mls N$ has the property that for every $x \in C$, the stalk $\mls N_x$ is contained in $\oplus s_{i, *}\frac{1}{m_i} \mathbf N$. The proof of \ref{thm:curves-are-algebraic} shows that $\fM^{glt}_{g, n, \bm}$ is an open substack of $\fM^{glt}_{g, n}$, and that restriction of \eqref{eq:forget} to this open substack is \'etale and quasi-compact.

% In fact the proof gives more information about the map
% \eqref{eq:forget}.  Namely it shows that $\fM _{g, n}^{new}$ admits a Zariski covering by open subsets which map to $\fM _{g, n}$ by open immersions.  Furthermore, if we impose a bound on the allowed sequences of integers $\bm $ for which $\mls N$ lies in $\frac{1}{\bm}\overline M_{C}$ then we see that the restriction of \eqref{eq:forget} to the open substack classifying $\mathbf{C}$ subject to this bound is \'etale and quasi-compact.
\end{rem}

\section{The stack associated to a generalized log twisted curve}\label{S:section3}

%While there is no ``equivalence of categories'' between generalized log twisted curves and their associated stacks (see for instance Example \ref{ex:two-to-one}), in the following appendix we characterize what stacks arise from generalized log twisted curves (this is not used in the rest of the article).

\subsection{The stack associated to an admissible monoid}
Let $N \subset \Q^n_{\geq 0}$ be an admissible monoid. Then $N^{gp}$ is a subgroup of $\Q^n$, hence torsion free. The affine scheme $\Spec(\mathbf{Z}[N])$ is an affine toric variety and the diagonalizable group scheme
\[
D(N^{gp}) := \Spec(\mathbf{Z}[N^{gp}])
\]
is a torus with a natural action on $\Spec(\mathbf{Z}[N])$. We write
$$
\mls S_N:= [\Sp (\mathbf{Z}[N])/D(N^\gp )]
$$
for the associated stack quotient (sometimes referred to as a ``toric stack'').  By \cite[Prop~3.25]{BV} there is an equivalence of fibered categories between $\mls S_N$ and the fibered category that to each scheme $T$ associates the groupoid of symmetric monoidal functors $N \to \mls Div^+(T)$, where $\mls Div^+(T)$ is as in \ref{SS:conventions}.
%the groupoid of generalized effective Cartier divisors on $T$. 

\begin{example}\label{ex:toric stack}
Let $(C/S, \{s_i\}_{i=1}^n)$ be a marked prestable curve. The markings define a canonical morphism $C \to \mls S_{\mathbf{N}^n}$. This morphism corresponds to the functor $\mathbf{N}^n \to \mls Div^+(C)$ sending the $i^{th}$ generator to $\mls I_i\rightarrow \mls O_C$, where $\mls I_i \subset \mls O_C$ is the ideal sheaf of the $i^{th}$ section.
\end{example}

\begin{rem} Given a morphism $N \to N'$ of fine monoids, the morphism 
$$
\Sp (\mathbf{Z}[N'])\rightarrow \Sp (\mathbf{Z}[N])
$$
is equivariant with respect to the homomorphism $D({N'}^\gp )\rightarrow D(N^\gp )$, and therefore induces a morphism $\mls S_{N'} \to \mls S_N.$
\end{rem}

\subsection{The stack associated to a generalized log twisted curve}\label{SS:3.4}
There is a stack $\mls C$ associated to a generalized log twisted curve  $\bC=(C/S, \{s_i\},  M_S \hookrightarrow M_S', \mls N)$ as follows.

The simple inclusion $M_S \hookrightarrow M_S'$ defines a tame stack $\mls C^{\node}$ with coarse space $C$ which introduces stacky structure at the nodes of $C$ (see \cite[1.8]{LogTwisted}).

The admissible sheaf $\mls N$ defines a tame stack $\mls C^{\mls N}$ with coarse space $C$ which is an isomorphism away from the images of the $s_i$ in $C$, as we now explain. The inclusion $\oplus s_{i, *} \mathbf{N} \hookrightarrow \mls N$ is a system of denominators in the sense of \cite[4.3]{BV}, and $\mls C^{\mls N}$ is the associated ``stack of roots'' arising from \cite[4.19]{BV}.  
This stack is the fibered category over $Sch/C$ whose fiber over $f: U \to C$ is the groupoid of pairs $(F, \alpha)$ where $F:f^{-1}\mls N\rightarrow \mathfrak{D}iv^+_U$ is a symmetric monoidal functor and $\alpha :G\to F\circ \iota $ is an isomorphism of monoidal functors defining a $2$-commutative diagram 
\begin{equation}\label{eq:def CN}
\begin{tikzcd}
f^{-1}(\oplus s_{i, *}\mathbf{N}) \arrow[r, hookrightarrow, "\iota"] \arrow[d, "G"'] & f^{-1}\mls N \arrow[dl, dashrightarrow, "F"]  \\
\mathfrak{D}iv^+_{U}
\end{tikzcd}
\end{equation}
where $G$ is the Deligne-Faltings structure \eqref{eq:df2} induced by the marked points. 

\subsection{The log stack associated to a generalized log twisted curve}\label{SS:log-stack}
% \Rachel{I extracted this material from the previous subsection because I always had a hard time finding the defintion of the notation for log structures.}
Let $\bC = (C, \{s_i\}_{i=1}^n, \ell:M_S \hookrightarrow M_S', \mls N)$ be a generalized log twisted curve. We now explain a canonical structure of a log stack $(\mls C, M_{\mls C})$ such that the coarse moduli map extends to a log morphism $(\mls C, M_{\mls C}) \to (C, M_C)$, where the log structure $M_C$ was defined in \ref{SS:2.19}.

As explained in \cite[4.1]{LogTwisted}, the stack $\mls C^{\node}$ classifies certain simple inclusions $M_C^{\node} \to M$. In particular there is a universal log structure $M_{\mls C^{\node}}$ on $\mls C^{\node}$ and inclusion $M^{\node}_C|_{\mls C^{\node}} \to M_{\mls C^{\node}}$. On the other hand, the stack $\mls C^{\mls N}$ is the stack classifying inclusions of fine log structures $M_C^{\underline s}\hookrightarrow M'$ for which the induced map $\overline M_{C, \mathbf{Q}}^{\underline s}\rightarrow \overline M'_{\mathbf{Q}}$ is an isomorphism and identifies $\overline M'$ with $\mls N$. Hence we have a universal log structure $M_{\mls C}^{\mls N}$ and inclusion $M^{\underline{s}}_C|_{\mls C^{\mls N}} \hookrightarrow M_{\mls C}^{\mls N}.$ We define
\[
M_{\mls C} := M_{\mls C}^{\node}\oplus _{\mls O_{\mls C}^*}M_{\mls C}^{\mls N}
\]
where the log structures on the right hand side are pulled back from $\mls C^{\node}$ and $\mls C^{\mls N}$, respectively.
The log morphism $(\mls C, M_{\mls C}) \to (C, M_C)$ is induced by the inclusions $M^{\node}_C|_{\mls C^{\node}} \to M_{\mls C^{\node}}$ and $M^{\underline{s}}_C|_{\mls C^{\mls N}} \hookrightarrow M_{\mls C}^{\mls N}.$

% Furthermore, the description of $\mls C^{\node}$, $\mls C^{\mls N}$, and $\mls C$ as moduli stacks of certain log structures defines log structures $M_{\mls C}^{\node}$, $M_{\mls C}^{\mls N}$,  and $M_{\mls C} = M_{\mls C}^{\node}\oplus _{\mls O_{\mls C}^*}M_{\mls C}^{\mls N}$ on $\mls C$ (note that we are defining these as log structures on $\mls C$ but the first two are pulled back from $\mls C^{\node}$ and $\mls C^{\mls N}$ respectively) and a morphism of log stacks $(\mls C, M_{\mls C})\rightarrow (C, M_C)$ sending $M_C^{\node}$ (resp. $M^{\underline s}$) to $M_{\mls C}^{\node}$ (resp. $M_{\mls C}^{\mls N}$).

\subsection{Local descriptions of $\mls C^{\mls N}$}\label{SS:local}
Let $\bC = (C/S, \{s_i\}_{i=1}^n, \ell:M_S \hookrightarrow M_S', \mls N)$ be a generalized log twisted curve over $S$.
We give several local descriptions of $\mls C^{\mls N}$. Let $U \to  C$ be a morphism such that $\mls N|_U$ is the image of an admissible monoid $N \subset \Q^n_{\geq 0}$ under $\beta$ \eqref{eq:beta}. (By \ref{lem:admissible-base}, one can take $U$ to be the preimage of a Zariski open subscheme of the base $S$.) %\Rachel{put in here what \ref{lem:admissible-base} says about local descriptions }

First, it follows from \cite[Prop~4.18]{BV} that the fiber of $\mls C^{\mls N}$ over $V \to U$ is the groupoid of pairs $(F, \alpha)$ where $F$ is a symmetric monoidal functor and $\alpha:G\rightarrow F\circ \iota $ is an isomorphism of functors as in the diagram
\begin{equation}\label{eq:local description3}
\begin{tikzcd}
\mathbf{N}^n \arrow[r, hookrightarrow, "\iota"] \arrow[d, "G"'] &N \arrow[dl, dashrightarrow, "F"]  \\
\mls Div^+(V).
\end{tikzcd}
\end{equation}
 Indeed, the diagram
\begin{equation}\label{eq:chart}
\begin{tikzcd}
\mathbf{N}^n \arrow[r] \arrow[d] & N \arrow[d, "\beta"] \\
\oplus s_{i,*}\mathbf{N}|_V \arrow[r] & \mls N|_V
\end{tikzcd}
\end{equation}
is a chart for the system of denominators $\oplus s_{i,*}\mathbf{N} |_V\hookrightarrow \mls N|_V$.

Second, it follows from the above and \cite[Prop~4.13]{BV} that there is a fiber square
\begin{equation}\label{eq:local description1}
\begin{tikzcd}
\mls C_U^{\mls N} \arrow[d] \arrow[r] & \mls S_N \arrow[d] \\
U \arrow[r] & \mls S_{\mathbf{N}^n}
\end{tikzcd}
\end{equation}
where $U \to \mls S_{\mathbf{N}^n}$ is the restriction of the morphism in \ref{ex:toric stack} and $\mls C^{\mls N}_U := U \times_C \mls C^{\mls N}$.

Third, if we moreover we assume $U = \Spec(R)$ and that each $\mls I_i|_U$ is trivial, then the maps  $\mls I_i|_U\rightarrow \mls O _U$ are identified with elements $a_1, \ldots, a_n$ of $R$, and we may explicitly compute the fibered product $\mls C_U^{\mls N}$:
\begin{equation}\label{eq:local description2}
\mls C_U^{\mls N} = [\Spec(R \otimes_{\mathbf{Z}[\mathbf{N}^n]} \mathbf{Z}[N]) / D(N^{gp}/\mathbf{Z}^n)].
\end{equation}
The homomorphism $\mathbf{Z}[\mathbf{N}^n] \to R$ is induced by the elements $a_i$ and the group $D(N^{gp}/\mathbf{Z}^n)$ acts on $\mathbf{Z}[N]$ via the natural action and trivially on $R$.

%This shows, in particular, that at a point where all the sections $s_i$ coincide the stabilizer group scheme is equal to $D(N^\gp /\mathbf{Z}^n)$. \Rachel{I think this isn't completely obvious. We must be using the fact that $N$ is sharp; for example, for the inclusion $\mathbf{N} \to \mathbf{N} \times \mathbf{Z}/r\mathbf{Z}$ I think the stabilizer group scheme is trivial, not $\mu_r$.}

Finally, if some of the sections $s_i$ do not meet $U$ we can simplify the above descriptions. Let $I \subset \{1, \ldots, n\}$ be the set of indices $i$ such that $s_i$ meets $U$. Then in \eqref{eq:chart} we may replace $\mathbf{N}^n$ with $\mathbf N^I$ and $N$ with $N^I$, and the resulting diagram is still a chart. It follows that we may replace $\mathbf{N}^n$ with $\mathbf N^I$ and $N$ with $N^I$ in \eqref{eq:local description3}, \eqref{eq:local description1}, and \eqref{eq:local description2}.

\begin{example}\label{E:example3.7} If the sections $s_i$ are disjoint then by \ref{E:rank1} an admissible subsheaf amounts to the data of an integer $m_i$ attached to each marking and the associated stack $\mls C^{\mls N}$ is obtain by applying the $m_i$-th root stack construction at each of the sections $s_i$.
\end{example}

\begin{cor}\label{cor:properties}
Let $\bC = (C/S, \{s_i\}_{i=1}^n, M_S \hookrightarrow M_S', \mls N)$ be a generalized log twisted curve over $C$ and let $\mls C$ be the associated stack. 
\begin{enumerate}
\item $\mls C$ is a tame Artin stack over $S$ with coarse moduli space $C$.
\item The map $\mls C \to C$ is flat over $C^{sm}$ and an isomorphism on $C^{sm}$ minus the images of the $s_i$.
\item The maximal open substack $\mls C^\circ \subset \mls C$ over which the stabilizer groups are trivial is schematically dense in $\mls C$.
\item The stabilizer group scheme of a geometric point $\bar x \to \mls C$ mapping to the smooth locus is $D((N^I)^{gp})/\mathbf{Z}^I$ (as a group scheme over the residue field of $\bar x$).
\item Geometric fibers of $\mls C \to S$ are reduced.
\end{enumerate}
\end{cor}
\begin{proof}
Part (1) follows from the preceeding discussion. Flatness follows from the local description \ref{eq:local description2} combined with \ref{L:integral}. 
Statement (3) follows from (2) and the structure of $\mls C$ over nodes.
To prove (4), it suffices to consider the case when $I = \{1, \ldots, n\}$ and so we omit $I$ from the notation. If $\bar x : \mathrm{Spec}(k) \to U$ is a geometric point where all the sections $s_i$ coincide, then from \eqref{eq:local description2} the fiber product $\mls C^{\mls N} \times_U \mathrm{Spec}(k)$ is given by 
\[
[\Spec(k[N]/\langle x_1, \ldots, x_n \rangle) / D(N^{gp}/\mathbf{Z}^n)]
\]
where $x_i\in \mls O_{U, \bar x}$ is a generator for the ideal defining $s_i$. For every element $\alpha$ of $N$, some multiple is in $\mathbf{N}^n$, and if $\alpha$ is nontrivial then this multiple is also nontrivial (i.e. $N$ is \textit{sharp}). It follows that every nonconstant element of $k[N]/\langle x_1, \ldots, x_n \rangle$ is nilpotent and in particular this ring is local, so the stabilizer of the unique closed point of its spectrum must be the entire group $D(N^{gp}/\mathbf{Z}^n)$.
Finally (5) follows from (3) and the observation that over a field the substack $\mls C^\circ \subset \mls C$ is reduced.
\end{proof}

% \begin{rem}\label{R:flatremark}
% The local description \ref{eq:local description2} combined with \ref{L:integral} implies that the morphism $\mls C^{\mls N}\rightarrow C$ is flat.
% \end{rem}

% \Rachel{Maybe should summarize properties of $\mls C$ here: tame artin, flat over coarse space away from nodes, generically a scheme, reduced geometric fibers---I guess the last can be checked at the local rings of $\mls C$. At the nodes it holds by the explicit description of the ring at a node. At non-special points it holds because $\mls C = C$ here and $C$ is reduced. At a marked point it follows from flatness: If $R = \mls O_{C, c}$ is the local ring at the marked point then $\mls C_R = [\Spec(A)/G]$ where $R \to A$ is flat. Since $R$ is a domain (in fact equal to $k[t]_{(t)}$), we have an inclusion $R \hookrightarrow R_t$ where $t$ is a uniformizer of the maximal ideal of $R$. Since $A$ is flat when we tensor with $A$ we get an inclusion $A \hookrightarrow A_t$. But $\Spec(A_t)$ is a $G$-torsor over the domain $R_t$, hence reduced---REALLY NOT SURE ABOUT THIS WHEN $G$ IS NOT REDUCED---CAN I EVEN CHECK REDUCEDNESS ON A FLAT COVER?? so $A$ is a subring of a domain hence reduced.

% Alternatively, for flat morphisms associated points map to associated points, and there is Ravi's remark 24.4.11---but this picture seems to contradict my picture of flat limits with nonreduced fibers.}

\section{Contractions of coarse curves}\label{A:appendixB}

In this section we study contraction morphisms of prestable nodal curves over a base scheme. For scheme curves, such morphisms were introduced by Knudsen in \cite[1.3]{knudsen} to study projectivity of the moduli spaces $\overline{\mls M}_{g, n}$. While most of the results of this section presumably are well-known to experts, we include them here for lack of a suitable reference.  The essential ideas in this section can be found in \cite{Hassett} and  \cite{knudsen}. 

For purposes of the current paper, the main result of this section is \ref{C:sequences} which says that all contractions are locally determined by a sequence of contractions of rational bridges and rational tails. To prove it, we first show the result \ref{T:contractthm} that weighted prestable maps have unique stabilizations. This latter result will also be used in \cite{OWarticle2}.

%\Rachel{do we need to disallow some degenerate values of $g, \ba$ for which there are no stable maps?}

\subsection{Contraction morphisms}

Fix a base scheme $S$. In this section, we work with the category of prestable curves over $S$ whose morphisms between objects $C \to S$ and $C' \to S$ are $S$-morphisms $C \to C'$ (not necessarily isomorphisms). 

\begin{defn}\label{def:contraction}
A \textit{contraction} of a prestable curve $C/S$ is a pair $(C'/S, q)$ where $C'/S$ is a prestable curve over $S$ and $q: C \to C'$ is a morphism of prestable curves such that $q$ is surjective and the canonical map $\mls O_{C'} \to Rq_*\mls O_{C}$ is an isomorphism. We also call $q$ a \textit{contraction morphism} or simply a contraction.
\end{defn}

We observe that our definition differs from the one in \cite[1.3]{knudsen}. We will eventually see in \ref{cor:fibers} that they are equivalent. We now note several basic properties of this definition.

\begin{lem}\label{L:properties}
Let $q: (C_1/S, f_1) \to (C_2/S, f_2)$ and $p: (C_2/S, f_2) \to (C_3/S, f_3)$ be two morphisms of prestable maps over a scheme $S$. Let $S' \to S$ be any morphism of schemes.
\begin{enumerate}
\item[(i)] If $q$ and $p$ are contractions, then so is $p \circ q$.
\item[(ii)] If $q$ and $p \circ q$ are contractions, so is $p$.
\item[(iii)] If $q$ is a contraction, its pullback $q': C_1 \times_S S' \to C_2 \times_S S'$ is a contraction. 
\item[(iv)] If $q$ is a contraction then $C_1$ and $C_2$ have the same genus (in every geometric fiber).
\Rachel{This part is new, please check}
\item[(v)] If $q$ is a contraction (resp. contraction of curves represented by schemes) then it is an epimorphism in the category of separated algebraic spaces over $S$ (resp. arbitrary $S$-schemes).
\end{enumerate}
\end{lem}
%\Rachel{I think we use (ii) in B.19 and (iii) is used to prove (iv) (as far as I can tell) and also to prove B.10.}
\begin{proof}
Part (i) is immediate. For part (ii), observe that we have a commuting diagram
\begin{equation}\label{eq:basic-properties}
\begin{tikzcd}
\mls O_{C_3} \arrow[r] \arrow[d] & R(p\circ q)_* \mls O_{C_1} \arrow[d, equal]\\
Rp_* \mls O_{C_2} \arrow[r] & R p_*Rq_*\mls O_{C_1}
\end{tikzcd}
\end{equation}
where the horizontal arrows are isomorphisms since $q$ and $p \circ q$ are contractions.

To prove (iii) it suffices to show that if  $h: S' \to S$ denotes the given map then  the base change map
\begin{equation}\label{eq:properties2}
Lh^* Rq_*\mls O_{C_1} \longrightarrow Rq'_{*} \mls O_{C_1 \times_S S'}
\end{equation}
is an isomorphism.  To verify that \eqref{eq:properties2} is an isomorphism we may work locally on $S$ and $S'$ and $C_2$. Therefore we write $S = \Sp (A)$ and $S' = \Sp (A')$ and fix $\Sp (B)\subset C_2$ an affine open subset.  Let $C_{1, B}$ denote $q^{-1}(\Sp (B))$ and write $B'$ for $B\otimes _AA'$ so we have a diagram of cartesian squares
$$
\xymatrix{
C_{1, B'}\ar[d]\ar[r]& C_{1, B}\ar[d]\\
\Sp (B')\ar[r]\ar[d]& \Sp (B)\ar[d]\\
\Sp (A')\ar[r]& \Sp (A),}
$$
and we need to show that the natural map
\begin{equation}\label{E:basechange3}
R\Gamma (C_{1, B}, \mls O_{C_{1, B}})\otimes ^{\mathbf{L}}_BB'\rightarrow R\Gamma (C_{1, B'}, \mls O_{C_{1, B'}})
\end{equation}
is an isomorphism.
Since $B$ is flat over $A$ we can also identify the left side of this map with $R\Gamma (C_{1, B}, \mls O_{C_{1, B}})\otimes ^{\mathbf{L}}_AA'$ and the map in question with the base change map for the cartesian square
$$
\xymatrix{
C_{1, B'}\ar[d]\ar[r]& C_{1, B}\ar[d]\\
\Sp (A')\ar[r]& \Sp (A).}
$$
That \eqref{E:basechange3} is an isomorphism now follows from the fact that $C_{1, B}$ is flat over $A$ and cohomology and base change \cite[\href{https://stacks.math.columbia.edu/tag/0A1K}{Tag 0A1K}]{stacks-project} proving (iii). 

For part (iv), observe that by part (3) the morphism $q_{\bar s}: C_{1,\bar s} \to C_{2, \bar s}$ is a contraction for every geometric point $\bar s \to S$, and therefore
$$
H^1(C_{2, \bar s}, \mls O_{C_{2, \bar s}}) = H^1(C_{2, \bar s}, Rq_{\bar s*}\mls O_{C_{1, \bar s}}) = H^1(C_{1, \bar s}, \mls O_{C_{1, \bar s}}).
$$

For part (v), first note a contraction $q$ of scheme curves is an epimorphism in the category of schemes since it is surjecive and the induced map $\mls O_C \to q_*\mls O_D$ is an isomorphism. Now suppose $X$ is an algebraic space with separated morphism $X \to S$ and let $f_1, f_2: D \to X$ be morphisms such that $f_1 \circ q = f_2 \circ q$. To show $f_1=f_2$ we may replace $X$ with the scheme theoretic image of the morphisms $f_i \circ q$ and hence assume these maps are surjective. In this case $X$ is proper over $S$ (by \cite[\href{https://stacks.math.columbia.edu/tag/08AJ}{Tag 08AJ}]{stacks-project}), locally of finite presentation, and has fibers that are of dimension $\leq 1$, hence by \ref{lem:is-a-scheme} it is \'etale-locally on $S$ represented by a scheme. On the other hand to shhow $f_1=f_2$ we may make an \'etale base change on $S$. Via such a base change we can arrange for $C, D,$ and $X$ to all be represented by schemes, and the result follows from the scheme case.

\Rachel{TODO FROM HERE}
\end{proof}

There are two basic examples of contraction morphisms (see \ref{P:B.21}). Our main result about contractions is \ref{C:sequences} which says that all contractions are in some sense built from these two. 

\begin{example}\label{P:B.21}
Let $C$ be a prestable curve over a separably closed field $k$, so by \cite[\href{https://stacks.math.columbia.edu/tag/0C4D}{Tag 0C4D}]{stacks-project} all the nodes of $C$ have residue field equal to $k$. Let $E \subset C$ be a rational irreducible component. There are two explicit contraction morphisms in this setting:
\begin{itemize}
\item $E$ is a \emph{rational bridge} if it contains two nodes and no marked points. There is a contraction $C \to C'$ sending $E$ to a node of $C'$ \cite[\href{https://stacks.math.columbia.edu/tag/0E3M}{Tag 0E3M}]{stacks-project}. 
\item $E$ is a \emph{rational tail} if it contains one node and $\sum_{I, s_i \in E} a_i \leq 1$. There is a contraction $C \to C'$ sending $E$ to a smooth point of $C'$ \cite[\href{https://stacks.math.columbia.edu/tag/0E3H}{Tag 0E3H}]{stacks-project}.
\end{itemize}
\end{example}

\begin{thm}\label{C:sequences}
Let $q:C\rightarrow C'$ be a contraction of prestable curves over a strictly henselian local ring.  Then $q$ can be written as a sequence of contractions
$$
C = C_0\rightarrow C_1\rightarrow \cdots \rightarrow C_r = C',
$$
where each contraction $C_i\rightarrow C_{i+1}$ either contracts a rational bridge or a rational tail in the closed fiber (see \ref{P:B.21}). 
\end{thm}

We prove \ref{C:sequences} in \ref{SS:applications} below.

\subsection{Line bundles and contractions}
We now make a careful study of line bundles and contraction morphisms. One application will be showing that contractions are unique when they exist \ref{L:B.12}. 
%Using this uniqueness, we prove \ref{C:sequences} in the case when the base is a separably closed field (this is \ref{L:B.16}).
For a prestable curve $C/S$ let $LB(C)$ denote the category of line bundles on $C$. (The notation $\mls Pic(C)$ might be more natural but we avoid that notation since it is often used to denote the groupoid of line bundles and here we consider also non-invertible maps).

\begin{lem}\label{L:B.11} Let $q:C\rightarrow C'$ be a contraction over a scheme $S$.

(i) The pullback functor $q^*:LB(C')\rightarrow LB(C)$ is fully faithful.

(ii) If $\mls L$ is a line bundle on $C$ and $s\in S$ is a point such that the restriction $\mls L_{\bar s}$ to the geometric fiber $C_{ \bar s}$ over $s$ is in the image of $LB(C'_{\bar s})$ then there exists a neighborhood $U\subset S$ of $s$ such that $\mls L|_{C_{U}}$ is in the image of $LB(C'_{U})$.
\end{lem}
\begin{proof}
    For (i), note that for two line bundles $\mls L_1, \mls L_2\in LB(C')$ the pullback map is given by
    $$
    \text{Hom}(\mls L_1, \mls L_2)\simeq H^0(C', \mls L_1^\vee \otimes \mls L_2)\xrightarrow{\sim} H^0(C, q^*\mls L_1^\vee \otimes q^*\mls L_2)\simeq \text{Hom}(q^*\mls L_1, q^*\mls L_2),
    $$
    where the middle isomorphism follows from the fact that $\mls O_{C'}\rightarrow Rq_*\mls O_{C}$ is an isomorphism.

    For (ii) it suffices to show that $q_*\mls L$ is a line bundle on $C$ and the adjunction map $q^*q_*\mls L\rightarrow \mls L$ is an isomorphism, at least after shrinking on $S$. Let $A$ be the completion of $\mls O_{S, \bar s}$. Since $\Sp(A) \to S$ is faithfully flat, it suffices to show these hold after replacing $S$ by $\Sp(A)$.  After this replacement we construct a line bundle $\mls M$ on $C'$ pulling back to $\mls L$ using deformation theory as follows.  Let $\mls M_{\bar s}$ be a line bundle on $C'_{ \bar s}$ whose pullback to $C_{\bar s}$ is $\mls L_{\bar s}$.  Since the deformation theory of $\mls M_{\bar s}$ (resp. $\mls L_{\bar s}$) is governed by $H^*(C_{\bar s}, \mls O_{C_{\bar s}})$ (resp. $H^*(C'_{ \bar s}, \mls O_{C', \bar s})$) and the map $H^*(C_{ \bar s}, \mls O_{C_{ \bar s}})\rightarrow H^*(C'_{ \bar s}, \mls O_{C'_{ \bar s}})$ is an isomorphism, we conclude that for each power $\mathfrak{m}^n$ of the maximal ideal in $A$ the line bundle $\mls L_{A_n}$ on $C_{ A_n}:= C\otimes _A(A/\mathfrak{m}^n)$ is obtained by pullback from a unique line bundle $\mls M_{A_n}$ on $C'_{ A_n}$.  By the Grothendieck existence theorem it then follows that $\mls L$ is the pullback of a line bundle $\mls M$ on $C'$. 
    
    To finish, note that by the projection formula we have 
    \begin{equation}\label{eq:B.15.1}
    q_*\mls L \simeq q_*q^*\mls M\simeq (q_*\mls O_C)\otimes _{\mls O_{C'}}\mls M\simeq \mls M,\end{equation} where the last isomorphism uses the fact that $q_*\mls O_C\simeq \mls O_{C'}$. So $q_*\mls L$ is a line bundle as required. Under this isomorphism, the adjunction map $q^*q_*\mls L\rightarrow \mls L$ is identified with the given isomorphism $q^*\mls M\simeq \mls L$.
    % Again using that $\mls O_{C'}\simeq Rq_*\mls O_{C}$ we conclude that $\mls M\simeq q_*\mls L$ as desired.  
\end{proof}

\begin{lem}\label{L:B.12} Let $q_1:C\rightarrow C_1$ and $q_2:C\rightarrow C_2$ be two contractions of prestable curves over a scheme $S$.  Let $\mls L_2$ be a relatively ample invertible sheaf on $C_2$ and let $\mls L$ denote $q_2^*\mls L_2$.  Then there exists a morphism $p:C_1\rightarrow C_2$ such that $p\circ q_1 = q_2$ if and only if the line bundle $\mls L$ descends to a line bundle on $C_1$.  Furthermore such a  morphism $p$ is unique and a contraction.
\end{lem}
\begin{proof}
%Let $\pi _i:C_i\rightarrow S$ and $\pi :C\rightarrow S$ be the structure morphisms. 
That $p$ is a contraction (when it exists) follows from \ref{L:properties}.

For the ``only if'' direction note that given a morphism $p:C_1\rightarrow C_2$ such that $p\circ q_1 = q_2$ the sheaf $p^*\mls L_1$ is a line bundle on $C_1$ whose pullback to $C$ is $\mls L$.

To prove the ``if'' direction and the uniqueness statment of the lemma we may replace $S$ by an \'etale cover (by descent), so it suffices to consider the case when $S$ is affine, in which case $\mls L_2$ is ample on $C_2$. In this case we have
$$
C_2\simeq \mathrm{Proj}(\oplus _{n\geq 0}H^0(C_2, \mls L_2^{\otimes n}))\simeq \mathrm{Proj}(\oplus _{n\geq 0}H^0(C, \mls L^{\otimes n}))
$$
and the map $q_2$ is identified with the natural map $C\rightarrow \mathrm{Proj}(\oplus _{n\geq 0}H^0(C, \mls L^{\otimes n}))$.  Moreover, since $\mls L_2$ is ample, for $n$ sufficiently large the sheaf $\mls L^{\otimes n}$ is generated by global sections and hence the pullback $\mls L_2^{\otimes n}$ is as well.

To see the ``if'' part of the lemma, let $\mls L_1$ be a line bundle on $C_1$ with $q_1^*\mls L_1\simeq \mls L$.  Then since $q_1$ is a contraction the pullback map $H^0(C_1, \mls L_1^{\otimes n})\rightarrow H^0(C, \mls L^{\otimes n})$ is an isomorphism.  This implies that for $n$ sufficiently large the sheaf $\mls L_1^{\otimes n}$ is generated by global sections and that the map 
\begin{equation}\label{eq:B.16.1}
C_1\rightarrow \mathrm{Proj}(\oplus _{n\geq 0}H^0(C_1, \mls L_1^{\otimes n}))\simeq \mathrm{Proj}(\oplus _{n\geq 0}H^0(C, \mls L^{\otimes n}))\simeq C_2
\end{equation}
defines a morphism $p:C_1\rightarrow C_2$ such that $p\circ q_1 = q_2$ proving the ``if'' direction.  

For uniqueness, let $p': C_1 \to C_2$ be a morphism such that $p' \circ q_1 = q_2$. Then 
\[
p'^*\mls L_2 \simeq q_{1, *}q_1^*p'^*\mls L_2 \simeq q_{1, *}q_2^*\mls L_2 \simeq q_{1, *} \mls L \simeq q_{1, *}q_1^*\mls L_1 \simeq \mls L_1,
\]
where the first and last isomorphisms follow from the projection formula and the fact that $q_1$ is a contraction (see e.g. \eqref{eq:B.15.1}). In fact by \ref{L:B.11}(i) this isomorphism $\alpha: p'^*\mls L_2 \to \mls L_1$ is the unique one such that the composition
\[
q_2^*\mls L_2 \simeq p_1^*p'^*\mls L_2 \xrightarrow{q_1^*\alpha} q_1^*\mls L_1 \simeq \mls L
\]
agrees with the identification already given. Since $\mls L_2$ is ample on $C_2$ we have an identification of $p'$ with a morphism of the form $\eqref{eq:B.16.1}$ where the isomorphism 
$\mathrm{Proj}(\oplus _{n\geq 0}H^0(C_1, \mls L_1^{\otimes n}))\simeq \mathrm{Proj}(\oplus _{n\geq 0}H^0(C, \mls L^{\otimes n}))$ is the one induced by $\alpha$. Now uniqueness of $\alpha$ implies that $p'$ agrees with $p$.
\end{proof}

\begin{lem}\label{L:B.13} Let $q_1:C\rightarrow C_1$ and $q_2:C\rightarrow C_2$ be two contractions of prestable curves over a scheme $S$.  If $\bar s\rightarrow S$ is a geometric point for which there exists a morphism $p_{\bar s}:C_{1, \bar s}\rightarrow C_{2, \bar s}$ such that $p_{\bar s}\circ q_{1, \bar s} = q_{2, \bar s}$, then after replacing $S$ by an \'etale neighborhood of $\bar s$ there exists a morphism $p:C_1\rightarrow C_2$ such that $p\circ q_1 = q_2$, and such a morphism $p$ is unique.
\end{lem}
\begin{proof}
    Combining \ref{L:B.11} (ii) and \ref{L:B.12} we can, after replacing $S$ by an \'etale neighborhood of $\bar s$, find a relatively ample invertible sheaf $\mls L_2$ on $C_2$ such that $q_2^*\mls L_2$ descends to an invertible sheaf on $C_1$.  The result then follows from \ref{L:B.12}.
\end{proof}

\subsection{Weighted stable maps and contractions}\label{SS:4.10}
We now recall the notion of a \textit{weighted stable map} from \cite{AlexeevGuy, bayermanin}, an extension of the notion of weighted stable curves in \cite{Hassett}. This definition and the contraction result \ref{T:contractthm} proved in this section will figure prominently in the companion paper \cite{OWarticle2}. In this paper, we use \ref{T:contractthm} to prove \ref{C:sequences}.

Let $S$ be a scheme, let $X \to S$ be a separated morphism locally of finite presentation from an algebraic space. Fix a nonnegative integer $g$ and rational \textit{weights} $\ba = (a_1, \ldots, a_n)$ with $a_i \in \mathbf{Q} \cap (0, 1]$, satisfying
\[
2g-2 + a_1 + \ldots + a_n > 0.
\]
A
  \emph{prestable map to $X$ of type $(g, \ba)$} is a tuple 
$(C/S, \{s_i\}_{i=1}^n, f:C\rightarrow X)$
where $(C/S, \{s_i\}_{i=1}^n)$ is an $n$-marked prestable curve over $S$ of genus $g$, $f$ is an $S$-morphism, and
for every geometric point $\bar s\rightarrow S$ and $x\in C_{\bar s}$ we have
$$
\sum _{i, \, s_i(\bar s) = x}a_i\leq 1.
$$
% An isomorphism of prestable maps of type $(g, \ba)$ over $S$ is an isomorphism of the associated $n$-marked prestable curves over $S$ that commutes with the morphisms to $X$.
The prestable map is \emph{stable} if moreover the following conditions holds for every geometric point $\bar s\rightarrow S$:
\begin{enumerate}
\item [] If $E\subset C_{\bar s}$ is an irreducible component such that $f_{\bar s}(E)$ has dimension zero, then either the normalization $\widetilde E$ of $E$ has positive genus, or
\begin{equation}\label{eq:numerics}
\#\{e\in \widetilde E \mid \text{$e$ maps to a node of $C_{\bar s}$}\}+\sum _{i, s_i(\bar s)\in E}a_i>2.
\end{equation}
\end{enumerate}

The main result regarding contractions and weighted stable maps is the following.

\begin{thm}\label{T:contractthm} Let $(C, \{s_i\}_{i=1}^n, f:C \to X)$ be a prestable map of type $(g, \ba)$. There exists a unique factorization
$$
\xymatrix{
C\ar@/^2pc/[rr]^-f \ar[r]^-{q }& C^c\ar[r]^-{f^c}& X,}
$$
where $q$ is a contraction and $(C^c/S, \{q\circ s_i\}_{i=1}^n; f^c)$ is a stable map to $X$ of type $(g, \ba)$.
\end{thm}

In proving \ref{T:contractthm} we will in fact prove the following slightly stronger version of the uniqueness part of \ref{T:contractthm}. (The following theorem also explains what is meant by uniqueness in \ref{T:contractthm}). 
%We will prove both results together.

\begin{thm}\label{T:stronger}
Let $(C, \{s_i\}_{i=1}^n, f)$ be a prestable map to $X$ of type $(g, \ba)$. Suppose we have two factorizations of $f: C \to X$ as indicated by the solid arrows in the diagram 
  $$
    \xymatrix{
    & C'\ar[rd]^-{f'}\ar@{-->}[dd]^-{p}& \\
    C\ar[ru]^-{q'}\ar[rd]_-{q}&& X.\\
    & C^c\ar[ru]_-{f^c}& }
    $$
    Assume that $q'$ and $q$ are contractions, that $(C', \{q\circ s_i)\}_{i=1}^n, f')$ is prestable of type $(g, \ba)$, and that $(C^c, \{q\circ s_i)\}_{i=1}^n; f^c)$ is stable of type $(g, \ba)$. Then there exists a unique morphism $p: C' \to C^c$ such that $p \circ q' = q$ and $f^c \circ p = f'$.

\end{thm}

\begin{rem}
If $X$ is a scheme projective over $S$, if $C$ is also a scheme and if $(C/S, \{s_i\}_{i=1}^n; f)$ is stable of weight $\bb=(b_1, \ldots, b_n)$ for some choice of $b_i \in (0, 1]$, then the existence part of \ref{T:contractthm} is proved in \cite[3.1]{AlexeevGuy} \cite[1.2.1]{bayermanin}  (see also \cite[Thm~4.1]{Hassett}). 
\end{rem}

We prove \ref{T:contractthm} and \ref{T:stronger} together in several steps \ref{SS:first}-\ref{SS:last} below. 

\subsubsection{Contractions over separably closed fields}
\label{SS:first}
To begin we completely describe contractions over separably closed fields. Note that in particular that when $X=S$, part (iv) of the next lemma is \ref{C:sequences} in the case when the base is a separably closed field.

\begin{lem}\label{L:B.16} Let $q:(C/k, \{s_i\}_{i=1}^n)\rightarrow (C'/k, \{s_i'\}_{i=1}^n)$ be a contraction of prestable $n$-marked curves over a separably closed field $k$ and let $g:C'\rightarrow X$ be a morphism.  Assume that $(C/k, \{s_i\}_{i=1}^n, g\circ q)$ and $(C'/k, \{s_i'\}_{i=1}^n, g)$ are prestable maps of type $(g, \ba) $.
\begin{enumerate}
    \item [(i)] Any irreducible component $E\subset C$ contracted in $C'$ is a smooth rational curve.
    \item [(ii)] 
If $q$ is not an isomorphism, there exists an irreducible $E\subset C$ contracted in $C'$ with $E$  a smooth rational curve for which one of the following holds:
\begin{enumerate}
    \item [(a)] $E$ contains two nodes and no marked points.
    \item [(b)] $E$ contains one node and $\sum _{i, s_i\in E}a_i\leq 1.$
\end{enumerate}
\item [(iii)] 
If $(C'/k, \{s_i'\}_{i=1}^n; g)$ is stable then any component $E\subset C$ as in (ii)  is contracted in $C'$.
\item[(iv)] The contraction $q$ can be factored as a sequence of contractions$$
C = C_0\rightarrow C_1\rightarrow \cdots \rightarrow C_r = C',
$$
where each $C_i\rightarrow C_{i+1}$ either contracts a rational bridge or a rational tail as in \ref{P:B.21}.
\item[(v)] The contraction $q$ has a unique ``greedy factorization''
\[
C = C_0 \to C_1 \to \ldots \to C_r = C'
\]
where for $i=0, \ldots, r-2$ the contraction $C_i \to C_{i+1}$ contracts all rational tails of $C_i$ with zero-dimensional images in $C'$, and every irreducible component contracted by $C_{r-1} \to C_r$ is a rational bridge.

\end{enumerate}
\end{lem}
\begin{proof}
For (i), 
let $D\subset C$ be the union of the components other than $E$ and let $N\subset C$ be the intersection $D\cap E$, so $N$ is a finite number of nodes.  We then have an exact sequence of sheaves on $C$
$$
0\rightarrow \mls O_C\rightarrow \mls O_D\oplus \mls O_E\rightarrow \mls O_N\rightarrow 0.
$$
Applying $Rq_*$  and noting that $R^1q_*\mls O_N = 0$ we get a surjection $R^1q_*\mls O_C\rightarrow R^1q_*\mls O_E$ which since $q$ is a contraction implies that $R^1q_*\mls O_E = 0$.  This sheaf is the skyscraper sheaf at the image point of $E$ given by $H^1(E, \mls O_E)$ and therefore $E$ has arithmetic genus $0$.  It follows that $E$ is a smooth rational curve proving (i).

A similar argument shows that for a point $x\in C'(k)$ the preimage $q^{-1}(x)$ is a tree of rational curves. This implies (ii). Indeed, if $q^{-1}(x)$ contains a component $E$ which has only one node, then since $(C'/k, \{s_i'\}_{i=1}^n; g)$ is prestable we must have $\sum _{i, s_i\in E}a_i\leq 1$ and (b) holds. If not, the tree $q^{-1}(x)$ is a chain of rational curves and any irreducible component $E$ of this chain has two nodes. We claim $E$ has no marks, in which case (a) holds; for this it is enough to show that $x$ is a node. But by the uniqueness in \ref{L:B.12} it must be that $C'$ is obtained by contracting the rational bridges in this chain via the explicit construction in \cite[\href{https://stacks.math.columbia.edu/tag/0E3M}{Tag 0E3M}]{stacks-project}. From this it follows that $x$ is a node.

%and if $x_1, x_2\in C(k)$ are the two points where this chain is attached to the rest of the curve we find that $x$ is the node obtained by removing the components in $q^{-1}(x)$ and gluing the result along $x_1$ and $x_2$.  In particular, $x$ is a node so (a) holds.

To see (iii) note that if such a component $E$ is not contracted in $C'$ then $E$ maps isomorphically to its image $E'\subset C'$ since $q_*\mls O_C = \mls O_{C'}$, and hence if $x\in E(k)$ is a smooth point then $x = q^{-1}q(x)$.  In particular, the image of a point of $E$ distinct from the nodes is a smooth point of $C'$ in $E'$.  Therefore the number of nodes of $E'$ is less than or equal to the number of nodes of $E$, and the number of marked points on $E'$, counted with weights, is less than or equal to the number of marked points on $E$, counted with weights.  Therefore if $E$ is not contracted in $C'$ we obtain a component of $C'$ which violates the stability assumption.

To prove (iv), 
if we can find some contraction $C \to C'$ that factors as a sequence of contractions of rational tails and rational bridges, then this contraction is equal to $q$ by the uniqueness of contraction morphisms in \ref{L:B.12}.
Let $D \subset C$ be the union of irreducible components contracted in $C'$. We construct the desired contraction $C \to C'$ by first contracting all rational tails in $D$ and then contracting all rational bridges. Once this is done the curves must be equal by part (ii).

Finally to prove (v),
the statement provides an algorithm for the factorization, and hence we have uniqueness. Furthermore, since the number of irreducible components of $C$ is finite the repeated contractions of rational tails in the $C_i$ arrives eventually at a morphism $C_{\ell-1}\rightarrow D$ which does not contract any rational tails. Hence every irreducible component contracted by $C_{\ell-1} \to C_\ell=C'$ is a rational bridge.

\end{proof}

\subsubsection{Reduction of \ref{T:contractthm} and \ref{T:stronger} to the case when $X$ is projective}\label{S:reduction to projective}

Let $X \to S$ be as in \ref{SS:4.10} and let $(C/S, \{s_i\}_{i=1}^n; f)$ be a prestable map to $X$ of type $(g, \ba)$.
Let $Y\subset X$ be the scheme-theoretic image of $f$.  Then $Y\rightarrow S$ is a proper morphism (by \cite[\href{https://stacks.math.columbia.edu/tag/08AJ}{Tag 08AJ}]{stacks-project}), locally of finite presentation, and has fibers that are of dimension $\leq 1$.   Let $g:C\rightarrow Y$ be the map induced by $f$.  Any factorization of $f$ through another prestable map is induced by a unique factorization of $g$, so it suffices to prove the theorems for $g:C\rightarrow Y$.  Furthermore, by descent theory and the uniqueness part of the theorems it suffices to prove the theorems after replacing $S$ by an \'etale cover.  Now by the following variant of 
\cite[\href{https://stacks.math.columbia.edu/tag/0E6F}{Tag 0E6F}]{stacks-project} (removing the flatness assumption), the algebraic space $Y$ is \'etale locally on $S$ a projective scheme.  This reduces the proofs of \ref{T:contractthm} and \ref{T:stronger} to the case when $X$ is projective. 

\begin{lem}\label{lem:is-a-scheme} Let $S$ be a scheme and $\pi :X\rightarrow S$ a proper morphism locally of finite presentation with fibers of dimension $\leq 1$.  Then there exists an \'etale cover $\{S_i\rightarrow S\}$ of $S$ such that the base changes $X_{S_i}\rightarrow S_i$ are projective.
\end{lem}
\begin{proof}
We fix a geometric point $\bar s\rightarrow S$ and show that in some \'etale neighborhood of $\bar s$ the map $\pi $ is projective.  Since the morphism $\pi $ is locally of finite presentation there exists an affine open neighborhood of $U\subset S$ of $\bar s$ and a morphism $U\rightarrow U_0$ with $U_0$ of finite type over $\mathbf{Z}$ such that $X_{U}$ is obtained by base change from a scheme $X_{U_0}$ over $U_0$ satisfying the same assumptions as in the theorem.  It therefore suffices to consider the case when $S$ is noetherian.
Since relative ampleness is an open condition, and since geometric fibers of $\pi$ admit ample line bundles, it is enough to show that a line bundle on $X_{\bar s}$ extends to a bundle on $X$, after possibly replacing $S$ by an \'etale neighborhood of $\bar s$. By a standard limit argument for this it suffices in turn to show that such an invertible sheaf exists over the base change of $\pi $ to the strict henselization of $S$ at $\bar s$.  By the Artin approximation theorem we can then even pass to the completion $\widehat {\mls O}_{S, \bar s}$ of  $\mls O_{S, \bar s}$.  By the Grothendieck existence theorem \cite[\href{https://stacks.math.columbia.edu/tag/08BE}{Tag 08BE}]{stacks-project} it then suffices to prove the following statement: If $A'\rightarrow A$ is a surjective map of artinian local rings with kernel $J$ annihilated by the maximal ideal of $A'$ and if $X_{A'}$ is a proper scheme whose reduction $X_k$ to the residue field $k$ has dimension $\leq 1$ then any invertible sheaf $\mls L$ on $X_A$ lifts to $X_{A'}$.  For this note that the map
$$
K:= \text{Ker}(\mls O_{X_{A'}}^*\rightarrow \mls O_{X_A}^*)\rightarrow \text{Ker}(\mls O_{X_{A'}}\rightarrow \mls O_{X_A}), \ \ u\mapsto u-1
$$
is an isomorphism of abelian sheaves.  Since the target of this map is coherent and $X_k$ has dimension $\leq 1$ we conclude that $H^2(X_{A'}, K) = 0$ and  that $H^1(X_{A'}, \mls O_{X_{A'}}^*)\rightarrow H^1(X_A, \mls O_{X_A}^*)$ is surjective.
\end{proof}

\subsubsection{Proof of \ref{T:stronger}}
By \ref{S:reduction to projective} we may assume that $X \to S$ is projective.
By \ref{L:B.13} it suffices to show that for every geometric point $\bar s$ the desired factorization $p$ exists in the geometric fiber over $\bar s$, which reduces the proof to the case when $S$ is the spectrum of an algebraically closed field.  
We prove the theorem in this case by induction on the number $M$ of components contracted in $C'$.  The base case is $M=0$ in which case $C'=C$ and there is nothing to prove.  For the inductive step \ref{L:B.16} implies that if $C\rightarrow C'$ is not an isomorphism then there exists a component $E\subset C$ contracted in both $C^c$ and $C'$.  Hence we may replace $C$ by the result $D$ of contracting $E$, and the induced map $D \to C'$ contracts fewer components of $D$. This induced map is also a contraction by \ref{L:properties}(ii), so we may apply the inductive hypothesis. \qed

\subsubsection{Proof of \ref{T:contractthm}}\label{SS:last}
By \ref{S:reduction to projective} we may assume that $X \to S$ is projective.
By the uniqueness already shown, it suffices to show that there exists a stable contraction \'etale locally on $S$.  Making a suitable such base change we can arrange that we have  additional sections $y_1, \dots, y_m:S\rightarrow C$ such that $(C/S, \{s_i\}_{i=1}^n\cup \{y_j\}_{j=1}^m; f)$ is stable with respect to the weight vector  $\bb := (a_1, \dots, a_n, 1, \dots, 1)$.  In this case the existence of the stabilization is shown in \cite[3.1]{AlexeevGuy}. \qed 

\subsection{Conclusion of proof of \ref{C:sequences}}\label{SS:applications}
In this subsection we use contractions of weighted prestable maps to prove \ref{C:sequences} and some other results.

\begin{cor}\label{C:B.9} With notation as in \ref{T:contractthm}, suppose $\bar s\rightarrow S$ is a geometric point over which $f_{\bar s}$ factors as 
$$
\xymatrix{
C_{\bar s}\ar[r]^{q_{\bar s}}\ar@/^2pc/[rr]^-{f_{\bar s}}&C'_{\bar s}\ar[r]^-{g_{\bar s}}& X,}
$$
with $q_{\bar s}$ a contraction of prestable curves.  Then after replacing $S$ by an \'etale neighborhood of $\bar s$ there exists a factorization of $f$ as
$$
\xymatrix{
C\ar[r]^{q}\ar@/^2pc/[rr]^-f&C'\ar[r]^-{g}& X,}
$$
with $q$ a contraction of prestable curves, inducing the given factorization over $\bar s$.
\end{cor}
\begin{proof}
Choose $m$  auxiliary smooth points on $C_{\bar s}$ such that if we assign these to have weight one, the resulting map $(C'_{\bar s}, \{q_{\bar s}(s_i)\}_{i=1}^{n+m}; g_{\bar s})$ is stable of weight $\bb = (a_1, \ldots, a_n, 1, \ldots, 1)$. After replacing $S$ by an \'etale neighborhood of $\bar s$ we may assume these auxiliary markings extend to $C$, and the desired factorization is then the stabilization of $f$ provided by \ref{T:contractthm}.

%By possibly choosing some auxiliary smooth points on $C_{\bar s}$ which we assign weight $1$ and lifting them to $C$, after replacing $S$ by an \'etale neighborhood of $\bar s$, we may assume that $(C'_{\bar s}, \{q_{\bar s}(s_i)\}; g_{\bar s})$ is a stable map of type $\ba $.  The desired factorization is then the stabilization of $f$ provided by \ref{T:contractthm}.
\end{proof}

From \ref{C:B.9} and \ref{L:properties} part (iii) we obtain the following corollary. 

\begin{cor}\label{cor:fibers}
Let $q: (C_1/S, f_1) \to (C_2/S, f_2)$ be a morphism of prestable maps. Then $q$ is a contraction if and only if for every geometric point $\bar s \to S$, the restriction $q_{\bar s}: C_{1, \bar s} \to C_{2, \bar s}$ is a contraction.
\end{cor}

\qed

Finally we prove the result \ref{C:sequences} which states that when $S$ is the spectrum of a strictly Henselian local ring, every contraction factors as a sequence of maps that contract either a rational bridge or a rational tail in the closed fiber.

\begin{proof}[Proof of \ref{C:sequences}]
Let $\bar s \to S$ be a geometric point mapping to the closed point of $S$. By \ref{L:properties} the restriction of $q$ to $q_{\bar s}: C_{\bar s} \to C'_{\bar s}$ is also a contraction, and by \ref{C:B.9} it suffices to construct such a factorization for $q_{\bar s}$. Existence of a factorization in this case follows from \ref{L:B.16} (iv). 
\end{proof}

\section{Contractions of log curves}\label{A:log-contractions}
The goal of this section is to explain how a contraction of $n$-marked prestable curves can be canonically enhanced to a morphism of log curves (with the canonical log structures defined in \ref{SS:2.19}). That is, given a contraction $q: (C/S, \{s_i\}_{i=1}^n) \to (C'/S, \{s_i'\}_{i=1}^n)$ we construct a canonical diagram
\begin{equation}\label{eq:enhancement}
\xymatrix{
(C, M_C)\ar[r]\ar[d]& (C', M_{C'})\ar[d]\\
(S, \MS{S}{C})\ar[r]& (S, \MS{S}{C'}).}
\end{equation}

We start this section by describing two explicit examples of \eqref{eq:enhancement} that will be used repeatedly in the remainder of the discussion.
 In \ref{sec:proof} we prove the existence of \eqref{eq:enhancement}, and in  \ref{S:contract-monoid} we explicitly describe the morphism of characteristic monoids induced by $q$.

 In our two examples \ref{S:contract-bridge} and \ref{S:contract-tail} we work over a separably closed field $k$ (so $S = \Sp (k)$). In this case all nodes of $C$ are $k$-points \cite[\href{https://stacks.math.columbia.edu/tag/0C4D}{Tag 0C4D}]{stacks-project}.  In each case we offer two descriptions.  The first is as a log blowup, which has the advantage of being global.  The second is local but more concrete in terms of Deligne-Faltings structures. In \ref{sec:proof} we then return to the setting of a general base.

\subsubsection{Minimal log structures}\label{S:log-structures}
As noted in \ref{SS:2.19}, if $(C/S, \{s\})$ is a prestable curve with one marked point, there is a canonical structure of a log smooth morphism of log schemes $(C, M_C) \to (S, \MS{S}{C})$. Locally it can be described as follows. Let $J$ be the set of nodes in $C$, let $\{e_j\}_{j \in J}$ be the standard basis for $\mathbf{N}^J$, and if $x \in C$ is a smooth point let $\mls O(-x)$ denote the ideal sheaf of $x$.
%, and for $x \in C$ set $C^{sh}_X = \Sp(\mls O^{sh}_{C, x})$. 
\begin{itemize}
\item The Deligne-Faltings structure corresponding to $\MS{S}{C}$ has a global chart 
\begin{align*}
\mathbf{N}^J &\to \mls Div^+(S)\\
e_j &\mapsto (\mls O \xrightarrow{0} \mls O)
\end{align*}
\item If $x \in C$ is the marked point $s$ then $x$ is a smooth point and the Deligne-Faltings structure corresponding to $M_C$ has a global chart 
\begin{align*}
\mathbf{N}^J \oplus \mathbf{N}e_s &\to \mls Div^+(S)\\
e_j &\mapsto (\mls O \xrightarrow{0} \mls O)\quad \quad\text{for}\; j \in J\\
e_s &\mapsto (\mls O(-x) \hookrightarrow \mls O)
\end{align*}
It is clear that $\mathbf{N}^J \to \mathbf{N}^J \oplus \mathbf{N}e_s$ sending $e_j$ to $e_j$ for $j \in J$ is a morphism of Deligne-Faltings structures. This induces the morphism $(C, M_C) \to (S, \MS{S}{C})$ near $x$.
\item If $x \in C$ is the node corresponding to $i \in J$ then the Deligne-Faltings structure corresponding to $M_C$ has a local chart defined as follows. Let $A = B=\Sp (k[t]_{(t)}^{sh})$, and let $x_A \in A$ (resp. $x_B \in B$) denote the closed point, so that $C_x^{sh}:=\Sp(\mls O^{sh}_{C, x})$ is isomorphic to the gluing of $A$ and $B$ at $x_A$ and $x_B$. Then the objects of $\mls Div^+(C_x^{sh})$ are given by pairs $(\mls L_1\rightarrow \mls O_A, \mls L_2\rightarrow \mls O_B)\in \mls Div^+(A)\times \mls Div^+(B)$ together with an isomorphism $\mls L_1|_x\simeq \mls L_2|_x$ in $\mls Div^+(x)$.

Let $\alpha: \mls L_A \to \mls O_{C_x^{sh}}$ be the element of $\mls Div^+(C^{sh}_x)$ given by the element $\mls O(x_A) \xrightarrow{0} \mls O_A$ on $A$ and the ideal sheaf $\mls O(-x_B) \hookrightarrow \mls O_B$ on $B$, together with the isomorphism $\mls O_A(x_A)|_x\simeq \mls O_B(-x_B)|_x$ sending the generator $1/t$ in $\mls O_A(x_A)|_{x_A}$ to the generator $t$ in $\mls O_B(-x_B)|_{x_B}$.

Similarly let $\beta: \mls L_B \to \mls O_{C_x^{sh}}$ be given by $\mls O(-x_A) \xrightarrow{0} \mls O_A$ on $A$ and $\mls O(x_B) \hookrightarrow \mls O_B$ with the gluing isomorphism again defined by the local coordinates.

Observe that $(\mls L_A, \alpha )$ is supported on $A$ and $(\mls L_B, \beta )$ is supported on $B$.  Now on $C^{sh}_x$ the Deligne-Faltings structure has a global chart
\begin{align*}
\mathbf{N}^{J \setminus \{i\}} \oplus \mathbf{N}e_A \oplus \mathbf{N}e_B &\to \mls Div^+(C^{sh}_x)\\
e_j &\mapsto (\mls O \xrightarrow{0} \mls O) \quad\quad\text{for}\;j \neq i\\
e_A&\mapsto (\mls L_A, \alpha )\\
e_B&\mapsto (\mls L_B, \beta ).
\end{align*}
Since the sum of $( \mls L_A, \alpha )$ and $(\mls L_B, \beta )$ is the trivial element of $\mls Div^+(C^{sh}_x)$, the morphism
\begin{align*}
\mathbf{N}^{J \setminus \{i\}} \oplus \mathbf{N} &\to \mathbf{N}^{J \setminus \{i\}} \oplus \mathbf{N}^2\\
e_j &\mapsto e_j \quad \quad \text{for}\;j \neq i\\
e_i &\mapsto e_A+e_B
\end{align*}
is a morphism of Deligne-Faltings structures. This induces the morphism $(C, M_C) \to (S, \MS{S}{C})$ near $x$.
\end{itemize}

\subsection{Contraction of a rational bridge}\label{S:contract-bridge}
We work over a separably closed field $k$ and construct a contraction $q:C\rightarrow C'$ over $k$ fitting into a diagram \eqref{eq:enhancement}, which contracts a single rational bridge with no markings.

\subsubsection{Description as a log blowup}

 Let $C'/k$ be a nodal curve with a node $y\in C'(k)$.  Then $\MS {\Sp (k)}{C'}\simeq N\oplus \mathbf{N}$, where the factor of $\mathbf{N}$ corresponds to the node $y$, and $\overline M_{C', y}\simeq \overline {N}\oplus \mathbf{N}^2$ with the map $\bMS {\Sp (k)}{C'}\rightarrow \overline M_{C', y}$ given by the identity map on $\overline N$ and the diagonal map $\mathbf{N}\rightarrow \mathbf{N}^2$.  Let $N\oplus \mathbf{N}^2$ denote the log structure on $\Sp (k)$ where the map to $k$ sends all nonzero elements to $0$, and let $M^\dag _{C'}$ denote the pushout in the category of log structures of the diagram
$$
\xymatrix{
N\oplus \mathbf{N}|_{C'}\ar[r]^-{\text{id}\oplus \Delta }\ar[d]& N\oplus \mathbf{N}^2|_{C'}\\
M_{C'}&}
$$
so we have a cartesian diagram of log schemes
$$
\xymatrix{
(C', M^\dag _{C'})\ar[d]\ar[r]& (C', M_{C'})\ar[d]\\
(\Sp (k), N\oplus \mathbf{N}^2)\ar[r]& (\Sp (k), \MS {S}{C'}).}
$$
We can then construct a morphism of log schemes $(C, M_C)\rightarrow (C', M^\dag _{C'})$ such that $C\rightarrow C'$ contracts a rational bridge and the induced diagram
$$
\xymatrix{
(C, M_C)\ar[r]\ar[d]& (C', M_{C'})\ar[d]\\
(\Sp (k), N\oplus \mathbf{N}^2)\ar[r]& (\Sp (k), \MS {S}{C'})}
$$
is the contraction diagram \ref{eq:enhancement}.

\begin{pg}\label{P:1.2}
The scheme $C$ will be obtained as representing a functor on the category of log schemes.  
The monoid $\overline M^\dag _{C', \bar y}\simeq \overline N\oplus (\mathbf{N}^2\oplus _{\mathbf{N}}\mathbf{N}^2)$ has two natural ideals given by two ideals in $\mathbf{N}^2\oplus _{\mathbf{N}}\mathbf{N}^2$: If $a\in \mathbf{N}^2$ is a nonzero irreducible element (there are two such) then let $I_a$ denote the ideal generated by $(a, 0)$ and $(0,a)$ (elements of $\mathbf{N}^2\oplus _{\mathbf{N}}\mathbf{N}^2$). 
\end{pg}

\begin{lem}\label{L:9.4} Let $P$ denote the monoid generated by four elements $a_1, b_1, a_2, b_2$ modulo the relation $a_1+b_1=a_2+b_2$.  If $\rho :P\rightarrow M$ is a morphism of integral sharp monoids such that $\rho (a_1, a_2)\cdot M$ is a principal ideal then the ideal $\rho (b_1, b_2)\cdot M$ is also principal, and similarly interchanging $a_i$ and $b_i$.
\end{lem}
\begin{proof}
    Suppose $\rho (a_2) = \rho (a_1)+c$ for some $c\in M$.  Then we obtain
    $$
    \rho (a_1)+\rho (b_1) = \rho (a_1)+c+\rho (b_2),
    $$
    which implies that $\rho (b_1) = c+\rho (b_2).$  Similarly if $\rho (a_1) = \rho 
    (a_2)+c$ then $\rho (b_2) = \rho (b_1)+c$.  This implies the lemma.
\end{proof}

\begin{pg}
Let $F$ be the functor on the category of fine log schemes over $(C', M_{C'}^\dag )$ which associates to a morphism $f:(T, M_T)\rightarrow (C', M_{C'}^\dag )$ the unital set if the following condition holds and the emptyset otherwise:
\begin{enumerate}
    \item [ \ ] For every geometric point $\bar t\rightarrow T$ for which $f(\bar t)= y$ the images of the ideals $I_a\subset \overline M_{C', \bar y}^\dag $ ($a\in \mathbf{N}^2$ irreducible) in $\overline M_{T,\bar t}$ generate principal ideals. 
\end{enumerate}
\end{pg}

\begin{pg} In an \'etale neighborhood of $y$ we can choose a chart $N\oplus \mathbf{N}^2\rightarrow M_{C'}$ and $\mathbf{N}^2\rightarrow \widetilde {\mc N}$ defining the  isomorphism $N\oplus \mathbf{N}^2\simeq \overline M_{C', \bar y}$. Let $J\subset \mathbf{N}^2\oplus _{\mathbf{N}}\mathbf{N}^2$ be the ideal generated by $(1,0)$ and $(0,1)$.  Then  by \ref{L:9.4} the functor $F$ can also be characterized as sending $(T, M_T)\rightarrow (C', M_{C'}^\dag )$ to the unital element if $J$ generates a principal ideal in $\overline M_T$ and the empty set otherwise. This implies that $F$ is represented by the log blowup $(C, M_C)\rightarrow (C', M_{C'}^\dag )$ as defined in 
\cite[II, 1.7.5]{Ogus}.

In particular, \'etale locally on $C'$ the functor $F$ is representable and since it is evidently a sheaf for the  \'etale topology this implies that $F$ is representable by a log algebraic space $(C, M_C)\rightarrow (C', M_{C'}^\dag )$ globally.  A calculation in local coordinates shows that $C$ is obtained by introducing a rational bridge at $y$.  
\end{pg}
  
\subsubsection{Explicit local description}
Let $C'_{\bar s'}$ denote the spectrum of the strictly henselian local $\mls O_{C', \bar s'}$.  The base change $C_{(\bar s')}:= C\times _{C'}C'_{\bar s'}$ can be described as follows.
\begin{itemize}
\item Set $A$ and $B$ equal to $\Spec((k[t]_{(t)})^{sh})$ with $x_A \in A$ and $y_B \in B$ equal to the respective closed points.
\item Let $P$ be a rational curve over $S$ with closed points $x_P$ and $y_P$.
\item $C'_{\bar s'}$ is  obtained by gluing $A$ to $B$ by identifying $x_A$ with $y_B$.
\item $C_{(\bar s')}$ is obtained by gluing $A$, $B$, and $P$ by identifying $x_A$ with $x_P$ and $y_P$ with $y_B$ (so $C_{(\bar s')}$ has three irreducible components).
\item $q: C_{(\bar s')} \to C_{\bar s'}'$ contracts the rational curve $P$.
\end{itemize}

Over $C'_{(\bar s')}$ the Deligne-Faltings structures corresponding to $\MS{S}{C}$, $\MS{S}{C'}$, $M_C$, and $M_{C'}$  all have global charts. We choose bases for the corresponding free monoids as follows:
\begin{itemize}
\item $\bMS{S}{C} = \mathbf{N}^2$ has basis $(e_{AP}, e_{PB})$ corresponding to the nodes connecting $A$ to $P$ and $P$ to $B$, respectively.
\item $\bMS{S}{C'} = \mathbf{N}$ has basis $e_{AB}$.
\item $\overline{M}_C = \mathbf{N}^4$ has basis $(e^A_{AP}, e^P_{AP}, e^P_{PB}, e^B_{PB})$, where the upper index indicates the component where the generator is supported. For example, the generator $e^A_{AP}$ corresponds to the line bundle equal to $\mls O(x_A)$ on $A$ with the zero map to $\mls O$ and equal to $\mls O(-x_A)$ on $P \cup B$ with the inclusion map to $\mls O$.
\item $\overline{M}_{C'} =\mathbf{N}^2$ has basis $(e^A_{AB}, e^B_{AB})$.
\end{itemize}
A commuting diagram of log schemes as in \eqref{eq:enhancement} is determined by the commuting diagram of charts 
\[
\begin{tikzcd}[column sep=1.2in, row sep = .5in]
\mathbf{N}^4 & \mathbf{N}^2 \arrow[l, "e^A_{AP}+e^P_{PB}\, \mapsfrom \,e^A_{AB}"', "e^P_{AP}+e^B_{PB} \,\mapsfrom\, e^B_{AB}"' {yshift=.5cm}]\\
\mathbf{N}^2 \arrow[u, "e^A_{AP}+e^P_{AP}\, \mapsfrom \,e_{AP}"{yshift=-.25cm}, "e^P_{PB}+e^B_{PB} \,\mapsfrom \, e_{BP}" {yshift=.25cm}] & \mathbf{N} \arrow[l, "e_{AP} + e_{PB} \,\mapsfrom \, e_{AB}"] \arrow[u, "e^A_{AB} + e^B_{AB} \,\mapsfrom \, e_{AB}"']
\end{tikzcd}
\]
One checks using the descriptions in \ref{S:log-structures} that this is a commuting diagram of Deligne-Faltings structures, meaning that the morphisms given also commute with the maps to $\mls Div^+$. 
For example, that the top arrow is a morphism of Deligne-Faltings structures includes the equality
\[
\left( \begin{array}{rl}
\mls O(x_A) \xrightarrow{0} \mls O & \text{on}\; A\\
\mls O(-x_P) \hookrightarrow \mls O & \text{on}\; P\\
\mls O \xrightarrow{1} \mls O & \text{on}\; B\\
\end{array}\right) + \left( \begin{array}{rl}
\mls O \xrightarrow{0} \mls O & \text{on}\; A\\
\mls O(y_P) \xrightarrow{0} \mls O & \text{on}\; P\\
\mls O(-y_B) \hookrightarrow \mls O & \text{on}\; B\\
\end{array}\right) = \left( \begin{array}{rl}
\mls O(x_A) \xrightarrow{0} \mls O & \text{on}\; A\\
\mls O \xrightarrow{0} \mls O & \text{on}\; P\\
\mls O(-y_B) \hookrightarrow \mls O & \text{on}\; B\\
\end{array}\right)
\]
in $\mls Div^+(C)$, where the right hand side is the pullback of the element of $\mls Div^+(C')$ corresponding to $e^A_{AB}$, and the left hand side is the sum of elements of $\mls Div^+(C)$ corresponding to $e^A_{AP}$ and $e^P_{PB}$, respectively.

\subsection{Contraction of a rational tail with one marking}\label{S:contract-tail}

Again we work over a separably closed field and construct a contraction $q: C\rightarrow C'$ over $k$ fitting into a diagram \eqref{eq:enhancement}, such that $q$ contracts a single rational tail with one marked point to a point $s'\in C'$.

\subsubsection{Description as a log blowup}

  Let $C'/k$ be a nodal curve.  Let $s'\in C'(k)$ be a smooth point, and consider the log structure $M_{C'}:= M_{C'}^{\node}\oplus _{\mls O_{C'}^*}M^{s'}$ incorporating the point $s'$.  There is a log structure $\MS {\Sp (k)}{C'}\oplus \mathbf{N}$ with map to $k$ sending the  nonzero elements of $\mathbf{N}$ to $0$, and similarly there is a log structure $M_{C'}\oplus \mathbf{N}$.  There is a coherent sheaf of ideals $J\subset \overline M_{C'}\oplus \mathbf{N}$ whose stalks at points different from $s'$ are given by $\mathbf{N}$ and whose stalk at $s'$ is given by the two generators of $\overline M^{s'}_{s'}\oplus \mathbf{N}\simeq \mathbf{N}^2$.  We can then consider the log blowup $(C, M_C)\rightarrow (C', M_{C'}\oplus \mathbf{N})$ which sits in a commutative diagram
    $$
    \xymatrix{
    (C, M_C)\ar[r]\ar[d] & (C', M_{C'})\ar[d]\\
    (\Sp (k), \MS {\Sp (k)}{C'}\oplus \mathbf{N})\ar[r]& (\Sp (k), \MS{\Sp (k)}{C'}).}
    $$
    Then $C$ is a nodal curve obtained by attaching a copy of $\mathbf{P}^1$ at $s'$, and there exists a unique smooth point $s\in C(k)$ on the contracted component such that $M_C = M_C^{\node}\oplus _{\mls O_{C}^*}M^s$ and the map $(C, M_C)\rightarrow (\Sp (k), \MS {\Sp (k)}{C'}\oplus \mathbf{N})$ is identified with $(C, M_C)\rightarrow (\Sp (k), \MS {\Sp (k)}{C})$.  These statements follow from a  calculation of the log blowup in local coordinates where it is described by
    $$
    \text{Proj}(k[x][u,v]/(xu))\rightarrow \Sp (k[x])
    $$
and the observation that  the log structure on $C$ obtained in this way is special in the sense of \cite[2.6]{OlssonTohoku}.

\subsubsection{Explicit local description}
We define a contraction $q: (C, \{s\}) \to (C', \{s'\})$ of 1-marked curves over $S$ as follows. Let $A$ and $P$ be curves over $S$ as in \ref{S:contract-bridge}.
\begin{itemize}
\item $(C, \{s\})$ is the curve obtained by gluing $A$ and $P$ by identifying $x_A$ with $x_P$  with marking $s = y_P$ (so $C$ has two irreducible components).
\item $(C', \{s'\})$ is equal to $A$ with marking $s = x_A$.
\item $q: (C, \{s\}) \to (C', \{s'\})$ contracts the rational curve $P$.
\end{itemize}
The Deligne-Faltings structures corresponding to $\MS{S}{C}$, $\MS{S}{C'}$, $M_C$, and $M_{C'}$  all have global charts. We choose ordered bases for the corresponding free monoids as follows:
\begin{itemize}
\item $\bMS{S}{C} = \mathbf{N}$ has basis $e_{AP}$.
\item $\bMS{S}{C'} = 0$.
\item $\overline{M}_C = \mathbf{N}^3$ has basis $(e^P_{AP}, e^A_{AP}, e_s)$, where $e^P_{AP}$ is supported on $P$, $e^A_{AP}$ is supported on $A$, and $e_s$ is supported on $s \in P$.
\item $\overline{M}_{C'} =\mathbf{N}$ has basis $e_{s'}$.
\end{itemize}
A commuting diagram of log schemes as in \eqref{eq:enhancement} is determined by the commuting diagram of charts 
\[
\begin{tikzcd}[column sep=1.2in, row sep = .5in]
\mathbf{N}^3 & \mathbf{N} \arrow[l, "e^P_{AP}+ e_s\, \mapsfrom \,e_{s'}"']\\
\mathbf{N} \arrow[u, "e^A_{AP}+e^P_{AP}\, \mapsfrom \,e_{AP}"] & 0 \arrow[l] \arrow[u]
\end{tikzcd}
\]
As in \ref{S:contract-bridge} one checks that this is a commuting diagram of Deligne-Faltings structures. For example, that the top arrow is a morphism of Deligne-Faltings structures amounts to the equality
\[
\left( \begin{array}{rl}
\mls O(x_P) \xrightarrow{0} \mls O & \text{on}\; P\\
\mls O(-x_A) \hookrightarrow \mls O & \text{on}\; A\\
\end{array}\right) + \left( \begin{array}{rl}
\mls O(-y_P) \hookrightarrow \mls O & \text{on}\; P\\
\mls O \xrightarrow{1} \mls O & \text{on}\; A\\
\end{array}\right) = \left( \begin{array}{rl}
\mls O \xrightarrow{0} \mls O & \text{on}\; P\\
\mls O(-x_A) \hookrightarrow \mls O & \text{on}\; A\\
\end{array}\right)
\]
in $\mls Div^+(C)$, where the right hand side is the pullback of the element of $\mls Div^+(C')$ corresponding to $e_{s'}$.

\subsection{Canonical log enhancements of contractions}\label{sec:proof}

Let $S$ be a scheme and let $(C/S, \{s_i\}_{i=1}^n)$ be an $n$-marked prestable curve over $S$. Recall from \ref{SS:2.19} that the canonical log structure $M_C$ decomposes as $M_C = M^{\node}_C \oplus_{\mls O^*_C} M^{\underline{s}}$. Define
\[
M^i_C := M^{\node}_C \oplus_{\mls O^*_C} M^{s_i}.
\]
Note that $M_C$ is the pushout of the diagram
\begin{equation}\label{eq:Mis}
\xymatrix{
M^{\node \oplus n}_C\ar[d]^-{\text{sum}}\ar[r]& \oplus _{i=1}^nM^i_C\\
M^{\node }_C.& }
\end{equation}

Let $q: (C/S, \{s_i\}_{i=1}^n) \to (C'/S, \{s_i'\})$ be a contraction of $n$-marked prestable curves.
\begin{prop}\label{prop:contracting}
There is a canonical isomorphism of log structures
\begin{equation}\label{eq:toprove}
q_* M_C^{\node} \oplus_{\mls O^*_{C'} } M^{s_i'} \xrightarrow{\sim} q_*M^i_C\end{equation} induced by the canonical morphism $M_C^{\node}  \to M^i_C$ and a unique inclusion of log structures $M^{s_i'} \hookrightarrow q_*M^i_C$.
\end{prop}

Before proving \ref{prop:contracting} we give a lemma that will be used in the proof and elsewhere.

\begin{lem}\label{L:contracting}
If $M$ is any log structure on $C$, the pushforward $q_*M$ is also a log structure and there is an inclusion $\overline{q_*M} \hookrightarrow q_*\overline{M}$. If $\bar y \to C'$ is a geometric point with fiber $P = q^{-1}(\bar y)$, then $(\overline{q_*M})_{\bar y}$ is canonically isomorphic to the elements of $H^0(P, \overline{M}|_P)$ whose associated $\mls O_P^*$-torsor of lifting to $M|_P$ induce the trivial line bundle on $P$.
\end{lem}
\begin{proof}
One sees that $q_*M$ is a log structure by the same reasoning as in \cite[B.3]{MR3329675}. To compute its characteristic sheaf, we
take cohomology of the short exact sequence of groups on $C$
\[
0 \to \mls O^*_C \to M^{gp} \to \overline{M}^{gp} \to 0
\]
to get a short exact sequence
\[
\xymatrix{
0\ar[r]&  \mls O_{C'}^*\ar[r]&  q_*M^{\gp }\ar[r]&  q_*\overline M^{\gp }\ar[r]^-{\partial }&  R^1q_*\mls O_C^*.}
\]
It follows that $(\overline{q_*M})_{\bar y}$ is identified with elements of $q_*(\overline{M})_{\bar y}$ mapping to 0 in $R^1q_*\mls O^*_C$. By the proper base change theorem  \cite[\href{https://stacks.math.columbia.edu/tag/0A3S}{Tag 0A3S}]{stacks-project} we have
$$
(q_*\overline M)_{\bar y}\simeq H^0(P, \overline M|_{P}).
$$
It remains to show that the base change map
\begin{equation}\label{eq:contracting}
R^1q_*\mls O^*_C \to H^1(P, \mls O^*_P) = \Pic(P)
\end{equation}
is an isomorphism.

It suffices to work locally on $S$ so assume $S$ is strictly henselian local. The fiber $P$ is a tree of $m$ rational curves by \ref{L:B.16}, and by our assumption on $S$ we can choose sections $x_1, \ldots, x_m$ of $C \to S$ such that each section meets a unique component of $P$. It follows that the classes $[\mls O_C(x_i)] \in H^1(C, \mls O^*_C)$ map to a generating set of $\Pic(P) \simeq \mathbf{Z}^m$. We claim the resulting map $\mathbf{Z}^m \to (R^1q_*\mls O^*_C)_{\bar y}$ is an isomorphism and inverse to \eqref{eq:contracting}.

Let $C_{(\bar y)}$ denote the base change $C\times _{C'}\Sp (\mls O_{C', \bar y})$ so that $\text{Pic}(C_{(\bar y)})\simeq (R^1q_*\mls O_C^*)_{\bar y}$.  Then it suffices to show that a line bundle $\mls L$ on $C_{(\bar y)}$ whose reduction to $P$ is trivial must be trivial.  For this note that it suffices to show that $H^0(C_{(\bar y)}, \mls L)$ is a free $\mls O_{C', \bar y}$-module of rank $1$ which generates $\mls L$.  For this in turn it suffices to show that this is the case after base change to the completion $\widehat {\mls O}_{C', \bar y}$ and then by the Grothendieck existence theorem that the reduction of $\mls L$ modulo each power of the maximal ideal of $\mls O_{C', \bar y}$ is trivial.  This follows, as in the proof of \ref{L:B.11}, from the fact that $H^i(P, \mls O_P) = 0$ for $i>0$.
\end{proof}

To prove \ref{prop:contracting} it suffices to consider the case when $S$ is the spectrum of a strictly henselian local ring.  Let $s\in S$ be the closed point.
Let $\bar y\rightarrow C'$ be a geometric point over $s$  and let $P$ denote the fiber of $q$ over $\bar y$.  We then verify that the map \ref{eq:toprove} is an isomorphism on the stalks at $\bar y$.  If $\bar y$ does not meet $s_i'$ then the result is immediate, so we can further assume that $\bar y$ maps to $s_i'$.  We show in this case  that there exists a unique inclusion of log structures $M^{s_i'}_{\bar y}\hookrightarrow (q_*M_C^i)_{\bar y}$ and that the induced map \eqref{eq:toprove} induces an isomorphism on stalks at $\bar y$.

Let $C_{(\bar y)}$ denote the fiber product $C\times _{C'}\Sp (\mls O_{C', \bar y})$  The log structure $\MS{S}{C}$ decomposes as an amalgemated sum
\begin{equation}\label{E:decompose1}
\MS{S}{C} = \oplus _{\mls O_{S}^*}\mc L_e
\end{equation}
induced by the nodes of the closed fiber, where $\mc L_e$ is induced by a chart $\mathbf{N}\rightarrow \mls O_{S}$ sending $1$ to a parameter $t_e$ such that the node $e$ is \'etale locally given by $xy-t_e$ (see \cite[\S 2]{FKato} for a discussion of this).  Likewise the log structure $M_C^\node $ decomposes as an amalgemated sum
\begin{equation}\label{E:decompose}
M_C^\node = \oplus _{\mls O_C^*}M_C^e,
\end{equation}
where in an \'etale neighborhood of $e$ as above the log structure $M_C^e$ is given by the map $\mathbf{N}^2\rightarrow \mls O_{S}[x, y]/(xy-t_e)$ sending the two generators to $x$ and $y$.

Let $P$ denote the fiber of $q$ over $\bar y$, so $P$ is a tree of rational curves getting contracted to a point with one marked point.  Let $N = \{x_n\}$ be the set of nodes of $P$ (including the attachment point of $P$ to the components not in $P$) and let $I$ be the set of irreducible components of $P$. For $i\in I$ let $P_i\subset P$ denote the corresponding rational curve.  There is a distinguished component $i_0\in I$ with $P_{i_0}$ meeting the rest of the curve $C$.  Let $i_{\infty }$ denote the component containing $s_i$, and let $x_0\in N$ be the node meeting $P_{i_0}$ and the rest of the curve.  Let $\nu :\widetilde P\rightarrow P$ be the normalization, so $\widetilde P = \coprod _{i\in I}P_i$.

Let $N^c$ denote the nodes of $C_s$ not in $P$.    Then the decompositions \eqref{E:decompose1} and \eqref{E:decompose} gives a decompositions according to whether nodes lie in $P$ or $P^c$
$$
\MS {S}{C} \simeq \MS{S}{C}^{N^c}\oplus _{\mls O_S^*}\MS{S}{C}^N, \ \ M_C^\node \simeq M_C^{N^c}\oplus _{\mls O_{C}^*}M_C^N.
$$
Since the pullback map $f^*\MS{S}{C}^{N^c}\rightarrow M_C^{N^c}$ is an isomorphism over $C_{(\bar y)}$ we have (using \ref{L:contracting})
$$
(q_*M_C^\node )_{\bar y}\simeq f^*\MS{S, \bar y}{C}^{N^c}\oplus _{\mls O_{C', \bar y}^*}q_*M_C^N, \ \ q_*M_C^i\simeq f^*\MS{S, \bar y}{C}^{N^c}\oplus _{\mls O_{C', \bar y}^*}q_*(M_C^N\oplus _{\mls O_{C}^*}M^{s_i}).
$$
It therefore suffices to prove the analogous statement for $M_C^N$ in \ref{prop:contracting}.  
This we do by explicit calculation.

Write $N = N'\cup \{x_0\}$ so that $M_C^N = M_C^{N'}\oplus _{\mls O_{C}^*}M_C^{x_0}$.
Since $P$ is a tree, for each $x_i\in N'$ we can write $P=P_+^i\cup P_-^i$ as the gluing of two connected trees of rational curves glued together at $x_i$, and with $P_+^i$ meeting the point $x_0$.  Let $\nu _i:P_+^i\coprod P_-^i\rightarrow P$ be the projection.  Then we have $\overline M_C^{x_i}|_P = \nu _{i*}\mathbf{N}$ (see for example \cite[\S 2]{FKato}).  Similarly $\overline M^{x_0}_C|_P\simeq \mathbf{N}_P\oplus \mathbf{N}_{x_0}$, where $\mathbf{N}_{x_0}$ is the skyscraper sheaf at $x_0$.  We therefore find that 
$$
q_*(\overline M_C^N\oplus \overline M^{s_i})_{\bar y}\simeq (t_{0+}^{\mathbf{N}}\oplus t_{0-}^{\mathbf{N}})\oplus (\bigoplus _{x_j\in N'}(t_{j+}^{\mathbf{N}}\oplus t_{j-}^{\mathbf{N}}))\oplus w^{\mathbf{N}},
$$
where we write $t^{\mathbf{N}}$ for the free monoid $\mathbf{N}$ with generator $t$ and $t_{j+}$ (resp. $t_{j-1}$) is the generator which is $1$ on $P_{+}^j$ (resp. $P_{-}^j$) and $0$ on $P_{-}^j$ (resp. $P_{+}^j$).  The generator $t_{0+}$ vanishes on all of $P-\{x_0\}$ and the generator $t_{0-}$ is $1$ on $P$.  The factor $w^{\mathbf{N}}$ is $H^0(P, \overline M^{s_i})$.

Since $P$ is a tree of rational curves, we have $\text{Pic}(P) = \oplus _{i\in I}\mathbf{Z}\cdot f_i$, where $f_i$ is the class of a line bundle which restricts to $\mls O_{P_i}(1)$ on $P_i$ and $\mls O_{P_j}$ on $P_j$ for $j\neq i$.  The map 
\begin{equation}\label{E:picmap}
q_*(\overline M_C^N\oplus \overline M^{s_i})_{\bar y}\rightarrow \text{Pic}(P)
\end{equation}
is given explicitly as follows.  The generator $t_{0+}$ (resp. $t_{0-}$) maps to $f_{i_0}$ (resp. $-f_{i_0}$), and $w$ maps to $-f_{i_\infty }$.  Finally for $x_{j}\in N'$ let $i^{j}_+$ (resp. $i^{j}_-$) be the component of $P^j_+$ (resp. $P^j_-$) containing $x_j$.  Then the map \eqref{E:picmap} sends $t_{j+}$ (resp. $t_{j-}$) to $f_{i^j_+}-f_{i^j_-}$ (resp. $-f_{i^j_+}+f_{i^j_-}$).

Let $E$ denote the preimage of $0$ under the map \eqref{E:picmap}.Note that the diagonal elements $t_{j+}+t_{j-}$ define elements of $E$ corresponding to the inclusion $\mathbf{N}^N\simeq \overline M_{C\rightarrow S, \bar y}^N\hookrightarrow q_*(\overline M_C^N\oplus \overline M^{s_i})_{\bar y}$.

\begin{lem}\label{L:5.10c}
(i) The map $\overline M_{C\rightarrow S, \bar y}^N\rightarrow (q_*M_{C}^N)_{\bar y}$ is an isomorphism.

(ii) There exists a unique element $w'\in E$ with coefficient $1$ of $w$.  The coefficient of $t_{0+}$ (resp. $t_{0-}$) in $w'$ is $0$ (resp. $1$).

(iii) The induced map 
$$
\overline M_{C\rightarrow S, \bar y}^N\oplus w^{\prime \mathbf{N}}\rightarrow E
$$
is an isomorphism.
\end{lem}
\begin{proof}
Let $u\in (t_{0+}^{\mathbf{N}}\oplus t_{0-}^{\mathbf{N}})\oplus (\bigoplus _{x_j\in N'}(t_{j+}^{\mathbf{N}}\oplus t_{j-}^{\mathbf{N}}))$ be an element mapping to $0$ in $\text{Pic}(P)$.  Define a node $x_j\in N'$ to be \emph{$u$-unbalanced} if the coefficients of $t_{j+}$ and $t_{j-}$ are not equal.  Then we see from the explicit description of the map to the Picard group that each of $P_{j+}$ and $P_{j-}$ must also contain an additional $u$-unbalanced node.  It follows that we can find a path in the graph associated to $P$ of edges each of which contains at least two vertices.  This is impossible since this graph is a finite tree.  This argument also shows that the coefficients of $t_{0+}$ and $t_{0-}$ must be the same proving (i).

For (ii) and (iii) we first define an element $w'\in E$ with the indicated coefficients.  Notice that since $P$ is a tree, there exists a unique path $\gamma $ in $P$ connection $x_0$ to $s_i$.  Let $w'$ be the element $(\sum _{x_j\in \gamma }t_{j-})+w$ (included in the sum is $x_0$).  Then this element has the desired properties and to complete the proof we have to show that if $u\in E$ is an arbitrary element with coefficient $\alpha _w$ of $w$, then $u$ can be written uniquely as a sum of $\alpha _ww'$ and a unique element $\sum _{j}\beta _j(t_{j+}+t_{j-1})$.

We do this by induction on the number of components of $P$.  If the number of components is $1$ the result is immediate.

For the inductive step let $P_2$ be the tree of rational curves obtained by contracting $P_{i_\infty }$ to a point $s_2\in P_2$.  Let $x_\infty $ be the node connecting $P_{i_\infty }$ to the rest of $P$, and let $\alpha _{i_{\infty }+}$ (resp. $\alpha _{i_\infty -}$) be the coefficient in $u$ of $t_{i_\infty +}$ (resp. $i_{\infty -}$).  Then we must have $\alpha _w +\alpha _{i_\infty -}-\alpha _{i_\infty +} = 0$ since $u$ defines the trivial line bundle on $P_{i_\infty }$.  Equivalently the coefficients of $u$ at $x_{i_\infty }$ is given by
$$
\alpha _wt_{i_\infty -}+ \alpha _{i_\infty +}(t_{i_\infty -}+t_{i_\infty +}).
$$
Now consider the element over $P_2$ given by the coefficients of $u$ at nodes not equal to $x_{i_\infty }$ plus $\alpha _w$ times $s_2$. By induction we then see that this element over $P_2$ can be written uniquely as $\alpha _w$ times the $P_2$-version of $w'$ plus a unique sum of linear combinations of $t_{j+}+t_{j-}$.  From this (ii) and (iii) follow.
\end{proof}

\begin{proof}[Proof of \ref{prop:contracting}]
The proposition now follows almost immediately from \ref{L:5.10c}.

Indeed giving an inclusion $M_{\bar y}^{s_i'}\hookrightarrow (q_*M_C^i)_{\bar y}$ is equivalent to giving an element of $(q_*M_C^N)_{\bar y} = E$ whose coefficient of $t_{0+}$ is $1$ and whose coefficient of $t_{0-}$ is $0$ as follows from considering the local description of the node $x_0$.  By parts (i) and (ii) of the lemma there exists a unique such element; namely, $w'$.  Furthermore that the induced map \eqref{prop:contracting} is an isomorphism follows from part (iii) of the lemma.
\end{proof}

\begin{cor}\label{cor:i-enhancement}
Let $q: (C/S, \{s_i\}_{i=1}^n) \to (C'/S, \{s_i'\}_{i=1}^n)$ be a contraction of $n$-marked prestable curves. Then the map $q$ extends uniquely to a commutative square of log schemes
\begin{equation}\label{eq:i-enhancement}
\xymatrix{
(C, M_C^i)\ar[r]\ar[d]& (C', M^i_{C'})\ar[d]\\
(S, \MS{S}{C})\ar[r]& (S, \MS{S}{C'}).}
\end{equation}
\end{cor}
\begin{proof}
By \cite[B.6]{MR3329675} the map $q$ extends uniquely to a commutative square of log schemes
\begin{equation}\label{E:C.3.1}
\xymatrix{
(C, M_C^{\node})\ar[d]\ar[r]& (C', M_{C'}^{\node})\ar[d]\\
(S, \MS {S}{C})\ar[r]& (S, \MS {S}{C'}),}
\end{equation}
where the bottom horizontal map is the identity on $S$.
Indeed it is shown in loc. cit. that $q_*M_C^{\node}$ is a log structure on $C'$ making $(C', q_*M_C^{\node})$ log smooth over $(S, \MS {S}{C})$, and then the existence and uniqueness of the diagram \eqref{E:C.3.1} follows from \cite[1.2]{OlssonTohoku}. From here the existence and uniqueness of \eqref{eq:i-enhancement} follows from \ref{prop:contracting}.
\end{proof}

\begin{cor}\label{cor:log-contractions}
Let $q: (C/S, \{s_i\}_{i=1}^n) \to (C'/S, \{s_i'\}_{i=1}^n)$ be a contraction of $n$-marked prestable curves. Then the map $q$ has a unique extension to a commutative square of log schemes as in \eqref{eq:enhancement} such that the extension restricts to a commutative square as in \eqref{eq:i-enhancement} for each $i=1, \ldots, n$.
\end{cor}
\begin{proof}
Existence of \eqref{eq:enhancement} comes from the diagrams \eqref{eq:i-enhancement}, the decomposition \eqref{eq:Mis}, and the fact that for any monoids $M, L, P$ on $C$ we have a natural morphism $q_*M \oplus_{q_*L} q_*N \to q_*(M \oplus_L N)$. Uniqueness comes from the uniqueness in \ref{cor:i-enhancement}.
\end{proof}

\subsection{All-at-once description of log contractions}
\label{S:contract-monoid}

We described above explicitly what log contractions look like when only one component of $C$ is collapsed at a time, with at most one marking. In this section we give a description for arbitrary contractions of $n$-marked curves. 
Let $S = \Spec(k)$ be the spectrum of a separably closed field and let $q: (C/S, \{s_i\}_{i=1}^n) \to (C'/S, \{s_i'\}_{i=1}^n)$ be a contraction. 

% It is a special property of contractions that the pushforward by $q$ of a Deligne-Faltings structure on $C$ is again a Deligne-Faltings structure. Indeed, if $M$ is a log structure on $C$ and $\overline{M}$ is the associated Deligne-Faltings structure, the proof of \ref{} shows that the sheaf $q_*\overline{M}$
% Since $q: C \to C'$ induces an isomorphism of $\mls O^*_C$ and $\mls O^*_{C'}$, if $M$ is any log structure on $C$ the pushforward $q_*\overline{M}$ of the Deligne-Faltings structure is a Deligne-Faltings structure on $C'$ equal to a subobject of the Deligne-Faltings structure $\overline{q_* M_C}$. \Rachel{ask martin about this}
% We begin by defining morphisms of Deligne-Faltings structures as follows. 

Let $J(C)$ (resp. $J(C')$) denote the set of nodes of $C$ (resp. of $C'$) and let $\{e_j\}_{j\in J(C)}$ (resp. $\{e_j\}_{j'\in J(C')}$) denote the set of irreducible elements in $\mathbf{N}^{J(C)}$ (resp. $\mathbf{N}^{J(C')}$). If $x \in C'$ is a closed point let $C^{\prime, sh}_x$ denote $\Sp(\mls O^{sh}_{C', x})$, and let $A$ (resp. $A, B$) denote the irreducible component(s) of $C^{\prime, sh}_x$ if $x$ is smooth (resp. a node). Let $\Gamma_x$ be the graph dual to $C \times_{C'} C^{\prime, sh}_x$. In particular,
\begin{itemize}
\item Edges $E(\Gamma_x)$ are in bijection with nodes of $C$ mapping to $x$.
\item Half edges of $\Gamma_x$ are in bijection with markings $s_i$ mapping to $x$.
\item Vertices $V(\Gamma_x)$ of $\Gamma_x$ are in bijection with irreducible components of $C$ whose images in $C'$ meet $x$. If $x \in C'$ is smooth (resp. a node) there is a distinguished vertex $A \in V(\Gamma_x)$ (resp. two distinguished vertices $A, B \in V(\Gamma_x)$) mapping to the component $A$ (resp. the components $A, B$) of $C^{\prime, sh}_x$.
\end{itemize}
If $x=j' \in J(C')$ is a node, we let $E^*(\Gamma_{j'}) \subset E(\Gamma_{j'})$ be the subset of edges separating vertex A from vertex B.

Since $M_C$ is equal to $M^{\node}_C \oplus_{\mls O_{C}^*} M_C^{\underline{s}},$ the morphism of log structures $M_{C'} \to q_*M_C$ is determined by two morphisms $\phi^{\node}: M^{\node}_{C'} \to q_*M_C$ and $\phi^{\underline{s}}: M^{\underline{s}}_{C'} \to q_*M_C$.  We let $\phi: \MS{S}{C'} \to \MS{S}{C}$ denote the morphism of log structures on the base. We now describe $\phi$, $\phi^{\node}$, and $\phi^{\underline{s}}$ by calculating the induced morphisms of Deligne-Faltings structures.
\begin{itemize}
\item $\overline \phi: \bMS{S}{C'} \to \bMS{S}{C}$ is given by
\begin{align*}
\mathbf{N}^{J(C)} &\leftarrow \mathbf{N}^{J(C')}\\
\sum_{j \in E^*(\Gamma_{j'})} e_j &\mapsfrom e_{j'}.
\end{align*}
\item $\overline \phi^{\node}: \overline{M}^{\node}_{C'} \to \overline{q_*{M}_C}$ factors through the inclusion $\overline{q_*M^{\node}_C} \subset \overline{q_*M_C}$. The composition
\begin{equation}
\label{eq:node-computable}
\overline{M}^{\node}_{C'} \to \overline{q_*M^{\node}_C} \to q_*\overline{M}^{\node}_C
\end{equation}
of the factoring morphism and the inclusion in \ref{L:contracting} can be described explicitly as follows. It is identified with $\overline{\phi}$ away from the nodes of $C'$. 
%If $x=j' \in J(C')$ is a node let $E^*(\Gamma_x)\subset E(\Gamma_x)$ denote the subset of edges separating vertex $A$ from vertex $B$. 
Note that if $x=j' \in J(C')$ is a node, for edge $\nu \in E^*(\Gamma_{x})$ there is a generator $e_\nu^A$ of $q_*\overline{M}^{\node}_C$ on $C^{\prime, sh}_x$ supported on the component containing vertex $A$ and a generator $e_\nu^B$ supported on the component containing vertex $B$. Let $e_A$ and $e_B$ be the generators of $\overline{M}^{\node}_{C'}$ on $C^{\prime, sh}_x$ supported on the components $A$ and $B$, respectively, and let $J(C) \setminus E(\Gamma_x)$ be the set of nodes of $C$ not mapping to $x$. Then \ref{eq:node-computable} has a global chart on $C^{\prime, sh}_x$ given by 
\[
(\mathbf{N} \oplus \mathbf{N})^{E(\Gamma_x)} \oplus \mathbf{N}^{J(C) \setminus E(\Gamma_x)} \leftarrow  \mathbf{N}^2 \oplus \mathbf{N}^{J(C') \setminus \{j'\}}
\]
\[
\sum_{\nu \in E^*(\Gamma_x)} e^A_\nu \mapsfrom e_A, \quad \quad \quad \sum_{\nu \in E^*(\Gamma_x)} e^B_\nu \mapsfrom e_B, \quad \quad \quad \overline{\phi}(e_{i'})\mapsfrom e_{i'} \; \text{for}\; i' \in J(C') \setminus \{j'\}.
\]
% \item $\overline \phi_1: \overline{M}^{\node}_{C'} \to q_*\overline{M}^{\node}_C$ is the identity away from nodes of $C'$. If $x=j' \in J(C')$ is a node let $E^*(\Gamma_x)\subset E(\Gamma_x)$ denote the subset of edges separating vertex $A$ from vertex $B$. Note that for each such edge $\nu$, there is a generator $e_\nu^A$ of $q_*\overline{M}^{\node}_C$ on $C^{\prime, sh}_x$ supported on the component containing vertex $A$ and a generator $e_\nu^B$ supported on the component containing vertex $B$. Let $e_A$ and $e_B$ be the generators of $\overline{M}^{\node}_{C'}$ on $C^{\prime, sh}_x$ supported on the components $A$ and $B$, respectively. Then $\overline \phi_1$ has a global chart on $C^{\prime, sh}_x$ given by 
% \[
% (\mathbf{N} \oplus \mathbf{N})^{E(\Gamma_x)} \oplus \mathbf{N}^{J(C') \setminus \{j'\}} \leftarrow  \mathbf{N}^2 \oplus \mathbf{N}^{J(C') \setminus \{j'\}}
% \]
% \[
% \sum_{\nu \in E^*(\Gamma_x)} e^A_\nu \mapsfrom e_A, \quad \quad \quad \sum_{\nu \in E^*(\Gamma_x)} e^B_\nu \mapsfrom e_B, \quad \quad \quad e_{i'} \mapsfrom e_{i'} \; \text{for}\; i' \in J(C') \setminus \{j'\}.
% \]
\item $\overline \phi^{\underline{s}}: \overline{M}^{\underline s}_{C'} \to \overline{q_*M_C}$ followed by the inclusion $\overline{q_*M_C} \to q_*\overline{M}_C$ can be described as follows. It is zero away from marked points of $C'$. If $x \in C'$ is a marked point, let $I(x) \subset \{1, \ldots, n\}$ denote the set indices such that $s_i'=x$, and for $i \in I(x)$ let $E^i(\Gamma_x) \subset E(\Gamma_x)$ denote the subset of edges separating vertex $A$ from the half edge corresponding to $s_i$. For each such edge $\nu$, there is a generator $e^P_\nu$ of $q_*\overline{M}^{\node}_C$ on $C^{\prime, sh}_x$ supported on the component containing the half edge corresponding to $s_i$.
%(There is also a generator supported on the component containing vertex $A$, but we will not need notation for this generator.) 
Let $\{e_i\}_{i \in I(x)}$ denote the irreducible elements of $\mathbf{N}^{I(x)}$ and let $J(C) \setminus E(\Gamma_x)$ be the set of nodes in $C$ not mapping to $x$. Then $\overline{M}^{\underline s}_{C'} \to q_*\overline{M}_C$ has a global chart on $C^{\prime, sh}_x$ given by
\[
(\mathbf{N} \oplus \mathbf{N})^{E(\Gamma_x)} \oplus \mathbf{N}^{J(C) \setminus E(\Gamma_x) } \oplus \mathbf{N}^{I(x)} \leftarrow   \mathbf{N}^{I(x)}
\]
\[
\sum_{\nu \in E^i(\Gamma_x)} e^P_\nu  + e_i  \mapsfrom e_i\;\text{for}\;i \in I(x) 
%\quad \quad \quad \quad \quad \quad 
%e_{j'} \mapsfrom e_{j'} \;\text{for}\; j' \in J(C').
\]
\end{itemize}

These descriptions follow from the two examples of contractions in considered previously, the fact that $q$ can be factored as a sequence of contractions of rational bridges and rational tails \ref{C:sequences}, and the uniqueness of contractions in \ref{cor:log-contractions}.

\begin{rem}
We point out that the morphism of log structures $M^{\node}_{C'} \to q_*M_C$ factors through $q_*M^{\node}_C$, but $M^{\underline{s}}_{C'} \to q_*M_C$ does not factor through $q_*M^{\underline{s}}_C$.
\end{rem}

\section{Contractions of generalized log twisted curves}\label{sec:glt-contractions}

In this section (in fact in \ref{P:3.11}) we define contractions of generalized log twisted curves. The definition is chosen so that such a contraction induces a morphism of associated stacks, and we explain this in \ref{S:contract-stack}. In \ref{S:contract-curve} we explain how a contraction of coarse curves lifts to a canonical \textit{initial} contraction of generalized log twisted curves (initial contractions are defined in \ref{def:initial}).
Our discussion throughout this section uses the canonical log structures on the coarse space of a generalized log twisted curve and on the associated stack defined in \ref{SS:2.19} and \ref{SS:log-stack}, respectively.
%The functor that sends a generalized log twisted curve to its associated stack is not full: it is possible for non-isomorphic generalized log twisted curves to have isomorphic associated stacks (see Example \ref{ex:two-to-one}). In this section we discuss a class of morphisms of stacks associated to generalized log twisted curves that do arise from the data specifying generalized log twisted curves.  

% \begin{example}\label{ex:two-to-one}
% Let $C = \mathbf{P}^1_k$ and let $s_1=s_2$ be the same marking in $C$. Consider the admissible sheaves $\mls N_{i}$ defined by submonoids $N_i\subset \mathbf{Q}^2$ defined as follows:
% $$
% N_1 = \langle (\frac{1}{2}, 0), (0, 1)\rangle , \ \ N_2 = \langle (1, 0), (0, \frac{1}{2})\rangle .
% $$
% Then we have two non-isomorphic generalized log twisted curves 
% \[
% \bC_i = (\mathbf{P}^1_k, \{s_1, s_2\}, M_k \xrightarrow{\sim} M_k, \mls N_i) \quad \quad \text{for}\; i=1, 2
% \] but it follows from \eqref{eq:local description2} that the associated stacks are both isomorphic to the root stack of $C$ at $s_1=s_2$ with $\mu_2$ stabilizer.
% \end{example}
% The material in this section is \Rachel{currently} only used to prove the valuative criterion for separatedness in \Rachel{ref}.

%\subsection*{Contractions of generalized log twisted curves}

\subsubsection{Contractions of generalized log twisted curves} \label{P:3.11}
%Let $\bC = (C/S, \{s_i\}_{i=1}^n, \ell: \MS{S}{C} \hookrightarrow \MS {S}{C}', \mls N)$ be a generalized log twisted curve and let $\mls C$ be the associated stack. Let $\mls C^{\node}$ be the stack associated to the simple inclusion $\ell$, so there is a morphism $\mls C \to \mls C^{\node}$ given by rigidification along the marked points. 

% Recall from \Rachel{ref} that there is a universal log structure $M_{\mls C^{nd}}$ on $\mls C^{nd}$. Likewise there is a universal log structure $M_{\mls C^{\mls N}}$ on $\mls C^{\mls N}$ lifting the universal Deligne-Faltings structure associated to $\mls N$ (see \ref{SS:3.4}). \Rachel{I think I'm not explaining this very well.} Hence there is a canonical log structure $M_{\mls C}$ on $\mls C$ given by
% \[
% M_{\mls C} := M^{nd}_{\mls C} \oplus_{\mls O_\mls C^*}  M^{\mls N}_{\mls C}
% \]
% where $M^{nd}_{\mls C} := M_{\mls C^{nd}}|_{\mls C}$ and where $M^{\mls N}_{\mls C} := M_{\mls C^{\mls N}}|_{\mls C}.$ The coarse space morphism for $\mls C$ is a morphism of log stacks $(\mls C, M_{\mls C}) \to (C, M_C)$ and the associated morphism of log structures sends $M^{nd}_C|_{\mls C}$ into $M^{nd}_{\mls C}$ via a simple inclusion and it sends $M_C^{\bs}|_{\mls C}$ into $M^{\mls N}_{\mls C}.$

Suppose we have two generalized log twisted curves
\begin{equation}\label{eq:CD}
\begin{gathered}
\bC = (C/S, \{s_i^C\}_{i=1}^n, \ell^C: \MS{S}{C} \hookrightarrow \MS{S}{C}', \mls N^C)\\
\bD = (D/S, \{s_i^D\}_{i=1}^n, \ell^D: \MS {S}{D} \hookrightarrow \MS{S}{D}', \mls N^D)
\end{gathered}
\end{equation}
and let $q: C \to D$ be a contraction of coarse curves (see \ref{def:contraction}). By \ref{cor:log-contractions} the map $q$ extends to a diagram of log schemes
$$
\xymatrix{
(C, M_C)\ar[r]^-{(q, q^b)}\ar[d]& (D, M_D)\ar[d]\\
(S, \MS {S}{C})\ar[r]^-{(\text{id}, \phi )}& (S, \MS {S}{D}).}
$$
%where $\phi$ is defined as in \ref{S:contract-monoid}.

\begin{defn}\label{D:6.2}
Let $\bC$ and $\bD$ be generalized log twisted curves as in \eqref{eq:CD}. A \emph{contraction} $\bC \to \bD$ is a contraction $q:C \to D$ of the underlying coarse curves together with a morphism of log structures $\phi ':\MS {S}{D}'\rightarrow \MS{S}{C}'$ such that the diagram 
    \begin{equation}\label{E:3.20.1}
    \xymatrix{
    \MS{S}{D}\ar[r]^-{\ell _D}\ar[d]^-\phi & \MS {S}{D}'\ar[d]^-{\phi '}\\
    \MS {S}{C}\ar[r]^-{\ell _C}& \MS {S}{C}'}
    \end{equation}
     commutes, and such that the dashed arrow in the following diagram exists (note that if it exists it is unique): 
%\Rachel{change  notation in this diagram}
\begin{equation}\label{eq:contract2}
\begin{tikzcd}
\overline{M}^{\underline{s}}_{D}|_{\mls C} \arrow[r, hookrightarrow] \arrow[d, hookrightarrow] & \mls N^{D}|_{\mls C} \arrow[r, hookrightarrow]  \arrow[d, dashrightarrow] & (\overline{M}^{\underline{s}}_{D}|_{\mls C})_{\mathbf Q} \arrow[d, hookrightarrow ]\\
\overline{M}_{C}|_{\mls C} \arrow[r, hookrightarrow] & \overline{M}_{\mls C} \arrow[r, hookrightarrow] & (\overline{M}_{C}|_{\mls C})|_{\mathbf Q}.
\end{tikzcd}
\end{equation}
%\Rachel{We mean restrictions of characteristic sheaves here, not characteristic sheaves of restrictions}
\end{defn}

% As motivation for this definition, we refer the reader to \ref{S:contract-stack}: the upshot is that diagrams \eqref{E:3.20.1} and \eqref{eq:contract2} are precisely what is needed for the contraction $q: C \to D$ to lift to a morphism $\mls C \to \mls D$ of stacks.

\begin{rem}\label{R:3.20}
If the contraction $q: C \to D$ is  an isomorphism then $\overline M_D^{\underline s}$ maps isomorphically to $\overline M_C^{\underline s}$ and the dashed arrow in \eqref{eq:contract2} exists if and only if $\mls N^{C} $ is contained in $\mls N^{D}$.
\end{rem}

\begin{rem}
The morphism of log structures $\phi': \MS{S}{D}' \rightarrow \MS{S}{C}'$ may not always exist and when it does is not unique.  For example, suppose $S = \Sp (A)$ is the spectrum of a strictly henselian local ring, $C=D$,   and  the diagram 
$$
 \xymatrix{
    \MS{S}{D}\ar[r]^-{\ell _D}\ar[d]^-\phi & \MS {S}{D}'\\
    \MS {S}{C}\ar[r]^-{\ell _C}& \MS {S}{C}'}
$$
is induced by a diagram of charts (for simplicity assume we are in the case of a single node)
$$
\xymatrix{
\mathbf{N}\ar[r]^-{n_D}\ar[d]^-{\simeq }& \mathbf{N}\ar@/^2pc/[rd]& \\
\mathbf{N}\ar[r]^-{n_C}& \mathbf{N}\ar[r]^-{\beta }& A}
$$
with $\alpha (1), \beta (1)\in A$ nonunits.
Then for the existence of $\phi '$ we must have $n_D|n_C$, and in this case the choice of $\phi '$ is determined by a map $\mathbf{N}\rightarrow A^*\oplus \mathbf{N}$ sending $1$ to $(u, n_C/n_D)$ for a unit $u$ satisfying 
$$
u\beta  (1)^{n_C/n_D} = \alpha (1), \ \ u^{n_D} = 1.
$$
In general there can be several solutions to these equations.  This can be viewed as a manifestation of the fact that 
 balanced twisted curves can have automorphisms acting trivially on the coarse space that come from stacky nodes \cite[7.1.1]{MR2007376}. 
\end{rem}

\subsubsection{Contractions of associated stacks}\label{S:contract-stack}
Let $\bC$ and $\bD$ be generalized log twisted curves as in \eqref{eq:CD}, let $\mls C$ and $\mls D$ be the associated stacks, and let $\bC \to \bD$ be a contraction. We will construct a morphism $\mls C\to \mls D$ determined by the contraction that lifts the coarse map $C \to D$ coming from the contraction. To begin, let $\mls C^{\node}$ and $\mls D^{\node}$ be the stacks associated to the simple inclusions $\ell_C$ and $\ell_D$, respectively.

\begin{prop}\label{P:induce-stack}
Let $\bC$, $\bD$, and $q$ be as in \ref{P:3.11}. A morphism of log structures $\phi ':\MS {S}{D}'\rightarrow \MS{S}{C}'$ such that the diagram \eqref{E:3.20.1} commutes induces a morphism of stacks $\mls C^{\node }\rightarrow \mls D^{\node}$ over $q$. 
\end{prop}

\begin{proof}
% \begin{pg}\label{P:3.20b} \Rachel{would you have any complaints if I moved 3.11 and 3.12 into \ref{S:contract-stack}?}
%     A morphism of log structures $\phi ':\MS {S}{D}'\rightarrow \MS{S}{C}'$ such that the diagram
%     \begin{equation}\label{E:3.20.1}
%     \xymatrix{
%     \MS{S}{D}\ar[r]^-{\ell _D}\ar[d]^-\phi & \MS {S}{D}'\ar[d]^-{\phi '}\\
%     \MS {S}{C}\ar[r]^-{\ell _C}& \MS {S}{C}'}
%     \end{equation}
%     commutes induces a morphism of stacks $\mls C^{\node}\rightarrow \mls D^{\node}$ over $q$. 
%    \Rachel{I think the maps of stalks of characteristic monoids inuduced by $\phi'$ are injective and unique if they exist. Does it follow that $\phi'$ is injective and unique if it exists? If not, what is an example? If so, we should reword Def. 3.13 so that $\phi'$ is not an additional datum.}
    %\Rachel{Is there some reason why $\phi'$ will ``look like'' $\phi$---for example, will it be injective?} 
   
     Recall from \cite[4.1]{LogTwisted} that by the moduli description of the stack $\mls D^\node $ to define a morphism $\mls C^{\node} \to \mls D^{\node}$ it suffices (in fact it is equivalent) to give a diagram of log structures on $\mls C^{\node}$
    \begin{equation}\label{eq:induce-stack1}
    \begin{tikzcd}
\MS {S}{D}|_{\mls C^{\node}} \arrow[r, "\ell_D"] \arrow[d] & \MS{S}{D}'|_{\mls C^{\node}} \arrow[d] \\
M_D^{\node}|_{\mls C^{\node}} \arrow[r] & \mls M
    \end{tikzcd}
    % \xymatrix{
    % \MS {S}{D}|_{\mls C^{\node}}\ar[r, "\ell_D"]\ar[d]& \MS{S}{D}'|_{\mls C^{\node}}\ar[d]\\
    % M_D^\node |_{\mls C^{\node}}\ar[r]& \mls M,}
    \end{equation}
    such that for every geometric point $\bar x \to \mls C^{\node}$, the induced diagram of stalks of characteristic monoids has the form
    \begin{equation}\label{eq:induce-stack2}
     \begin{tikzcd}
\mathbf{N}^n \arrow[r, "\cdot \bd"] \arrow[d, equal] &\mathbf{N}^n \arrow[d, equal]\\
\mathbf{N}^{n} \arrow[r, "\cdot {\bd}"] & \mathbf{N}^{n}
    \end{tikzcd}
    \quad \quad \text{or} \quad \quad
    \begin{tikzcd}
\mathbf{N}^n \arrow[r, "\cdot \bd"] \arrow[d, "\Delta_i"'] &\mathbf{N}^n \arrow[d, "\Delta_i"]\\
\mathbf{N}^{n+1} \arrow[r, "\cdot {(\bd, \,d_i)}"] & \mathbf{N}^{n+1}
    \end{tikzcd}
    \end{equation}
where $\bd = (d_1, \ldots, d_n)$ is a vector of positive integers and $\Delta_i$ sends $(x_1, \ldots, x_n)$ to $(x_1, \ldots, x_n, x_i)$. 
There is a universal diagram analogous to \eqref{eq:induce-stack1} associated to the identity map $\mls C^{\node} \to \mls C^{\node}$, with $M^{\node}_{\mls C}$ in the place of $\mls M$. Combining this universal diagram with \eqref{eq:induce-stack1}, we get the following commuting diagram of injective arrows, where the black part of the diagram is already known to exist:
\begin{equation}\label{eq:induce-stack3}
\begin{tikzcd}
\MS{S}{D}|_{\mls C^{\node}} \arrow[r, "\ell_D"] \arrow[d] & \MS{S}{D}'|_{\mls C^{\node}} \arrow[dr, "\phi'"] \arrow[d, gray!60]\\
M_D^{\node}|_{\mls C^{\node}} \arrow[dr, "q^\flat"'] \arrow[r, gray!60] & \color{gray!60} \mls M \color{black}& \MS{S}{C}'|_{\mls C^{\node}} \arrow[d]\\
&M^{\node}_C|_{\mls C^{\node}} \arrow[r, "k"] & M^{\node}_{\mls C}.
\end{tikzcd}
\end{equation}
The following lemma then defines the morphism $\mls C^{\node} \to \mls D^{\node}$.

\begin{lem} \label{L:3.12} There exists  a unique sublog structure $\mls M\subset M_{\mls C}^\node $ containing the images of $\MS {S}{D}'|_{\mls C}$ and $M_D^\node |_{\mls C}$ such that the diagram \eqref{eq:induce-stack1} defines a morphism $\mls C^{\node} \to \mls D^{\node}.$
%the induced map $M_D^\node|_{\mls C}\rightarrow \mls M$ is a simple inclusion of log structures. \Rachel{and the induced map $\MS{S}{D}'|_{\mls C} \to \mls M$ has a certain form?}
\end{lem}
\begin{proof}

A subsheaf of $M^{\node}_{\mls C}$ is determined by its stalks, and by \cite[3.15]{BV} the stalk of a characteristic monoid is a local chart for the Deligne-Faltings structure.
%Since our lemma includes a uniqueness statement, it follows that it is enough to consider the stalks of characteristic monoids at every $\bar x \to \mls C^{\node}$. 
%More precisely, f
For $\bar x \to \mls C^{\node}$ let $\mls M_{\bar x}$ be the abstract monoid isomorphic to $(M_{D}^{\node}|_{\mls C^{\node}})_{\bar x}$, and consider the diagram of stalks of characteristic monoids induced by \ref{eq:induce-stack3}, where we insert $\mls M_{\bar x}$ into the appropriate place in the diagram, with morphisms as specified by \ref{eq:induce-stack2}. It is enough to show that there is a unique inclusion $\iota: \mls M_{\bar x} \to (M_{\mls C}^{\node})_{\bar x}$ making the diagram commute. 

For this, let $\bar s \to S$ be the image of $\bar x$ in $S$. The stalks of $\ell_D$ and $\ell_C$ at $\bar s$ are given by
\[
\mathbf{N}^{J(D_{\bar s})} \xrightarrow{\cdot \bd}\mathbf{N}^{J(D_{\bar s})} \quad \quad \quad \quad \mathbf{N}^{J(C_{\bar s})} \xrightarrow{\cdot \bc}\mathbf{N}^{J(C_{\bar s})}
\]
for some positive integers $\bd = (d_j)_{j \in J(D_{\bar s})}$ and $\bc = (c_{j'})_{j' \in J(C_{\bar s})}$, where $J$ denotes the set of nodes. From commutativity of \eqref{E:3.20.1} it follows that $\overline{\phi}_{\bar x}$ is the unique morphism satisfying
\begin{equation}\label{eq:phibar}
d_j\overline{\phi}(e_j) = \bc \cdot \phi(e_j)
\end{equation}
where $e_j \in (\MS{S}{D})_{\bar s}$ is the unique irreducible element corresponding to node $j \in J(D_{\bar s})$.
In particular $d_j | c_{j'}$ for all nodes $j'$ in the preimage of $j$.

Now for $e \in \mls M_{\bar x}$ an irreducible element associated to a node $j \in J(D)$\footnote{
Usually this irreducible element is unique, but if $q(\bar x)$ is itself a node, there will be two irreducible elements of $\mls M_{\bar x}$ associated to $\bar x$.} we define $\iota(e)$ to be the unique element of $M_{\mls C}^{\node}$ satisfying
\[
d_j \iota(e) = k\circ q^\flat(e).
\]
Such an element exists (and is necessarily unique) because $d_j | c_{j'}$ for all nodes $j'$ in the preimage of $j$---the map $k$ will be multiplication by $\bc$, with some coordinates repeated. It follows from \eqref{eq:phibar} and commutativity of the black diagram \eqref{eq:induce-stack3} that $\iota$ also makes ``the other square'' in the diagram commute.
\end{proof}

This completes the proof of \ref{P:induce-stack}.
\end{proof}

\begin{lem}\label{L:induce-stack}
Let $\mathbf{C}$, $\mathbf{D}$, and $q$ be as in \ref{P:3.11}. A dashed arrow making \eqref{eq:contract2} commute induces a morphism of stacks $\mls C \to \mls D^{\mls N^D}$.
\end{lem}
\begin{proof}
To define the map $\mls C \to \mls D^{\mls N^D}$ it suffices to define an inclusion of log structures $M_D^{\bs }|_{\mls C}\hookrightarrow R$ such that the induced inclusion $\overline M_D^{\bs }|_{\mls C}\hookrightarrow \overline R$ identifies with the inclusion $\overline M_D^{\bs }|_{\mls C}\hookrightarrow \mls N^D|_{\mls C}$.  For this we take $R$ to be the fiber product $M_{\mls C}\times _{\overline M_{\mls C}}\mls N^D|_{\mls C}$, where the map $\mls N^D|_{\mls C}\rightarrow \overline M_{\mls C}$ is the one provided by the dashed arrow in \eqref{eq:contract2}.
\end{proof}

Recall that $\mls D\simeq \mls D^{\node}\times _D\mls D^{\mls N^D}$. Hence, to define a morphism $\mls C \to \mls D$ it suffices to define morphisms $\mls C \to \mls D^{\node}$ and $\mls C \to \mls D^{\mls N^D}$. 

\begin{defn}\label{D:6.8}
Let $\mathbf{C} \to \mathbf{D}$ be a contraction of generalized log twisted curves. The associated morphism of stacks $\mls C \to \mls D$ is the one induced by 
the morphism $\mls C\rightarrow \mls D^{\node}$ equal to the projection $\mls C\rightarrow \mls C^{\node}$ followed by the morphism from \ref{P:induce-stack}, and by the morphism $\mls C \to \mls D^{\mls N^D}$ from \ref{L:induce-stack}.
\end{defn}

\begin{rem}\label{R:3.24} Note that the map of stacks $\mls C\rightarrow \mls D$ is uniquely determined by the fact that it extends to a morphism of log stacks $(\mls C, M_{\mls C})\rightarrow (\mls D, M_{\mls D})$ filling in the diagram
$$
\xymatrix{
(\mls C, M_{\mls C})\ar[dd]\ar@{-->}[rr]\ar[rd]&& (\mls D, M_{\mls D})\ar[dd]\ar[rd]& \\
& (C, M_C)\ar[rr]\ar[dd]&& (D, M_D)\ar[dd]\\
(S, \MS {S}{C}')\ar[rr]\ar[rd]&& (S, \MS {S}{D}')\ar[rd]& \\
& (S, \MS {S}{C})\ar[rr]&& (S, \MS {S}{D}).}
$$
\end{rem}

% \begin{rem}
% If $\bC \to \bD$ and $\bD \to \bE$ are contractions of generalized log twisted curves, then it makes sense to speak of their composition $\bC \to \bE$: It is immediate from the definition that the composition of contractions of generalized log twisted curves is a contraction. \Rachel{right?}
% \end{rem}

\begin{lem}
If $(q, \phi '):\bC \to \bD$ and $(r, \psi '):\bD \to \bE$ are contractions of generalized log twisted curves, then $(r\circ q, \phi '\circ \psi '):\bC \to \bE$ is also a contraction, and the associated morphism $\mls C \to \mls E$ is the composition of the associated morphisms $\mls C \to \mls D \to \mls  E$.
\end{lem}
\begin{proof}
It is clear that if the dashed arrow in \eqref{eq:contract2} exists for both $\bC \to \bD$ and $\bD \to \bE$, then it exists for $\bC \to \bE$ so $(r\circ q, \phi '\circ \psi ')$ is a contraction. The statement that the two induced maps $\mls C\rightarrow \mls E$ agree follows from the fact that they both extend to morphisms of log stacks and \ref{R:3.24}.
\end{proof}

\begin{example}\label{ex:contract}
Let $k$ be a field, let $C = \mathbf{P}^1_k$ and let $s \in C$ be a marked point. Consider the admissible sheaves $\mls N_{i}$ defined by submonoids $N_1 = (1/a) \mathbf{N}$ and $N_2 = (1/b) \mathbf{N}$. We have two generalized log twisted curves 
\[
\bC_i := (\mathbf{P}^1_k, \{s\}, M_k \xrightarrow{\sim} M_k, \mls N_i) \quad \quad \text{for}\; i=1, 2.
\] 
In this case a contraction $\bC_1 \to \bC_2$ lifting the identity map $C \to C$ is unique if it exists, and referring to \eqref{eq:contract2} we see that it exists exactly when $(1/b)\mathbf{N} \subseteq (1/a)\mathbf{N}$, i.e. exactly when $a = qb$ for some integer $q$. In this case the stacks $\mls C_1$ and $\mls C_2$ are the root stacks of $C$ at $s$ to orders $a$ and $b$, respectively. So if $U \to C$ is a morphism from a scheme $U$ then $\mls C_1(U)$ is the groupoid of triples $(\mls L, \sigma, \phi: \mls L^{\otimes a} \to \mls O(1)|_U)$ where $\phi$ is an isomorphism sending $\sigma^a$ to the pullback of the section defining $s \in C$. The morphism $\mls C_1 \to \mls C_2$ induced by the contraction is the one that sends a triple $(\mls L, \sigma, \phi: \mls L^{\otimes a} \to \mls O(1))$ to $(\mls L^{\otimes q}, \sigma^q, \phi: (\mls L^{\otimes q})^{\otimes b} \to \mls O(1)).$
\end{example}

\subsubsection{Relative coarse moduli}
Let $\mls C$ be the stack associated to a generalized log twisted curve $\bC = (C/S, \{s_i\}_{i=1}^n, M_S \hookrightarrow M_S', \mls N)$ and let 
\begin{equation}\label{eq:rel coarse}f: \mls C \to \mc X
\end{equation}
be a locally finitely presented morphism of algebraic stacks. The relative moduli space of $f$ was constructed in \cite[Thm~3.1]{AOV}. Example \ref{ex:contract} motivates the following proposition.

\begin{prop}\label{lem:rel coarse}
Let $\mls C \to \mls C^\dagger \to \cX$ be the factorization of \eqref{eq:rel coarse} through its relative moduli space. Up to unique isomorphism, the stack $\mls C^{\dagger}$ is the stack associated to a generalized log twisted curve $\bC^\dag$ with the same underlying marked prestable curve as $\mls C$, and the morphism $\mls C \to \mls C^\dagger$  arises from a contraction $\bC \to \bC^\dag$.

Conversely, if $\bC \to \bD$ is a contraction whose underlying morphism of marked prestable curves is an isomorphism, and if $\mls D \to \cX$ is a representable morphism, then $\mls D$ is the relative coarse space of the composition $\mls C \to \mls D \to \cX$.
\end{prop}

\begin{proof}
Write $\mls C = \mls C^{\node} \times_C \mls C^{\mls N}$. In fact, $\mls C$ has a Zariski open cover $\mls C = \mls U^{\node}\cup \mls U^{\mls N}$, where $\mls U^{\node}$ is the complement  of the markings in $\mls C^{\node}$ and $\mls U^{\mls N}$ is the complement of the nodes in $\mls C^{\mls N}$.  These two open substacks have schematic intersection $U\subset \mls C$ and the stack $\mls C^\dagger$ is obtained by gluing the relative coarse spaces of $\mls U^{\node}\rightarrow \mc X$ and $\mls U^{\mls N}\rightarrow \mc X$ along $U$.
Since the non-representable locus of $\mls C^{\node}$ (resp. $\mls C^{\mls N}$) is contained in $\mls U^{\node},$ (rep. $\mls U^{\mls N}$), it is enough to prove the lemma in the two special cases $\mls C = \mls C^{\node}$ (i.e., $\oplus s_{i, *} \mathbf{N} = \mls N$) and $\mls C = \mls C^{\mls N}$ (i.e., $M_S = M_S'$).\\

 \noindent 
\textit{The case $\mls C = \mls C^{\mls N}$.}
 By Remark \ref{R:3.20} we need to construct an admissible subsheaf $\mls N^\dag\subset \mls N$ such that the associated contraction of $\bC$ induces the relative coarse moduli morphism $\mls C^{\mls N} \to \mls C^{\mls N, \dag}$. 
Fix $x \in |C|$ and let $\bar x: \Sp(k) \to \mls C$ be be a geometric point mapping to $x$. Define
\[
K_{\bar x} := \ker(D(\mls N_{\bar x}^{gp}/\mathbf{Z}^{n_{\bar x}}) \to G_{f(\bar x)})
\]
where $G_{f(\bar x)}$ is the automorphism group scheme of $f(\bar x):\Sp(k) \to \cX$, the arrow is the homomorphism of $k$-group schemes induced by $f$, and $n_{\bar x}$ is the number of sections containing $x$.
Since $D(\mls N_{\bar x}^{gp}/\mathbf{Z}^{n_{\bar x}})$ is a diagonalizable group scheme over $k$, so is  $K_{\bar x}$, and hence the dual $D(K_{\bar x})$ is an abelian group. We define $\mls N_{\bar x}^\dagger$ to be the fiber product
\[
\begin{tikzcd}
\mls N^\dagger_{\bar x} \arrow[r] \arrow[d] & \mls N_{\bar x} \arrow[d]\\
\ker(\mls N_{\bar x}^{gp}/\mathbf{Z}^{n_{\bar x}} \to D(K_{\bar x})) \arrow[r] &  \mls N_{\bar x}^{gp}/\mathbf{Z}^{n_{\bar x}}
\end{tikzcd}
\]
so $\mls N^\dagger_{\bar x}$ is an admissible monoid by construction. 

By the second local description in \ref{SS:local} there exists an \'etale neighborhood $\Sp (R)\rightarrow C$ of $\bar x$ such that 
$$
\mls C_R^{\mls N}\simeq [\Sp (R\otimes _{\mathbf{Z}[\mathbf{N}^I]}\mathbf{Z}[\mls N_{\bar x}])/D(\mls N_{\bar x}^\gp /\mathbf{Z}^{n_x})],
$$
where $I\subset \{1, \dots, n\}$ is the subset of indices $i$ for which $s_i$ contains $\bar x$.  By the construction in \cite[3.6]{AOV}, the relative coarse moduli space in this situation is given by the quotient
$$
\mls C_R^{\mls N, \dag }\simeq [\Sp (R\otimes _{\mathbf{Z}[\mathbf{N}^I]}\mathbf{Z}[\mls N_{\bar x}])^{K_{\bar x}}/D((\mls N_{\bar x}^\dag )^\gp /\mathbf{Z}^{n_x})],
$$
where we have used the natural isomorphism
\[
D(\mls N_{\bar x}^{gp}/\mathbf{Z}^{n_{\bar x}}) / K_x \simeq D((\mls N_{\bar x}^\dagger)^{gp}/\mathbf{Z}^{n_{\bar x}}).
\]
Now since $\mathbf{N}^I\hookrightarrow \mls N_{\bar x}$ is integral the ring $R\otimes _{\mathbf{Z}[\mathbf{N}^I]}\mathbf{Z}[\mls N_{\bar x}]$ is isomorphic, as a $D(\mls N_{\bar x}^\gp /\mathbf{Z}^{n_x})$-representation over $R$, to a direct sum of eigenspaces 
indexed by the minimal lifts of elements in $\mls N_{\bar x}^\gp /\mathbf{Z}^{n_x}$.  From this it follows that we have an isomorphism
$$
(R\otimes _{\mathbf{Z}[\mathbf{N}^I]}\mathbf{Z}[\mls N_{\bar x}])^{K_{\bar x}}\simeq R\otimes _{\mathbf{Z}[\mathbf{N}^I]}\mathbf{Z}[\mls N_{\bar x}^\dag ].
$$
From this it follows that there exists a unique admissible subsheaf  $\mls N^\dag \subset \mls N$ whose stalks agree with the above defined $\mls N_{\bar x}^\dag $.  Furthermore, if $M_C^{\underline s}|_{\mls C}\hookrightarrow M'$ denotes the universal log structure on $\mls C$ whose associated map on characteristic sheaves is give by $\overline M_{C}^{\underline s}\hookrightarrow \mls N$ then the fiber product $M'\times _{\mls N}\mls N^\dag $ defines an object of $\mls U^{\mls N^\dag }$ over $\mls U^{\mls N}$.  This defines a global map $\mls U^{\mls N}\rightarrow \mls U^{\mls N^\dag }$ restricting \'etale locally on $C$ to the above isomorphism.  We conclude that $\mls U^{\mls N^\dag }$ is the relative coarse moduli space of the map $\mls U^{\mls N}\rightarrow \mc X$.

For the converse statement in the lemma in this case, let $q:\bC \to \bD$ be a contraction whose underlying morphism of marked prestable curves is an isomorphism and let $\mls M$ be the sheaf of monoids on $C$ associated to $\bD$. By \ref{R:3.20} we have that $q$ is determined by an inclusion $\mls M \subset \mls N$. We show that this inclusion identifies $\mls M$ with $\mls N^\dag$ by checking at every geometric point $\bar x$ of $C$. The inclusion $\mls M \subset \mls N$ induces a surjection $D(\mls N_{\bar x}/\mathbf Z^{n_{\bar x}}) \twoheadrightarrow D(\mls M_{\bar x}/\mathbf Z^{n_{\bar x}})$ which we identify with a surjection of isotropy groups of $\mls C$ to $\mls D$
at $\bar x$. On the other hand, since the morphism $D(\mls N_{\bar x}/\mathbf Z^{n_{\bar x}}) \to G_{f(\bar x)} $ factors through an inclusion  $D(\mls M_{\bar x}/\mathbf Z^{n_{\bar x}}) \to G_{f(\bar x)} $, we see we have an exact sequence
\[
0 \to K_{\bar x} \to D(\mls N_{\bar x}/\mathbf Z^{n_{\bar x}}) \to D(\mls M_{\bar x}/\mathbf Z^{n_{\bar x}})   \to 0.
\]
It follows from the definition of $\mls N^\dag_{\bar x}$ above that $\mls N^\dag_{\bar x} = \mls M_{\bar x}.$\\

\noindent 
\textit{The case $\mls C = \mls C^{\text{\rm node}}$.}
This case is very similar to the previous. If $\bar s \to S$ is a geometric point then we have $\overline M_{S, \bar s}\simeq \mathbf{N}^I$, where $I$ is the set of nodes of the fiber, and $\ell _C:\overline M_{S, \bar s}\rightarrow \overline M_{S, \bar s}'$ is given by positive integers $\{e_i\}_{i\in I}$.  If $R$ denotes $\mls O_{S, \bar s}$ then \'etale locally on $C_R$ the stack $\mls C$ is given around the $i$-th node $x_i$ by
$$
[\Sp (R[u, v]/(uv-t'))/\mu _{e_i}],
$$
where $t'$ is the image in $R$ of a generator for $M_{S, \bar s}'$ corresponding to $i$  and $\zeta \in \mu _{e_i}$ acts by $\zeta *u = \zeta u$, $\zeta *v= \zeta ^{-1}v$.
In particular, the stabilizer group scheme at $x_i$ is $\mu _{e_i}$.  The image of the map on stabilizer groups $\mu _{e_i}\rightarrow G_{f(x_i)}$ is then isomorphic to $\mu _{e_i'}$ for some $e_i'|e_i$ and the kernel is given by $\mu _{e_i/e_i'}$.  The new sequence of integers $\{e_i'\}_{i\in I}$ defines a submonoid of $\overline M_{S, \bar s}'$ which gives a smaller simple inclusion $M_S\hookrightarrow M_S^\dag \hookrightarrow  M_S$ in some neighborhood of $\bar s$.  The associated map $\mls C\rightarrow \mls C^\dag $ is given in local coordinates as above by the map
$$
R[u^\dag , v^\dag ]/(u^\dag v^\dag -(t^{\prime})^{e_i/e_i'})\rightarrow R[u, v]/(uv-t'), \ \ u^\dag \mapsto u^{e_i/e_i'}, v^\dag \mapsto v^{e_i/e_i'}.
$$
It follows as in the preceding case that $\mls C\rightarrow \mls C^\dag $ is the relative coarse moduli space of $\mls C\rightarrow \mc X$ in some neighborhood of $\bar b$.

The converse statement also follows from this description of $\mathbf{C}^\dag $ in terms of the kernels of the maps on stabilizer groups.
\end{proof}

\begin{cor}\label{lem:rep is iso}
Let $(q, \phi '):\bC \to \bD$ be a contraction of generalized log twisted curves such that the underlying contraction of marked prestable curves is an isomorphism. If the associated morphism $\mls C \to \mls D$ is representable then $(q, \phi ')$ is an isomorphism.
\end{cor}
\begin{proof}
   Since $\mls C \to \mls D$ is representable, $\mls C$ is itself the relative coarse space of this morphism. But \ref{lem:rel coarse} says that $\mls D$ is also the relative coarse space. It follows that we have an isomorphism $\mls D \to \mls C$ of stacks over $\mls D$, which therefore must be inverse to $\mls C \to \mls D$.
\end{proof}

\subsection{Initial contractions}\label{S:contract-curve}

In this section we explain how a contraction of coarse curves induces a canonical contraction of general log twisted curves.

\begin{defn}\label{def:initial}
Let $(q, \phi'): \bC \to \bD$ be a contraction. We say this contraction is \emph{initial} if, given a contraction $(p, \psi'): \bC \to \bD'$ such that $p$ factors through $q$, there is a unique contraction $(r, \rho'): \bD \to \bD'$ such that $(p, \psi') = (r, \rho') \circ (q, \phi')$.
\end{defn}

%Fix an integer $n\geq 0$ and weights $a_1, \dots, a_n\in (0, 1]$.
Let $\bC =(C/S, \{s_i\}_{i=1}^n, \ell :M_B\hookrightarrow M_B', \mls N)$ be a 
%prestable (with respect to $\bf a$) 
generalized log twisted curve over a base scheme $S$.  Let $(D/S, \{t_i\}_{i=1}^n)$ be a second prestable marked curve over $S$ and let $q:C\rightarrow D$ be a contraction morphism with $q(s_i) = t_i$ for all $i$.

\begin{thm}\label{T:8.1} The marked curve $(D/S, \{t_i\}_{i=1}^n)$ admits the structure of a generalized log twisted curve $\bD $ for which the contraction morphism $q$ extends to an initial contraction $(q, \phi '):\bC \rightarrow \bD$.
%which is initial in the sense that for any other contraction $(p, \psi '):\bC \rightarrow \bD '$ for which $p: C\rightarrow D'$ factors through $q$ there exists a unique contraction morphism $(r, \rho '):\bD \rightarrow \bD '$ such that   $(p, \psi ') = (r, \rho' )\circ (q, \phi ')$.
\end{thm}

We note that an initial contraction $\bC \to \bD$ as in \ref{T:8.1} will be unique when it exists.
%We say that a contraction $\bC \rightarrow \bD$ is \emph{initial} if it has the universal mapping property of the theorem. 
The proof of Theorem \ref{T:8.1} occupies the remainder of this section.

\begin{pg}
Let $r: D \to D'$ denote the morphism factoring $p$; that is, we have $p = r \circ q$. It follows from \ref{L:B.12} that $r$ is the unique morphism satisfying $p = r \circ q$ and that $r$ is a contraction.
\end{pg}
\begin{pg}
To define $\bD$ we need to specify a simple inclusion $\ell^D: \MS{S}{D} \to \MS{S}{D}'$ of log structures and an admissible monoid $\mls N^D$. We define $\MS{S}{D}'$ to be the saturation of $\MS{S}{D}$ in $\MS{S}{C}$ via the composition $\MS{S}{D} \xrightarrow{\phi} \MS{S}{C} \xrightarrow{\ell_C} \MS{S}{C}'$, so there is a commuting diagram of log structures
\begin{equation}\label{eq:bridge1}
\begin{tikzcd}
\MS{S}{D} \arrow[r, dashrightarrow, "\ell^D"] \arrow[d, "\phi"] & \MS{S}{D}' \arrow[d, dashrightarrow, "\phi'"] \\
\MS{S}{C} \arrow[r, "\ell^C"] & \MS{S}{C}'.
\end{tikzcd}
\end{equation}

\begin{lem}
The inclusion $\ell^D$ is simple.
\end{lem}
\begin{proof}
We can verify the lemma by looking at stalks of characteristic monoids at each geometric point $\bar s$ of $S$. From \ref{S:contract-monoid}, for the morphisms $\MS{S}{D} \xrightarrow{\phi} \MS{S}{C} \xrightarrow{\ell_C} \MS{S}{C}'$ we have
\[
\mathbf{N}^{J(D_{\bar s})} \xrightarrow{\phi} \mathbf{N}^{J(C_{\bar s})} \xrightarrow{\bc} \mathbf{N}^{J(C^{\bar s})} 
\]
where $\bc$ is multiplication by a sequence of positive integers $(c_1, \ldots, c_m)$. This composition of morphisms is a direct sum of morphisms indexed by nodes $x \in J(D_{\bar s})$, and the summand corresponding to $x$ is 
\[
\mathbf{N} \xrightarrow{\Delta} \mathbf{N}^{E^*(\Gamma_x)} \xrightarrow{\bc_{x}} \mathbf{N}^{E^*(\Gamma_x)} 
\]
where $\Delta$ is the diagonal and $\bc_{x}$ denotes the restriction of the vector $\bc$ to coordinates indexed by $E^*(\Gamma_x)$. The saturation of $\mathbf{N}$ under this inclusion is the submonoid of $\mathbf{N}^{E^*(\Gamma_x)}$ generated by $g^{-1}\bc_x$, where $g$ is the gcd of the coordinates of $\bc_x$. In particular this saturation is isomorphic to $\mathbf{N}$ and the induced map $\mathbf{N} \to \mathbf{N}$ is multiplication by $g$. Since saturation commutes with taking stalks, this concludes the proof of the lemma. 
\end{proof}
\end{pg}

\begin{pg}
To define $\mls N^D$, let $\mls C$ be the stack associated to $\bC$ and let $\tilde q$ be the composition $\mls C \to C \xrightarrow{q} D$. There is a diagram of solid arrows
\begin{equation}\label{eq:contract-monoid}
\begin{tikzcd}
\overline{M}^{\underline{s}}_D \arrow[d, hookrightarrow] \arrow[r, dashrightarrow] & \mls N^D \arrow[d, dashrightarrow] \arrow[r, dashrightarrow] & (\overline{M}^{\underline{s}}_D)_{\mathbf{Q}}\arrow[d, hookrightarrow]\\
\tilde q_*(\overline{M}_C|_{\mls C}) \arrow[r, hookrightarrow] & \tilde q_*\overline{M}_{\mls C} \arrow[r, hookrightarrow] & \tilde q_*(\overline{M}_C|_{\mls C})_{\mathbf Q}
\end{tikzcd}
\end{equation}
where the vertical arrows are the composition $\overline{M}^{\underline{s}}_D \to \overline{\tilde q_*M_C|_{\mls C}} \to \tilde q_*\overline{M}_C|_{\mls C}$ given explicitly in \ref{S:contract-monoid}. We define $\mls N^D$ to be the fiber product
\[
\mls N^D := (\overline{M}^{\underline{s}}_D)_{\mathbf{Q}} \times_{q_*(\overline{M}_C|_{\mls C})_{\mathbf Q}} \tilde q_*\overline{M}_{\mls C} 
\]
so that it fits into the diagram as shown.

\begin{lem}
The sheaf of monoids $\mls N^D$ is admissible.
\end{lem}
\begin{proof}
Since $\mls N^D$ is constructible and all the arrows in \eqref{eq:contract-monoid} are injective, we only need to show that $\mls N^D$ is finitely generated and saturated. One can check (for example using \ref{S:contract-monoid}) that all the monoids appearing in the solid part of \eqref{eq:contract-monoid} are fine and saturated, so $\mls N^D$ is as well. 
\end{proof}

\end{pg}

\begin{pg}

We define $\bD := (D/S, \{t_i\}_{i=1}^n, \mls N^D, \ell_D: \MS{S}{D} \to \MS{S}{D}')$ where $\mls N^D$ and $\ell_D$ are as above. The contraction morphism $\bC \to \bD$ is induced by \eqref{eq:bridge1}, noting that \eqref{eq:contract2} is the $(\tilde q^*, \tilde q_*)$-adjoint of \eqref{eq:contract-monoid} so the factorization in \eqref{eq:contract2} exists.

\end{pg}

\begin{pg}
To complete the proof of Theorem \ref{T:8.1} it remains to show that $\bC \to \bD$ has the desired universal property. Let $\bD' = (D'/S, \{t_i'\}_{i=1}^n, \mls N^{D'}, \MS{S}{D'} \to \MS{S}{D'}')$ be a generalized log twisted curve and let $(p, \psi'): \bC \to \bD'$ be a contraction such that $p = r \circ q$. The data defined thus far give a solid commuting diagram of inclusions
\[
\begin{tikzcd}
\MS{S}{D'} \arrow[d] \arrow[r, "\ell_{D'}"] & \MS{S}{D'}' \arrow[d, dashrightarrow] \arrow[dd, bend left=50, "\psi'"] \\
\MS{S}{D} \arrow[r, "\ell_D"] \arrow[d] & \MS{S}{D}' \arrow[d, "\phi'"']\\
\MS{S}{C} \arrow[r, "\ell_C"] & \MS{S}{C}'.
\end{tikzcd}
\]
It follows from the fact that $\MS{S}{D}'$ is the saturation of $\MS{S}{D}$ in $\MS{S}{C}'$ and that $\ell_{D'}$ is simple that the dotted arrow exists and is unique. Call this dotted arrow $\rho'$.

To see that $(r, \rho')$ defines a contraction from $\bD \to \bD'$ it remains to show the existence of the dotted arrow in \eqref{eq:contract2}. Pushing forward preserves fiber products, so pushing forward \eqref{eq:contract-monoid} gives the bottom part of a solid diagram
\begin{equation}\label{eq:recent}
\begin{tikzcd}
\overline{M}^{\underline{s}}_{D'} \arrow[d, hookrightarrow] \arrow[r, hookrightarrow] & \mls N^{D'} \arrow[r, hookrightarrow] \arrow[d, dashrightarrow] & (\overline{M}^{\underline{s}}_{D'})_{\mathbf Q} \arrow[d, hookrightarrow] \\
r_*\overline{M}^{\underline{s}}_D \arrow[d, hookrightarrow] \arrow[r, hookrightarrow] & r_*\mls N^D \arrow[d, hookrightarrow] \arrow[r, hookrightarrow] & r_*(\overline{M}^{\underline{s}}_D)_{\mathbf{Q}}\arrow[d, hookrightarrow]\\
r_*\tilde q_*(\overline{M}_C|_{\mls C}) \arrow[r, hookrightarrow] & r_*\tilde q_*\overline{M}_{\mls C} \arrow[r, hookrightarrow] & r_*\tilde q_*(\overline{M}_C|_{\mls C})_{\mathbf Q}
\end{tikzcd}
\end{equation}
where the bottom right square is fibered. The top row of the diagram comes from $\bD'$ and the arrow $\overline{M}^{\underline{s}}_{D'} \to r_*\overline{M}^{\underline{s}}_D$ is equal to the composition $\overline{M}^{\underline{s}}_{D'} \to \overline{r_*M^{\underline{s}}_D} \to r_*\overline{M}^{\underline{s}}_D$. Since $(p, \psi')$ is a contraction, there exists an arrow $\mls N^{D'}|_{\mls C} \to \overline{M}_\mls C$ filling in \eqref{eq:contract2}, and the adjoint of this arrow is a map $\mls N^{D'} \to r_*\tilde q_*\overline{M}_{\mls C}$ commuting with the other solid arrows in \eqref{eq:recent}. It follows that the dashed arrow in \eqref{eq:recent} exists. Applying adjunction again, this time just to the top two cells, and using that $\mls N^D|_{\mls D}$ is a summand of $\overline{M}_{\mls D}$, this dashed arrow induces the required factorization in \eqref{eq:contract2}.
\end{pg}
This completes the proof of \ref{T:8.1}. \qed

\section{Tame abelian nodal orbicurves versus generalized log twisted curves}\label{S:image}

The category of generalized log twisted curves with distinct markings is precisely the category of log twisted curves considered in \cite[A.3]{AOV}. By \cite[A.5]{AOV}, this category is equivalent to the category of twisted curves as defined in \cite[2.1]{AOV}. In other words, the (restriction of the) functor defined in \ref{E:twistedfunctor} from log twisted curves to stacks is fully faithful and the essential image is the category of twisted curves. 

In this section, we see that the functor \ref{E:twistedfunctor} from generalized log twisted curves to stacks is faithful but no longer full \ref{ss:faith-not-full}, and we characterize its essential image \ref{C:4.13}. We also show that fibers of this functor are finite (at least over an algebraically closed field) \ref{P:4.15}.

% While not strictly necessary for what follows, it is natural to consider the question of whether the stack $\mls C$ attached to a generalized log twisted curve $\mathbf{C}$ in fact determines $\mathbf{C}$, as is the case when the points do not collide \cite{LogTwisted}.
% It turns out that $\mathbf{C}$ does not quite determine $\mls C$, but that we can characterize those stacks that may arise from generalized log twisted curves.  

%Our next task is to characterize the essential image of the functor in \ref{SS:3.4}. For this we need some vocabulary.
%Let $k$ be an algebraically closed field.

For this section we fix an algebraically closed field $k$.

\subsection{Tame abelian nodal orbicurves}

\begin{defn}\label{def:abelian curve}
A \emph{tame nodal orbicurve} over $k$ is a tame proper stack $\mls C$ of finite presentation over $k$ satisfying the following: 
\begin{enumerate}
\item [(i)] The coarse space $C$ of $\mls C$ is a nodal curve and $\mls C \to C$ is an isomorphism over a dense open substack of $C$.
\item [(ii)] The morphism $\mls C \to C$ is flat away from the nodes of $C$.
\item [(iii)] For any geometric point $\bar x \in C$ mapping to a node, there is an integer $r$ such that
\[
\Spec(\mls O_{C, \bar x}) \times_C \mls C \simeq[\Spec((k[u, v]/uv)^{sh} / \mu_r]
\]
where $\mu_r$ acts with weight 1 on $u$ and weight $r-1$ on $v$.
\end{enumerate} 
We say that $\mls C$ is \emph{abelian} if moreover the stabilizer group scheme at every geometric point of $\mls C$ is abelian.

If $S$ is a scheme then a \emph{tame nodal orbicurve over $S$} is a proper flat tame stack $\mls C\rightarrow S$ all  of whose geometric fibers are tame nodal orbicurves.  Similarly we define abelian tame nodal orbicurves over $S$.
\end{defn}

\begin{rem}\label{R:flat-reduced}
We remark that the flatness in condition (ii) can be replaced by the condition that $\mls C$ is reduced. That is, let $\mls C$ be a tame proper stack over an algebraically closed field $k$
satisfying (i) and (iii) in \ref{def:abelian curve}. Then (ii) holds if and only if 
\begin{enumerate}
\item[ (ii)'] $\mls C$ is reduced. 
\end{enumerate}
We note that (iii) ensures that $\mls C$ is always reduced at nodes of $C$, so both (ii) and (ii)' can be checked Zariski locally on $C^{sm}$. Therefore we may replace $C$ by the spectrum of a discrete valuation ring $A$ with uniformizer $x$ ($A$ is the local ring at a closed point of $C$) and $\mls C$ by its pullback to $\mathrm{Spec}(A)$. Let $\Spec(B) \to \mls C$ be a smooth cover, so $B$ is an $A$-algebra. 
The map $A \to B$ is flat if and only if the map $B \xrightarrow{\cdot x} B$ is injective, and this in turn is true if and only if $B \to B_x$ is injective. 

Let $U \to \mathrm{Spec}(A)$ be the complement of the divisor defined by $x$ and let $\mls C_U = \mls C \times_{\mathrm{Spec}(A)} U$.
If $\mls C$ is reduced, then $\mls C_U$ is schematically dense in $\mls C$. It follows that $B \to B_x$ is injective and hence $A \to B$ is flat by the preceeding paragraph. Conversely, if $A \to B$ is flat, then $B \to B_x$ is injective. But $B_x$ is reduced, being smooth over $A_x$ (smoothness holds since $G$ is smooth and $\mls C_U \to U$ is an isomorphism by (i)). 
So $B$ must also be reduced. 
\end{rem}

\subsection{Local monoids}
Let $\mls C$ be a tame abelian nodal orbicurve with coarse space $C$ and let $\bar x \to C^{sm}$ be a geometric point with residue field $k$. Let $R = \mls O_{C, \bar x}$ and let $\mls C_R$ be the fiber product $\mls C \times_C \mathrm{Spec}(R)$. Let $D_{\bar x} \subset \mls Div^+(\mls C_R)$ denote the subgroupoid  of generalized effective Cartier divisors $(\mc L, \gamma :\mc L\rightarrow \mls O_{\mls C_R})$ on $\mls C_R$ for which the map $\gamma $ restricts to an isomorphism away from the preimage of $x$. For such a pair $(\mc L, \gamma )$ the map $\gamma $ is injective, so any isomorphism between objects is uniquely determined. It follows that $D_{\bar x}$ is equivalent to its set of isomorphism classes and from now on we identify $D_{\bar x}$ with this set.  The addition operation on generalized effective Cartier divisors makes $D_{\bar x}$ a monoid.

 \begin{defn}
The \emph{local monoid at a $k$-point} $\bar x \in C^{sm}$ is the monoid $D_{\bar x}$ constructed above.
 \end{defn}

\begin{pg}\label{P:7.6}
\label{R:localmonoid} Let $\mls D_{\bar x}$ be the \'etale sheaf on $C$  that to a morphism $U \to C$ associates the set of isomorphism classes of generalized effective Cartier divisors $(\mls L, \gamma: \mls L \to \mls O_{\mls C_U})$ on $\mls C_U$ for which $\gamma$ is an isomorphism away from $\bar x$. Since the functor $\mls D_{\bar x}$ is limit-preserving, $D_{\bar x}$ is the stalk of $\mls D_{\bar x}$ at $\bar x$, and the stalks at other geometric points of $C$ are $0$.  Therefore $\mls D_{\bar x}\simeq x_*D_{\bar x}$.  Note also that for any morphism $V\rightarrow \mls C$ there is a tautological map $H^0(V, \mls D_{\bar x}|_V)\rightarrow \mls Div^+(V)$ defining a log structure $M_{\mls C}$ on $\mls C$ with characteristic sheaf $\mls D_{\bar x}|_{\mls C}$.  Furthermore, the inclusion $\mathbf{N}\hookrightarrow D_{\bar x}$ induces an inclusion of log structures $M^x|_{\mls C}\hookrightarrow M_{\mls C}$, where $M^x$ is the log structure on $C$ defined by $x$.
\end{pg}
\begin{pg}
Let $G_{\bar x}$ denote the stabilizer group scheme of  $\mls C$ at $\bar x \in C^{sm}$ (see \ref{SS:conventions}).  The group $G_{\bar x}$ is  a finite abelian tame group scheme, therefore a product of group schemes $\mu_a$ for various positive integers $a$. We define the \emph{character group $X_{\bar x}$} at $\bar x \in C^{sm}$ to be the character group of $G_{\bar x}$, noting that this is isomorphic to $\Pic(BG_{\bar x})$. 

Let $R$ denote $\mls O_{C, \bar x}^{sh}$, and let $\mls C_R$ denote the base change $\mls C\times _C\Sp (R)$.  Let $G$ be the diagonalizable group scheme over $R$ associated to $X_{\bar x}$.  Then by \cite[Proof of 3.6]{tame} we can write $\mls C_R\simeq [\Sp (S)/G]$, where $S$ is a finite local $R$-algebra.
\end{pg}

\begin{lem} The reduction functor $\text{\rm Pic}(\mls C_R)\rightarrow \text{\rm Pic}(BG_{\bar x}) = X_{\bar x}$ is an isomorphism.
\end{lem}
\begin{proof}
The presentation of $\mls C_R$ as a quotient by $G$ defines a map $\mls C_R\rightarrow BG$, and therefore a map $X_{\bar x} = \text{Pic}(BG)\rightarrow \text{Pic}(\mls C_R)$ defining a section of the map in the lemma.  This implies surjectivity.  For injectivity, note that a line bundle $\mc L$ on $\mls C_R$ with trivial stabilizer group action descent to a line bundle on $\Sp (R)$ and therefore is trivial.
\end{proof}

Consider the homomorphism
\[\chi: D_{\bar x} \to \Pic(BG_{\bar x})\] sending $(\mc L, \gamma)$ to the restriction of $\mc L$ to $BG_{\bar x}$.

\begin{lem}\label{lem:Dx}
For every 
 $\theta \in \Pic(BG_x)$ there exists a unique element $\delta_\theta \in \chi ^{-1}(\theta)$ such that the map $\mathbf{N}\rightarrow \chi ^{-1}(\theta)$ sending $a$ to $a+\delta_\theta$ is an isomorphism.
\end{lem}
\begin{proof}
The action of $G_{\bar x}$ on $S$ defines a decomposition $S = \oplus _{\theta \in X}Re^{\theta}$, where $G_{\bar \theta}$ acts on $e^\theta$ through the character $\theta$ (each eigenspace is a free $R$-module since $S$ is flat over $R$ and of rank $1$ since the action is generically free).  We take $\delta_\theta$ to be the element of $D_{\bar x}$ defined by the map $S\cdot e^\theta\rightarrow S$.  

To see that $\delta_\theta$ has the desired properties, let $M$ be a free $S$-module of rank $1$ with action of $G_{\bar x}$ and let $u:M\rightarrow S$ be a map which is an isomorphism after inverting a uniformizer $\pi \in R$.  Suppose further that the character defined by $M/\mathfrak{m}_SM$ is $\theta$. Then $u: M \to S$ defines an element of $\chi^{-1}(\theta)$. Since $G_{\bar x}$ is linearly reductive we can lift a basis element for $M/\mathfrak{m}_SM$ to a basis element $m_0\in M$ on which $G_{\bar x}$ acts through $\theta$. 
%\Rachel{what theorem is this?}  
This identifies $M$ with $S\cdot m_0$ as a $G_{\bar x}$-module.  The image of $m_0$ in $S$ is an element on which $G_{\bar x}$ acts through $\theta$ and therefore this image equals $fe^\theta$ for some element $f\in R$ which is invertible after inverting $\pi $.  It follows that $f = v\pi ^a$ for some $v\in R^*$ and $a\geq 0$.  Replacing $m_0$ by $v^{-1}m_0$ we then identify $u: M \to S$ with the tensor product of $\delta_\theta: S \cdot e^\theta \to S$ with the pullback of $\pi ^a:R\rightarrow R$.
\end{proof}

\begin{cor}\label{cor:Dx} The monoid $D_{\bar x}$ is a fine sharp monoid and the inclusion $\mathbf{N}\hookrightarrow D_{\bar x}$ is integral.
\end{cor}
\begin{proof}
    Lemma \ref{lem:Dx} implies that $D_{\bar x}$ is finitely generated.  The uniqueness part of that lemma implies that it is cancellative. It is sharp because if some power of $(\mc L, \gamma) \in D_{\bar x}$ is the trivial element, then some power of $\gamma$ is nowhere vanishing, and hence $\gamma$ is nowhere vanishing. The statement that $\mathbf{N}\hookrightarrow D_{\bar x}$ is integral follows from the description in \ref{lem:Dx}.
\end{proof}

\begin{defn}\label{def:local monoid} A \emph{local monoid} is a fine sharp monoid $D$ with an integral inclusion $\mathbf{N}\hookrightarrow D$ with $D^\gp /\mathbf{Z}$ a finite abelian group.
\end{defn}

\begin{rem}  Note that the quotient $D^\gp /\mathbf{Z}$ is also the quotient $D/\mathbf{N}$ in the category of monoids.
\end{rem}

\begin{rem}
If $D$ is a local monoid with $D^\gp /\mathbf{Z} = X$ then there is a bijection of sets $\mathbf{N}\times X\rightarrow D$ sending $(a, \theta )$ to $a+e^\theta $, where $e^\theta \in D$ is the minimal lift of $\theta $ to $D$ (using the integrality of $\mathbf{N}\hookrightarrow D$.  In this way we get a monoid structure on $\mathbf{N}\times X$ compatible with the monoid structure on $\mathbf{N}$.  Therefore a local monoid can be described by the set $\mathbf{N} \times X$ equipped with a binary operation $+$ given by
\[
(a, \theta) + (a', \theta') = (a+a'+c_{\theta \theta'}, \theta +\theta')
\]
where $c_{\theta \theta'} \in \mathbf{N}$ satisfy
\begin{enumerate}
\item $c_{0\theta}=0$  (identity)
\item $c_{\theta \theta'} = c_{\theta' \theta}$  (commutativity)
\item $c_{\theta \theta'} + c_{(\theta + \theta') \theta''} = c_{\theta(\theta'+\theta'')} + c_{\theta'\theta''}$  (associativity)
\item $c_{\theta \theta^{-1}} = 0$ implies $\theta = 1$ (sharpness)
%$c_{\theta \theta} + c_{\theta \theta^2} + \ldots + c_{\theta \theta^{|\theta|-1}} \neq 0$  (sharpness)
\end{enumerate}
for all $\theta, \theta'$, and $\theta''$ in $X$.
\end{rem}

\subsection{Characterization of tame abelian nodal orbicurves}

We can construct all tame abelian nodal orbicurves fairly explicitly. Consider a tuple 
\begin{equation}\label{eq:tanc}
(C, \{s_i\}_{i=1}^n, M_k \hookrightarrow M_k', \{D_i\}_{i=1}^n)
\end{equation}
where $(C, \{s_i\}_{i=1}^n)$ is a prestable curve over $k$ with \textit{distinct} markings, $M_k$ is the canonical log structure on $\Sp(k)$ associated to $C$, $M_k \to M_k'$ is a simple inclusion, and each $D_i$ is a local monoid as in \ref{def:local monoid}. Given a tuple as in \eqref{eq:tanc} we obtain an orbicurve $\mls C$ over $k$ as follows. Let $s_{i, *} D_i$ be the \'etale skyscraper sheaf with stalk $D_i$ at the image of $s_i$ and let $\mls P = \oplus s_{i, *} D_i$. As in \ref{SS:3.4} there is a tame stack $\mls C^{\node}$ associated to the simple inclusion $M_k \hookrightarrow M_k'$ and a stack of roots $\mls C^{\mls P}$ arising from the system of denominators $\oplus s_{i, *} \mathbf{N} \hookrightarrow \mls P$. We define
\[
\mls C := \mls C^{\node} \times_C \mls C^{\mls P}.
\]

\begin{lem}
The stack $\mls C$ associated to a tuple \eqref{eq:tanc} is a tame abelian nodal orbicurve with coarse space $C$ whose stabilizer group scheme at the image of $s_i$ is $D(D_i/\mathbf{N})$. 
\end{lem}
\begin{proof}
The local descriptions of $\mls C^{\mls N}$ in \ref{SS:local} are also valid for $\mls C^{\mls P}$, after replacing $\mathbf{N}^n \to N$ with $\mathbf{N} \to D_i$. The lemma follows from these local descriptions as in the proof of
\ref{cor:properties}.
\end{proof}

\begin{prop}\label{p:tanc} Every tame abelian nodal orbicurve arises from a tuple \eqref{eq:tanc}.
\end{prop}
\begin{proof}
Let $\mls C$ be a tame abelian nodal orbicurve with coarse space $C$, and let $\{s_i\}_{i=1}^n$ be the $k$-points of $C$ with nontrivial stabilizer. Let $\mls C^{\node}$ be the orbicurve obtained by gluing the complement of the markings in $\mls C$ with $C^{sm}.$ Define $M_k \hookrightarrow M_k'$ to be the simple inclusion associated to this stack by \cite[1.8]{LogTwisted}. There is a natural morphism $\mls C \to \mls C^{\node}.$

On the other hand, if we define $D_i$ to be the local monoid of $\mls C$ at $s_i$ and $\mls P = \oplus _is_{i*}D_i$, we get a canonical morphism $\mls C \to \mls C^{\mls P}$ from the log structure $M_{\mls C}$ defined in \ref{P:7.6}. 
We have an induced morphism 
\begin{equation}\label{eq:tanc1}
\mls C \to \mls C^{\node} \times_C \mls C^{\mls P}
\end{equation}
from $\mls C$ to the stack associated to the tuple $(C, \{s_i\}_{i=1}^n, M_k \hookrightarrow M_k', \{D_i\}_{i=1}^n),$ which we claim is an isomorphism.

That \eqref{eq:tanc1} is an isomorphism in a neighborhood of a node follows from the definition of $M_k'$.  It therefore suffices to show that the map $\mls C\rightarrow \mls C^{\mls P}$ is an isomorphism at the images $x_i$ of the $s_i$.  Fix one such point $x_i$ and let $R$ denote the strictly henselian local ring of $C$ at a geometric point over $x_i$.  It then suffices to show that the base change  $\mls C_{R}\rightarrow \mls C^{\mls P}_{R}$ is an isomorphism.  Fixing a uniformizer in $R$ we can describe the stack $\mls C^{\mls P}_R$ as a quotient
$$
\mls C^{\mls P}_R=[\Sp (R\otimes _{\mathbf{Z}[\mathbf{N}]}\mathbf{Z}[D_{i}])/G_{x_i}].
$$
Since the map $\mls C_R\rightarrow \mls C^{\mls P}_R$ is representable (it induces an isomorphism on stabilizer groups) the fiber product $\mls C_R\times _{\mls C^{\mls P}_R}\Sp (R\otimes _{\mathbf{Z}[\mathbf{N}]}\mathbf{Z}[D_{i}])$ is an algebraic space finite over $\Sp (R)$ and hence an affine scheme.  We therefore obtain a finite $R$-algebra $S$ with $G_{x_i}$-action such that $\mls C_R = [\Sp (S)/G_{x_i}]$ which fits into a commutative diagram
$$
\xymatrix{
\Sp (S)\ar[d]\ar[r]& \Sp (R\otimes _{\mathbf{Z}[\mathbf{N}]}\mathbf{Z}[D_{i}])\ar[d]\\
\mls C_R\ar[r]& \mls C_R^{\mls P}.}
$$
Furthermore, the log structure on $\Sp (S)$ induced by that on $\mls C_R$ defining the map to $\mls C^{\mls P}$ is the pullback of the log structure on $\Sp (R\otimes _{\mathbf{Z}[\mathbf{N}]}\mathbf{Z}[D_{i}])$ induced by the natural map $D_{i}\rightarrow R\otimes _{\mathbf{Z}[\mathbf{N}]}\mathbf{Z}[D_{i}]$. In particular we have a global chart
\begin{equation}\label{eq:tanc2}
D_i \to S
\end{equation}
for this log structure. 
Decomposing $S$ as $\oplus _{\theta \in X_i}Re^\theta$ as in the proof of \ref{lem:Dx} we see that the image of the class $\delta_\theta \in D_{i}$ maps under \eqref{eq:tanc2} to an element of $R^*$ times $e^\theta$ in $S$.  It follows that the map $R\otimes _{\mathbf{Z}[\mathbf{N}]}\mathbf{Z}[D_{x_i}]\to S$ decomposes into a direct sum of isomorphisms of free $R$-modules of rank $1$ and therefore is an isomorphism.
\end{proof}

\subsection{The functor from generalized log twisted curves to tame abelian nodal orbicurves}

As discussed in \cite[7.10]{bragg2024amplevectorbundlesmoduli} the collection of tame abelian nodal orbicurves, which a priori is a 2-category, is in fact equivalent to a $1$-category.
There is a functor
\begin{equation}\label{E:twistedfunctor}
(\text{generalized log twisted curves over $S$})\rightarrow (\text{tame abelian nodal orbicurves})
\end{equation}
defined as follows.

On objects, the functor from generalized log twisted curves to algebraic stacks is given by the rule
\begin{equation}\label{eq:functor}
\mathbf C \mapsto \mls C \quad \quad \;\text{where} \quad \quad\;\mls C := \mls C^{\node} \times_C \mls C^{\mls N},
\end{equation}
and we call $\mls C$ the \emph{stack associated to $\mathbf{C}$.} Note that if $\mathbf{C}$ is a generalized log twisted curve over $S$ then $\mls C$ has a canonical morphism to $S$.

On morphisms the functor is defined as follows. Let $\mathbf{C}^{(1)} \to \mathbf{C}^{(2)}$ be a morphism of generalized log twisted curves over a scheme $S$, where $\mathbf{C}^{(j)}$ are as in \eqref{eq:some curves}. Hence we have a pair $(f: C^{(1)} \to C^{(2)}, \rho: M^{(1)\prime}_S \to M^{(2)\prime}_S)$ satisfying the conditions in \ref{D:2.23}. These induce morphisms of stacks 
\[\mls C^{(1), \,\node} \to \mls C^{(2), \,\node} \quad \quad \quad \quad\quad \mls C^{\mls N^{(1)}} \to \mls C^{\mls N^{(2)}},
\]
whence we obtain $\mls C^{(1)} \to \mls C^{(2)}$. The morphism $\mls C^{\mls N^{(1)}} \to \mls C^{\mls N^{(2)}}$ is induced from the inclusion $\mls N^{(2)} \subset f_*\mls N^{(1)}$ and the definition of $\mls C^{\mls N^{(j)}}$ as the category of lifts filling in \ref{eq:def CN}.

\subsection{Faithful but not full}\label{ss:faith-not-full}

To see that the functor in \ref{E:twistedfunctor} is faithful, let $\bC$ be a generalized log twisted curve over a scheme $S$. An automorphism of $\bC$ is determined by the induced automorphism $\phi^{\node}$ of $\mls C^{\node}$ (which also preserves the monoid), so it is enough to show that if the automorphism $\phi$ of $\mls C$ induced by $\phi^{\node}$ is 2-isomorphic to the identity, then $\phi^{\node}$ is 2-isomorphic to the identity.  This follows from noting that the map $\mls C\rightarrow \mls C^\node $ is a coarse moduli space map away from the nodes, and therefore $\phi ^\node $ is determined by the automorphism $\phi $ of the stack.

We now give two examples to show that the functor in \ref{E:twistedfunctor} is not full. The first is an example of two nonisomorphic generalized log twisted curves with isomorphic associated stacks, and the second is an example of a single generalized log twisted curve with fewer automorphisms than its associated stack.

\begin{example}\label{ex:two-to-one}
Let $k$ be a field, let $C = \mathbf{P}^1_k$ and let $s_1=s_2$ be the same marking in $C$. Consider the admissible sheaves $\mls N_{i}$ defined by submonoids $N_i\subset \mathbf{Q}^2_{\geq 0}$ defined as follows:
$$
N_1 = (1/2)\mathbf{N} \times \mathbf{N} , \ \ N_2 =\mathbf{N} \times (1/2)\mathbf{N} .
$$
Then we have two non-isomorphic generalized log twisted curves 
\[
\bC_i = (\mathbf{P}^1_k, \{s_1, s_2\}, M_k \xrightarrow{\sim} M_k, \mls N_i) \quad \quad \text{for}\; i=1, 2
\] but it follows from \eqref{eq:local description2} that the associated stacks are both isomorphic to the root stack of $C$ at $s_1=s_2$ with $\mu_2$ stabilizer.
\end{example}

\begin{example} 
Let $k$ be a field, let $C = \mathbf{P}^1_k$, let $s_1=s_2$ be the same marking in $C$, and fix two other markings $s_3 \neq s_4$ in $C$ distinct from $s_1=s_2$. Consider the admissible sheaf defined by the submonoid
\[
N = (1/2)\mathbf{N} \times (1/2)\mathbf{N} \times \mathbf{N} \times \mathbf{N} \subset \mathbf{Q}^4_{\geq 0}
%\langle (\frac{1}{2}, 0,0,0), (0, \frac{1}{2},0,0),(0,0,1,0),(0,0,0,1) \rangle 
\]
and let $\bC$ be the generalized log twisted curve associated to $C, \{s_1, s_2\},$ and $N.$ Then the only automorphism of $\bC$ is the trivial automorphism.
However, the automorphism of $N$ that swaps the first two generators induces a nontrivial automorphism of the associated stack $\mls C$: it is nontrivial because it induces a nontrivial automorphism of the stabilizer group scheme at a geometric point corresponding to $s_1=s_2$. 
\end{example}

\subsection{Essential image}

We can now characterize the essential image of the associated stack functor \ref{E:twistedfunctor} from generalized log twisted curves to tame abelian nodal orbicurves.

\begin{cor}\label{C:4.13} A tame abelian nodal orbicurve $\mls C$ is isomorphic to the stack associated to a generalized log twisted curve if and only if for each $k$-point $\bar x \in C$ the monoid $D_{\bar x}$ can be realized as a pushout in the category of monoids %\Rachel{category of integral monoids?}
$$
\xymatrix{
\mathbf{N}^{m}\ar[d]\ar[r]& N\ar[d]\\
\mathbf{N}\ar[r]& D_{\bar x}}
$$
for some integer $m\geq 1$ and admissible submonoid $N\subset \mathbf{Q}^{m}_{\geq 0}$, where the left vertical map is summation.
\end{cor}
\begin{proof}
    If $\mathbf{C} = (C/k, \{s_i\}_{i=1}^n, M_k\hookrightarrow M_k', \mls N)$ is a generalized log twisted curve then we can take $m$ to be the number of marked points passing through $\bar x$.  
    %The valuation on $\mls O_{C, x_i}$ then defines a map $\oplus _{s_j = x_i}\mathbf{N}\rightarrow \mathbf{N}$ and 
    Chasing through the above construction of $D_{\bar x}$ we find that $D_{\bar x} = \mls N_{\bar x}\oplus_{\mathbf{N}^m} \mathbf{N} $.
    % this monoid can be realized as the pushout of the diagram
    % $$
    % \xymatrix{
    % \mathbf{N}^{m}\ar[d]\ar[r]& \mls N_{\bar x}\\
    % \mathbf{N}&}
    % $$
    % where the left vertical map is summation.
    %\Rachel{It will be a pushout in the category of integral monoids, right? Because we showed earlier that $D_{\bar x}$ is integral, but the pushout of this diagram needn't be in general, since none of these guys is a group?}
    
    Conversely if $\mls C$ is a tame abelian nodal orbicurve, let $\{\bar x_j\}$ be the $k$-points of the coarse space $C$ with nontrivial stabilizer. If the local monoid $D_{\bar x_j}$ can be written as a pushout $N_j \oplus_{\mathbf{N}^{m_j} }\mathbf{N}$, then define markings $\{s_i\}$ on $C$ by taking each $\bar x_j$ with multiplicity $m_j$ and let $\mls N := \oplus x_{j, *} M_j$. If $M_k \hookrightarrow M_k'$ is the simple inclusion induced by the log structure of the nodes, then $\bC := (C, \{s_i\}, M_k \hookrightarrow M_k', \mls N)$ is a generalized log twisted curve with associated stack $\mls C$ by \ref{p:tanc}.
    
\end{proof}

\begin{rem}\label{R:integral}
Since the inclusion $\mathbf{N}^m \to N$ into an admissible monoid is integral by \ref{L:integral}, the pushout $\mathbf{N} \oplus_{\mathbf{N}^m} N$ will be integral and in particular this pushout can be computed in the category of integral monoids.
\end{rem}

% \begin{example}\label{ex:admissible}
% Let $N \subset \mathbf{Q}^n_{\geq 0}$ be an admissible inclusion. We describe the local monoid arising via pushout as in \ref{C:4.13}.

% Using \ref{R:integral}, the pushout is the image of $\mathbf{N} \oplus N$ in $\mathbf{Z} \oplus_{\mathbf{Z}^n} N^{gp}$. Since elements of the latter group can be uniquely written as $(a, \mathbf{x})$ for some $a \in \mathbf{Z}$ and $\mathbf{x} \in N \cap [0, 1)^n$, and since elements of $\mathbf{Z}\oplus_{\mathbf{Z}^n} N^{gp}/\mathbf{Z} \simeq N^{gp}/\mathbf{Z}^n$ have unique representatives in $N \cap [0, 1)^n$, this image is the local monoid $\mathbf{N} \times N^{gp}/\mathbf{Z}^n$ with structure constants computed as follows. For $\theta, \eta \in N^{gp}/\mathbf{Z}^n$ let $\overline{\theta}, \overline{\eta}$ be the unique representatives in $N \cap [0, 1)^n$. Then
% \[
% c_{\theta \eta} = \#\{i \in \{1, \ldots, n\} \mid \overline{\theta}_i + \overline{\eta}_i \geq 1\}.
% \]
% \end{example}

\begin{example}\label{ex:mu2}
Every local monoid with quotient $D/\mathbf{N} \simeq \mathbf{Z}/2\mathbf{Z}$ arises as a pushout of an admissible monoid.  Indeed let $z\in D$ be the minimal element mapping to $1$ in $\mathbf{Z}/(2)$ and write $2z = c$ for some $c\geq 1$ (note that since $D$ is sharp we cannot have $c = 0$). Let $N$ be the monoid 
\[\mathbf N^{c} + \mathbf{N}(1/2, \ldots, 1/2)  \subset \mathbf{Q}^{c}_{\geq 0}.\]
Then the map $N\rightarrow D$ sending $(1/2, \ldots, 1/2)$ to $c$ identifies $D$ with the pushout of $N$.
\end{example}

\begin{example}\label{ex:mu3}
    Every local monoid $D$ with quotient $D/\mathbf{N}\simeq \mathbf{Z}/(3)$ arises as a pushout of an admissible monoid.  To see this let $z_1\in D$ (resp. $z_2\in D$) be the minimal lift of $1\in \mathbf{Z}/(3)$ (resp. $2\in \mathbf{Z}/(3)$).  Write
    $$
    2z_1 = a+z_2, 2z_2 = b+z_1, 3z_1 = c, 3z_2 = d
    $$
    for some $a,b, c, d\in \mathbf{N}$.  By associativity of addition we then have 
    $$
    3z_1 = a+z_2+z_1 = c, \ \ 3z_2 = b+z_1+z_2 = d.
    $$
It follows that $d-b = c-a$ and these quantities are equal to $z_1+z_2$ so positive (since $D$ is sharp).  Note also that 
$$
d-b + z_1 = z_1+z_2+z_1 = 2z_1 +z_2 = a+2z_2 = a+b+z_1.
$$
Therefore $a+b = d-b=c-a.$

Let $N$ be the monoid 
\[\mathbf N^{a+b} + \mathbf{N}((1/3)^{b}, \,(2/3)^{a}) + \mathbf{N}((2/3)^{b}, \,(1/3)^{a})\subset \mathbf{Q}^{a+b}_{\geq 0}\]
where an exponent $m$ on a fraction means to repeat that entry $m$ times. Then there is a map $N\rightarrow D$ sending $((1/3)^{b}, \,(2/3)^{a})$ to $z_1$ and $((2/3)^{b}, \,(1/3)^{a})$ to $z_2$ realizing $D$ as the pushout of $N$.
\end{example}

The next example shows that a tame abelian nodal orbicurve with a $\mu_4$ stabilizer may not arise from a generalized log twisted curve.

%\Rachel{rewrote this example. did I understand correctly?}
\begin{example}\label{ex:not} Consider the stack 
\begin{equation}\label{eq:not2}
\mc X:= [\Sp (k[y,w]/(y^2-w^2))/\mu _4],
\end{equation}
where $\mu _4$ acts on $y$ with weight $1$ and on $w$ with weight $3$. 
This is the stack associated to the monoid $D\subset \mathbf{N}\oplus \mathbf{Z}/(2)$ consisting of $(0,0)$ and pairs $(a, b)$ with $a>0$.  There is a map $D\rightarrow \mathbf{Z}/(4)$ sending $(a, b)$ to $a+2b$.  The fiber over $0$ of this map is the submonoid $\mathbf{N}\hookrightarrow D$, $m\mapsto m\cdot (2,1)$.  Observe that the fiber over $1$ (resp. $2$, $3$) consists of $\mathbf{N}+(1,0)$ (resp. $\mathbf{N}+(2,0)$, $\mathbf{N}+(1,1)$). 
%We therefore have a bijection $D\simeq \mathbf{N}\oplus \mathbf{Z}/(4)$.
The coarse space of $\mc X$ is $Y:=\Sp (k[x])$ with the map $\mc X\rightarrow Y$ induced by $x\mapsto yw$.  

This stack is not of the form 
\begin{equation}\label{eq:not1}\Sp (k[x]\otimes _{k[\mathbf{N}^r]}k[N])/D(X)]
\end{equation}
for an admissible inclusion $\mathbf{N}^r\hookrightarrow N$, where the map $\mathbf{N}^r\rightarrow k[x]$ sends each generator to $x$. To see this, suppose to the contrary that we have such a description of $\mc X$.
Consider the stack $\mc X_0:= \mc X\times _{\Sp (k[x]), x\mapsto 0}\Sp (k)$ and let $\fm \subset \mls O_{\mc X_0}$ be the ideal of $B\mu_4 \subset \mc X_0$. consider the graded sheaf of algebras $A = \oplus _{m\geq 0}\mathfrak{m}^m/\mathfrak{m}^{m+1}$ on $B\mu _4$, which we also abusively view as a graded algebra with $\mu _4$-action. On the one hand, from \eqref{eq:not2} we have $\mc X_0 = [\Sp (k[y,w]/(y^2-w^2, yw))/\mu _4]$ and $\fm = ( y, w )$, hence
\[
A = k[y, w]/(y^2-w^2, yw) = k \oplus ky \oplus ky^2 \oplus kw.
\]
On the other hand, from the description \eqref{eq:not1} we have
\[
A = k[N]/(\mathbf{N}^r )
\]
where the quotient is by the ideal in $k[N]$ generated by elements corresponding to the image of $\mathbf{N}^r \hookrightarrow N$. 

Looking at the stabilizer group scheme of the closed point, we must have $X = \mathbf{Z}/(4)$, and for  $i=1, \dots, 3$ we have a minimal lift $z_i\in N$ of $i\in \mathbf{Z}/(4)$. If these two descriptions of the algebra $A$ are truly isomorphic, then looking at eigenspaces for the $\mu _4$ actions we see that $z_1$ is identified with $y$ (up to unit) and $z_2$ is identified with $y^2$.  
This implies that $2z_1 = z_2$ in $N$.  Similarly we must have $2z_3 = z_2$.  This gives that $z_1-z_3\in N^\gp $ is torsion and therefore in $N$ since $N$ is saturated.  On the other hand, the element $y/w$ is not in the ring $k[y,w]/(y^2-w^2)$ giving a contradiction.
% This stack is not of the form $[\Sp (k[x]\otimes _{k[\mathbf{N}^r]}k[N])/D(X)]$ for an admissible inclusion $\mathbf{N}^r\hookrightarrow N$, where the map $\mathbf{N}^r\rightarrow k[x]$ sends each generator to $x$. To see this, suppose to the contrary that we have such a description of the stack.  Looking at the stabilizer group scheme of the closed point, we must have $X = \mathbf{Z}/(4)$, and for  $i=1, \dots, 3$ we have a minimal lift $z_i\in N$ of $i\in \mathbf{Z}/(4)$.  Consider the stack $\mc X_0:= \mc X\times _{\Sp (k[x]), x\mapsto 0}\Sp (k)$, which is isomorphic to $[\Sp (k[y,w]/(y^2-w^2, yw))/\mu _4]$.  Let $\mathfrak{m}\subset \mls O_{\mc X_0}$ be the ideal of $B\mu _4\subset \mc X_0$, and consider the graded sheaf of algebras $\oplus _{m\geq 0}\mathfrak{m}^m/\mathfrak{m}^{m+1}$ on $B\mu _4$, which we also abusively view as a graded algebra with $\mu _4$-action.  In fact, this algebra is simply $k[y, w]/(y^2-w^2, yw)= ky\oplus ky^2\oplus kw$. \Rachel{ this algebra is simply $k \oplus ky \oplus ky^2 \oplus kw$. On the other hand, if $\mc X$ is of the form $[\Sp (k[x]\otimes _{k[\mathbf{N}^r]}k[N])/D(X)]$ then } Looking at eigenspaces for the $\mu _4$ action we see that $z_1$ maps to $y$ (up to unit) and $z_2$ maps to $y^2$.  \Rachel{}
% This implies that $2z_1 = z_2$ in $N$.  Similarly we must have $2z_3 = z_2$.  This gives that $z_1-z_3\in N^\gp $ is torsion and therefore in $N$ since $N$ is saturated.  On the other hand, the element $y/w$ is not in the ring $k[y,w]/(y^2-w^2)$ giving a contradiction. 
\end{example}

\subsection{Finite fibers}

Let $k$ be an algebraically close field. Fix a tame abelian nodal orbicurve $\mls C$ with coarse space $C$. 

\begin{prop}\label{P:4.15} The set of generalized log twisted curves with $n$ marked points and associated stack $\mls C$  is finite.
\end{prop}
\begin{proof}
Suppose $(C, \{s_i\}_{i=1}^n, \ell: M_k \hookrightarrow M_k', \mls N)$ is a generalized log twisted curve with associated stack $\mls C$. Then $C$ is uniquely isomorphic to the coarse space of $\mls C$, and the simple inclusion $\mls \ell$ is also uniquely determined by $\mls C$. Let $x_1, \ldots, x_r$ denote the distinct points of $C^{sm}$ where $\mls C$ has nontrivial stabilizer. There are finitely many ways to distribute the $n$ markings $s_i$ among the $r$ points $x_j$, so to finish the proof, it is enough to fix one such distribution and show that there are finitely many possibilities for each stalk $\mls N_{x_j}$. Let $n_j$ denote the number of markings equal to $x_j$ (so $\sum n_j = n$). By \ref{lem:monoid-vs-group} it is enough to show that there are finitely many possibilities for the groupifications $\mls N_{x_j}^{gp}$. 

Let $A_j$ denote the character group of $\mls C$ at $x_j \in  C$ (this group is noncanonically isomorphic to $\mathrm{Pic}(BG_{x_j})$). It follows from \ref{C:4.13}
that the set of possible $\mls N_{x_j}^{gp}$'s is equal to a fiber of the map $\mathrm{Ext}^1(A_j, \mathbf{Z}^{n_j}) \xrightarrow{\Sigma} \mathrm{Ext}^1(A_j, \mathbf{Z})$, where $\mathbf{Z}^{n_j}\rightarrow \mathbf{Z}$ is summation. 
% it is enough to show that there are finitely many possibilities for each of the stalks $\mls N_{x_j}$. Let $n_j$ denote the number of $s_i$ equal to $x_j$ and let $X_j$ denote the character group of the stabilizer group scheme at $x_j$. If $\mls N$ determines a generalized log twisted curve whose associated stack is $\mls C$, then  $\mls N_{x_j}^{gp}$ is a  extension of $X_j$ by $\mathbf{Z}^{n_j}$ and $\mls N_{x_j}$ is determined by this extension.
Letting $K$ denote the kernel of this map, from short exact sequence
\[
0\rightarrow K\rightarrow \mathbf{Z}^{n_j}\rightarrow \mathbf{Z}\rightarrow 0
\]
we get an exact sequence
$$
0 = \text{Hom}(A_j, \mathbf{Z})\rightarrow \text{Ext}^1(A_j, K)\rightarrow \text{Ext}^1(A_j, \mathbf{Z}^{n_j})\xrightarrow{\Sigma}\text{Ext}^1(A_j,\mathbf{Z})\rightarrow \text{Ext}^2(A_j, K) = 0,
$$
where $\text{Ext}^2(A_j, K) = 0$ since $A_j$ has projective dimension $1$ (being a finite abelian group).  It follows that the required fiber is a torsor for the finite group $\text{\rm Ext}^1(A_j, K)$; hence, the set of possible $\mls N_{x_j}^{gp}$'s is finite.
\end{proof}

\bibliographystyle{amsplain}
\bibliography{bibliography}{}

\end{document}